\documentclass[onecolumn,final]{elsart3p}

\usepackage{comment}
\usepackage{amsmath}       
\usepackage{amssymb}       
\usepackage{chapterbib}    
\usepackage{color}
\usepackage{comment}
\usepackage{ifthen}
\usepackage{bm,bbm}
\usepackage{overpic}
\usepackage{times}
\usepackage{epsfig}
\usepackage{graphicx,graphics,rotating}
\usepackage{subfigure}
\usepackage{multicol}
\usepackage{multirow}

\usepackage{ifthen}
\newif\ifJournal
\Journalfalse

\newtheorem{theorem}{Theorem}[section]
\newtheorem{lemma}[theorem]{Lemma}
\newtheorem{corollary}[theorem]{Corollary}

\newtheorem{remark}[theorem]{Remark}

\newtheorem{proof}[theorem]{Proof}

\definecolor{MyDarkGreen}{rgb}{0,0.45,0}

\newcommand{\BLUE}[1]{{#1}}
\newcommand{\RED} [1]{{#1}}

\def\trait #1 #2 #3 {\vrule width #1pt height #2pt depth #3pt}
\def\fin{\hfill
        \trait .3 5 0
        \trait 5 .3 0
        \kern-5pt
        \trait 5 5 -4.7
        \trait 0.3 5 0
\medskip}
\newcommand{\ENDPROOF}{\fin}

\newcommand{\VEM}{\text{VEM}}

\newcommand{\INTP}{\footnotesize{I}}

\newcommand{\REAL}{\mathbbm{R}}
\newcommand{\INTG}{\mathbbm{N}}

\newcommand{\ASSUM}[1]{\textbf{(#1)}}

\newcommand{\av}{\mathbf{a}}

\renewcommand{\cv}{\mathbf{c}}

\newcommand{\ev}{\mathbf{e}}
\newcommand{\fv}{\mathbf{f}}
\newcommand{\gv}{\mathbf{g}}

\newcommand{\kv}{\mathbf{k}}

\newcommand{\nv}{\mathbf{n}}

\newcommand{\qv}{\mathbf{q}}

\newcommand{\uv}{\mathbf{u}}
\newcommand{\vv}{\mathbf{v}}
\newcommand{\wv}{\mathbf{w}}
\newcommand{\xv}{\mathbf{x}}
\newcommand{\yv}{\mathbf{y}}

\newcommand{\Iv}{\mathbf{I}}

\newcommand{\Kv}{\mathbf{K}}

\newcommand{\Mv}{\mathbf{M}}

\newcommand{\Qv}{\mathbf{Q}}

\newcommand{\Uv}{\mathbf{U}}
\newcommand{\Vv}{\mathbf{V}}

\newcommand{\as}{a}

\newcommand{\cs}{c}

\newcommand{\es}{e}

\newcommand{\ks}{k}

\newcommand{\ms}{m}

\newcommand{\ps}{p}
\newcommand{\qs}{q}

\newcommand{\ts}{t}

\newcommand{\vs}{v}
\newcommand{\ws}{w}
\newcommand{\xs}{x}
\newcommand{\ys}{y}

\newcommand{\Cs}{C}

\newcommand{\Fs}{F}
\newcommand{\Gs}{G}

\newcommand{\Ls}{L}

\newcommand{\Ps}{P}

\newcommand{\Ts}{T}

\newcommand{\matD}{\mathsf{D}}

\newcommand{\matI}{\mathsf{I}}

\newcommand{\matK}{\mathsf{K}}

\newcommand{\matM}{\mathsf{M}}

\newcommand{\calA}{\mathcal{A}}

\newcommand{\calD}{\mathcal{D}}

\newcommand{\calK}{\mathcal{K}}

\newcommand{\calM}{\mathcal{M}}

\newcommand{\calO}{\mathcal{O}}

\newcommand{\HONE}  {H^1}
\newcommand{\HONEzr}{H^1_0}

\newcommand{\HTWO}  {H^2}
\newcommand{\HONEgm}{H^1_{\Gamma_D}}

\newcommand{\LTWO}  {L^2}
\newcommand{\LINF}  {L^{\infty}}

\newcommand{\HS}[1] {H^{#1}}

\newcommand{\CS}[1] {C^{#1}}
\newcommand{\VS}[1] {V^{#1}}

\newcommand{\PS}[1] {\mathbbm{P}_{#1}}

\renewcommand{\P} {\textsf{P}}            
\newcommand  {\E} {e}

\newcommand{\hh}{h}
\newcommand{\Th}{\Omega_{\hh}}

\newcommand{\xvP}{\xv_{\P}}        
\newcommand{\xvE}{\xv_{\E}}        

\newcommand{\dims}{2}              
\newcommand{\DIM} {d}              

\newcommand{\hP}{\hh_{\P}}

\newcommand{\hE}{\hh_{\E}}

\newcommand{\mP}{\ABS{\P}}

\newcommand{\mE}{\ABS{\E}}

\newcommand{\NMB}{N}

\newcommand{\dV}{\,dV}
\newcommand{\dS}{\,ds}

\newcommand{\dt}{\,dt}

\newcommand{\norE} {\mathbf{n}_{\E}}

\newcommand{\norP} {\mathbf{n}_{\P}}

\newcommand{\vsh}{\vs_{\hh}}

\newcommand{\Fsh}{\Fs_{\hh}}

\newcommand{\uvh}{\uv_{\hh}}

\newcommand{\vvh}{\vv_{\hh}}

\newcommand{\evh}{\ev_{\hh}}

\newcommand{\uvI}{\uv_{\INTP}}

\newcommand{\wvh}{\wv_{\hh}}

\newcommand{\fvh}{\fv_{\hh}}

\newcommand{\asP}{\as^{\P}}

\newcommand{\msP}{\ms^{\P}}

\newcommand{\ash}{\as_{\hh}}

\newcommand{\msh}{\ms_{\hh}}

\newcommand{\ashP}{\as^P_{\hh}}

\newcommand{\mshP}{\ms^P_{\hh}}

\newcommand{\scal} [1]{\left(#1\right)}

\newcommand{\SP} {S^{\P}}

\newcommand{\nlen}{\hspace{-0.2mm}}
\newcommand{\hskp}{\hspace{0.2mm}}

\newcommand{\snorm}  [2]{|#1|_{#2}}

\newcommand{\SNORM}  [2]{\left|#1\right|_{#2}}

\newcommand{\norm}   [2]{|\nlen|#1|\nlen|_{#2}}
\newcommand{\Norm}   [2]{\Big|\nlen\nlen\Big|\hskp #1\hskp\Big|\nlen\nlen\Big|_{#2}}
\newcommand{\NORM}   [2]{\left|\nlen\left|#1\right|\nlen\right|_{#2}}

\newcommand{\TNORM}  [2]{\left|\nlen\left|\nlen\left|#1\right|\nlen\right|\nlen\right|_{{}_{#2}}}

\newcommand{\abs}    [1]{|#1|}

\newcommand{\ABS}    [1]{\left|#1\right|}

\newcommand{\Vhk} {\VS{\hh}_{k}}

\newcommand{\Vvhk} {\textbf{V}^{\hh}_{k}}

\newcommand{\Pin}[1]{\Pi^{\nabla}_{#1}}
\newcommand{\Piz}[1]{\Pi^{0}_{#1}}

\newcommand{\uvp}{\uv_{\pi}}

\newcommand{\restrict}[2]{{#1}_{|{#2}}}

\newcommand{\EOD}{\end{document}}

\newcommand{\Ndofs}{\NMB_{\textrm{dofs}}}

\newcommand{\Ph}{\mathcal{P}_{\hh}}

\newcommand{\TERM}[2]{\textsf{#1}_{#2}}

\newcommand{\bvarphi}{{\bm\phi}}

\newcommand{\beps}   {{\bm\varepsilon}}
\newcommand{\bsig}   {{\bm\sigma}}
\newcommand{\btau}   {{\bm\tau}}

\newcommand{\bPi} [1]{{\bm\Pi}^{0}_{#1}}
\newcommand{\bPix}[1]{{\bm\Pi}^{0,x}_{#1}}
\newcommand{\bPiy}[1]{{\bm\Pi}^{0,y}_{#1}}
\newcommand{\bPin}[1]{{\bm\Pi}^{\nabla}_{#1}}

\newcommand{\DIRI}{D}
\newcommand{\NEUM}{N}
\newcommand{\GamD}{\Gamma_{\DIRI}}
\newcommand{\GamN}{\Gamma_{\NEUM}}

\newcommand{\TRACE}{\textrm{tr}}

\newcommand{\gvN}{\gv_{\NEUM}}

\newcommand{\zero}{\mathbf{0}}

\newcommand{\symmSpace}{\REAL^{2\times2}_{\textrm{sym}}} 
\newcommand{\px}{\partial_x}
\newcommand{\py}{\partial_y}
\newcommand{\wsx}{\ws_x}
\newcommand{\wsy}{\ws_y}

\newcommand{\pxh}[1]{\widehat{\partial_x#1}}
\newcommand{\pyh}[1]{\widehat{\partial_y#1}}

\newcommand{\UP}{\textrm{up}}
\newcommand{\DW}{\textrm{down}}
\newcommand{\bphiUP}{\bvarphi^{\UP}}
\newcommand{\bphiDW}{\bvarphi^{\DW}}
\newcommand{\dofs}{\textrm{dofs}}
\newcommand{\poly}{k            }

\newcommand{\rhov}{\bm\rho}
\newcommand{\etav}{\bm\eta}

\newcommand{\PGRAPH}[1]{\medskip\noindent\textit{#1.}}

\newcommand{\psiv}{{\bm\psi}}
\newcommand{\psivI}{\psiv_{\INTP}}

\newcommand{\MM}{\rm M}
\newcommand{\KK}{\rm K}

\newcommand{\xx}{\mathbf{x}}

\newcommand{\UU}{\mathbf{U}}

\newcommand{\kk}{\mathbf{k}}

\begin{document}

\date{}

\begin{frontmatter}

  \title{The arbitrary-order virtual element method for linear elastodynamics models. Convergence, stability and dispersion-dissipation analysis}
  \author[MOXA]{P.~F.~Antonietti}
  \author[T5]  {, G.~Manzini}
  \author[MOXM]{, I.~Mazzieri}
  \author[T3]  {, H.~M.~Mourad}
  \author[MOXV]{, and M.~Verani}

  \address[MOXA]{
    MOX, Dipartimento di Matematica,
    Politecnico di Milano, Italy;
    \emph{e-mail: paola.antonietti@polimi.it}
  }
  \address[T5]{
    Group T-5,
    Theoretical Division,
    Los Alamos National Laboratory,
    Los Alamos, NM,
    USA;
    \emph{e-mail: gmanzini@lanl.gov}
  }
  \address[MOXM]{
    MOX, Dipartimento di Matematica,
    Politecnico di Milano, Italy;
    \emph{e-mail: ilario.mazzieri@polimi.it}
  }
  \address[T3]{
    Group T-3,
    Theoretical Division,
    Los Alamos National Laboratory,
    Los Alamos, NM,
    USA;
    \emph{e-mail: hmourad@lanl.gov}
  }
  \address[MOXV]{
    MOX, Dipartimento di Matematica,
    Politecnico di Milano, Italy;
    \emph{e-mail: marco.verani@polimi.it}
  }

  \begin{abstract}
    We design the conforming virtual element method for the numerical
    approximation of the two dimensional elastodynamics problem.
    We prove stability and convergence of the semi-discrete
    approximation and derive optimal error estimates under
    $\hh$-refinement in both the energy and the $L^2$ norms, and
    optimal error estimates under $\ps$-refinement in the energy norm.
    The performance of the proposed virtual element method is assessed
    on a set of different computational meshes, including non-convex
    cells up to order four in the $h$-refinement setting.
    Exponential convergence is also experimentally observed under
    $p$-refinement.
    Finally, we present a dispersion-dissipation analysis for both the
    semi-discrete and fully-discrete schemes, showing that polygonal
    meshes behave as classical simplicial/quadrilateral grids in terms
    of dispersion-dissipation properties.
  \end{abstract}


\end{frontmatter}


\maketitle
\raggedbottom

\section{Introduction}
\label{sec1:introduction}

In recent years, numerical modeling of elastic waves propagation
problems through the elastodynamics equation has undergone a
constantly increasing interest in the mathematical and geophysics
engineering community.
One of the most employed numerical technique is the Spectral Element
Method that has successfully been applied to the elastodynamics
equation.\cite{komatitsch1999introduction,Faccioli1996}
Elastic waves propagation problems have been treated numerically by
applying the Discontinuous Galerkin (DG) and the Discontinuos Galerkin
Spectral Element
method\cite{Riviere-Wheeler:2003,Antonietti-AyusodeDios-Mazzieri-Quarteroni:2016,Antonietti-Mazzieri-Quarteroni-Rapetti:2012}
to the \emph{displacement formulation} and the \emph{stress-velocity
  formulation}.\cite{Dumbser-Kaser:2006}
High-order DG methods for elastic and elasto-acoustic wave propagation
problems have been extended to arbitrarily-shaped polygonal/polyhedral
grids\cite{Antonietti-Mazzieri:2018,Antonietti-Bonaldi-Mazzieri:2018} to
further enhance the geometrical flexibility of the DG approach while
guaranteeing low dissipation and dispersion errors. 
Recently, the lowest-order Virtual Element Method (\VEM{}) has been
applied for the solution of the elastodynamics equation on nonconvex
polygonal meshes.\cite{Park-Chi-Paulino:2019a,Park-Chi-Paulino:2019b}

Studying the elastodynamic behavior of structures with complicated
geometrical features is often of interest in many practical
situations, e.g., in aerospace and power-generation applications.
Traditionally, triangular -- and in 3D, tetrahedral -- elements have
been the only available option when it came to the spatial
discretization of such structures.
This is problematic because low-order triangles/tetrahedra are known
to over-estimate the structure's stiffness (and hence its natural
frequencies and the wave propagation speed within it), especially with
nearly-incompressible materials.
In addition, very small triangular elements are often needed to
resolve intricate features of the geometry, which then requires a very
small time step to be used in view of the CFL stability condition.
Polygonal elements can alleviate such difficulties, since they often
obviate the need to refine the spatial discretization even in the
presence of complicated/intricate geometrical features.
This is especially true if the underlying numerical method allows
\emph{non-convex} polygonal elements to be used efficiently, as is the
case with the virtual element method.

The VEM, that can be seen as a variational reformulation of the
\emph{nodal} mimetic finite difference (MFD)
method,\cite{BeiraodaVeiga-Lipnikov-Manzini:2014} was originally
proposed as a conforming method for solving elliptic boundary-value
problems on polytopal
meshes\cite{BeiraodaVeiga-Brezzi-Cangiani-Manzini-Marini-Russo:2013}
and later extended to nonconforming
formulations,\cite{AyusodeDios-Lipnikov-Manzini:2016,Cangiani-Manzini-Sutton:2017,Cangiani-Gyrya-Manzini:2016,Antonietti-Manzini-Verani:2018,Mascotto-Perugia-Pichler:2018}
and to higher-order differential
equations,\cite{Antonietti-Manzini-Verani:2018,Antonietti-Manzini-Verani:2019}
and to "p/hp" formulations, also using adaptive mesh
refinements.~\cite{BeiraodaVeiga-Manzini-Mascotto:2019}
The VEM provide a very effective and flexible setting for designing
arbitrary-order accurate numerical methods suited to polygonal and
polyhedral meshes with respect to straightforward generalization of
the finite element method as in the polygonal
FEM.\cite{Manzini-Russo-Sukumar:2014}
In the last few years, great amount of work has also been devoted to
the numerical modeling of linear and nonlinear elasticity problems and
materials.
For example, it is worth mentioning
the VEM for plate bending
problems\cite{Brezzi-Marini:2013,BeiraodaVeiga-Brezzi-Marini:2013}
and VEM for the stress-displacement formulation of plane elasticity
problems,\cite{Artioli-deMiranda-Lovadina-Patruno:2017}
plane elasticity problems based on the {H}ellinger-{R}eissner
principle,\cite{Artioli-deMiranda-Lovadina-Patruno:2018}
mixed virtual element method for a pseudostress-based formulation of
linear elasticity,\cite{Caceres-Gatica-Sequeira:2018}
nonconforming virtual element method for elasticity
problems,\cite{Zhang-Zhao-Yang-Chen:2018}
linear\cite{Gain-Talischi-Paulino:2014} and nonlinear
elasticity,\cite{DeBellis-Wriggers-Hudobivnik:2019}
contact problems\cite{Wriggers-Rust-Reddy:2016} and frictional
contact problems including large
deformations,\cite{Wriggers-Rust:2019}
elastic and inelastic problems on polytope
meshes,\cite{BeiraodaVeiga-Lovadina-Mora:2015}
compressible and incompressible finite
deformations,\cite{Wriggers-Reddy-Rust-Hudobivnik:2017}
finite elasto-plastic
deformations,\cite{Chi-BeiraodaVeiga-Paulino:2017,Wriggers-Hudobivnik:2017,Hudobivnik-Aldakheel-Wriggers:2018}
virtual element method for coupled thermo-elasticity in
Abaqus,\cite{Dhanush-Natarajan:2018}
a priori and a posteriori error estimates for a virtual element
spectral analysis for the elasticity
equations,\cite{Mora-Rivera:2018}
virtual element method for transversely isotropic
elasticity.\cite{Reddy-vanHuyssteen:2019}
In this paper, we are aimed at investigating both theoretically and
numerically the performance of the high-order Virtual Element Method
for the numerical modeling of wave propagation phenomena in elastic
media.
In particular, we prove the stability and the convergence of the
semi-discrete approximation in the energy norm and derive error
estimates.
The performance of the method, in terms of convergence, stability and
dispersion-dissipation analysis, is assessed on a set of different
computational meshes.
Exponential convergence is also experimentally seen in the
$p$-refinement setting.
\BLUE{With the aim of focusing the reader on the novelties of
  the present work, we note that at the time of writing, only one work
  has been published on a somewhat similar, though simpler problem,
  namely the numerical approximation of the wave equation in the
  framework of the conforming virtual element method using the
  elliptic projection operator to define the stiffness
  matrix.\cite{Vacca:2016}}
However, our work is substantially different since we consider a more
complex mathematical model, a diverse range of applications, a
different discretization of the stiffness matrix, which is here based
on orthogonal projections of the simmetric gradients and not on the
elliptic projector, and, consequently, a substantially different
strategy for the analysis both in $\LTWO$ and energy
norm.
\BLUE{Moreover, the dispersion-dissipation analysis presented in the
  last part of the our work is specifically carried out for the elastodynamics problem
  here considered.}

\medskip
The outline of the paper is as follows.
We conclude this introductory section with
a subsection introducing the
notation used in this paper.
Then, in Section~\ref{sec2:model}, we introduce the model problem and
its virtual element approximation.
In Section~\ref{sec3:vem}, we present the design of the VEM.
In Section~\ref{sec4:convergence}, we discuss the convergence of the VEM and
derive suitable bounds for the approximation error.
In Section~\ref{sec5:numerical}, we investigate the performance of the
method on a set of suitable numerical experiments.
In particular, in Section~\ref{subsec51:manuf_soln} we show the
optimal convergence properties of the VEM by using a manufactured
solution on three mesh families with different difficulties, and in
Section~\ref{subsec52:diss-disp}, we carry out an experimental
investigation \emph{\`a la Von Neumann} on the stability and
dispersion-dissipation properties of the method.
In Section~\ref{sec6:conclusion}, we offer our final remarks and
conclusions.

\subsection{Notation of functional spaces and technicalities}
We use the standard definitions and notation of Sobolev spaces, norms
and seminorms.~\cite{Adams-Fournier:2003}
Let $k$ be a nonnegative integer number.
The Sobolev space $\HS{k}(\omega)$ consists of all square integrable
functions with all square integrable weak derivatives up to order $k$
that are defined on the open bounded connected subset $\omega$ of
$\REAL^{2}$.
As usual, if $k=0$, we prefer the notation $\LTWO(\omega)$.
Norm and seminorm in $\HS{k}(\omega)$ are denoted by
$\norm{\cdot}{k,\omega}$ and $\snorm{\cdot}{k,\omega}$, respectively,
and $(\cdot,\cdot)_{\omega}$ denote the $\LTWO$-inner product.
We omit the subscript $\omega$ when $\omega$ is the whole
computational domain $\Omega$. 

Given the mesh partitioning $\Th=\{\P\}$ of the domain $\Omega$ into
elements $\P$, we define the broken (scalar) Sobolev space for any
integer $k>0$
\begin{align*}
  \HS{k}(\Th) 
  = \prod_{\P\in\Th}\HS{k}(\P) 
  = \big\{\,\vs\in\LTWO(\Omega)\,:\,\restrict{\vs}{\P}\in\HS{k}(\P)\,\big\}, 
\end{align*}
which we endow with the broken $\HS{k}$-norm
\begin{align}
  \label{eq:Hs:norm-broken}
  \norm{\vs}{k,\hh}^2 = \sum_{\P\in\Th}\norm{\vs}{k,\P}^{2}
  \qquad\forall\,\vs\in\HS{k}(\Th),
\end{align}
and, for $k=1$, with the broken $\HONE$-seminorm
\begin{align}
  \label{eq:norm-broken}
  \snorm{\vs}{1,h}^2 = \sum_{\P\in\Th}\norm{\nabla\vs}{0,\P}^{2} 
  \qquad\forall\,\vs\in\HONE(\Th).
\end{align}

We denote the linear space of polynomials of degree up to $\ell$
defined on $\omega$ by $\PS{\ell}(\omega)$, with the useful
conventional notation that $\PS{-1}(\omega)=\{0\}$.
We denote the space of two-dimensional vector polynomials of degree up
to $\ell$ on $\omega$ by $\big[\PS{\ell}(\omega)\big]^2$; the space of
symmetric $2\times2$-sized tensor polynomials of degree up to $\ell$
on $\omega$ by $\PS{\ell,\textrm{sym}}^{2\times2}(\omega)$.
Space $\PS{\ell}(\omega)$ is the span of the finite set of
\emph{scaled monomials of degree up to $\ell$}, that are given by
\begin{align*}
  \calM_{\ell}(\omega) =
  \bigg\{\,
    \left( \frac{\xv-\xv_{\omega}}{\hh_{\omega}} \right)^{\alpha}
    \textrm{~with~}\abs{\alpha}\leq\ell
    \,\bigg\},
\end{align*}
where 
\begin{itemize}
\item $\xv_{\omega}$ denotes the center of gravity of $\omega$ and
  $\hh_{\omega}$ its characteristic length, as, for instance, the edge
  length or the cell diameter for $\DIM=1,2$;
\item $\alpha=(\alpha_1,\alpha_2)$ is the
  two-dimensional multi-index of nonnegative integers $\alpha_i$
  with degree $\abs{\alpha}=\alpha_1+\alpha_{2}\leq\ell$ and
  such that
  $\xv^{\alpha}=\xs_1^{\alpha_1}\xs_{2}^{\alpha_{2}}$ for
  any $\xv\in\REAL^{2}$.
\end{itemize}
We will also use the set of \emph{scaled monomials of degree exactly
  equal to $\ell$}, denoted by $\calM_{\ell}^{*}(\omega)$ and obtained
by setting $\abs{\alpha}=\ell$ in the definition above.

\medskip
Finally, we use the letter $C$ in the error estimates to denote a
strictly positive constant whose value can change at any instance and
that is independent of the discretization parameters such as the mesh
size $\hh$.
Note that $C$ may depend on the constants of the model equations or
the variational problem, like the coercivity and continuity constants,
or even constants that are uniformly defined for the family of meshes
of the approximation while $\hh\to0$, such as the mesh regularity
constant, the stability constants of the discrete bilinear forms, etc.
Whenever it is convenient, we will simplify the notation by using
expressions like $x\lesssim y$ and $x\gtrsim y$ to mean that $x\leq
Cy$ and $x\geq Cy$, respectively, $C$ being the generic constant in
the sense defined above.
\section{Model problem and virtual element formulation}
\label{sec2:model}

We consider an elastic body occupying the open, bounded polygonal
domain denoted by $\Omega\in\REAL^2$ with boundary denoted by
$\Gamma=\partial\Omega$.
We assume that boundary $\Gamma$ can be split into the two disjoint
subsets $\GamD$ and $\GamN$, so that
$\overline{\Gamma}=\overline{\Gamma}_{\DIRI}\cup\overline{\Gamma}_{\NEUM}$
and $\GamD\cap\GamN=\emptyset$.
For the well-posedness of the mathematical model, we further require
that the one-dimensional Lebesgue measure (length) of $\GamD$ is
nonzero, i.e., $\ABS{\GamD}>0$.
Let $T>0$ denote the final time.
We consider
the external load $\fv\in\LTWO((0,T);[\LTWO(\Omega)]^2)$,
the boundary function $\gvN\in\CS{1}((0,T);[H^{\frac{1}{2}}_{0,\GamN}]^2)$,
and the initial functions $\uv_{0}\in[H^{1}_{0,\GamD}(\Omega)]^2$,
$\uv_{1}\in[L^2(\Omega)]^2$.
For such time-dependent vector fields, we may indicate the dependence
on time explicitly, e.g., $\fv(t):=\fv(\cdot,t)\in[\LTWO(\Omega)]^2$,
or drop it out to ease the notation when it is obvious from the
context.

The equations governing the two-dimensional initial/boundary-value
problem of linear elastodynamics for the displacement vector
$\uv:\Omega\times[0,T]\to\REAL^2$ are:
\begin{align}
  \rho\ddot{\uv} - \nabla\cdot\bsig(\uv)
  &= \fv\phantom{\zero\gvN\uv_{0}\uv_{1}}\textrm{in}~\Omega\times(0,T],\label{eq:pblm:strong:A}\\[0.5em]
  \uv
  &= \zero\phantom{\fv\gvN\uv_{0}\uv_{1}}\textrm{on}~\GamD \times(0,T],\label{eq:pblm:strong:B}\\[0.5em]
  \bsig(\uv)\nv
  &= \gvN\phantom{\fv\zero\uv_{0}\uv_{1}}\textrm{on}~\GamN \times(0,T],\label{eq:pblm:strong:C}\\[0.5em]
  \uv
  &= \uv_{0}\phantom{\fv\zero\gvN\uv_{1}}\textrm{in}~\Omega\times\{0\},\label{eq:pblm:strong:D}\\[0.5em]
  \dot{\uv} &=
  \uv_{1}\phantom{\fv\zero\gvN\uv_{0}}\textrm{in}~\Omega\times\{0\}.\label{eq:pblm:strong:E}
\end{align}
Here, $\rho$ is the mass density, which we suppose to be a strictly
positive and uniformly bounded function and $\bsig(\uv)$ is the stress
tensor.
In~\eqref{eq:pblm:strong:B} we assume homogeneous Dirichlet boundary
conditions on $\GamD$.
This assumption is made only to ease the exposition and the analysis,
as our numerical method is easily extendable to treat the case of
nonhomogeneous Dirichlet boundary conditions.

\medskip
\noindent
We denote the space of symmetric $2\times2$-sized real-valued tensors
by $\symmSpace$ and assume that the stress tensor
$\bsig:\Omega\times[0,T]\to\symmSpace$ is expressed, according to
Hooke's law, by $\bsig(\uv) = \calD\beps(\uv)$,
where, $\beps(\uv)$ denotes the symmetric gradient of $\uv$ , i.e.,
$\beps(\uv)=\big(\nabla\uv+(\nabla\uv)^{T}\big)\slash{2}$, and
$\calD=\calD(\xv)\,\,:\symmSpace\longrightarrow\symmSpace$
is the \emph{stiffness} tensor
\begin{align}
  \calD\btau = 2\mu\btau + \lambda\TRACE(\btau)\matI
  \label{eq:bsig:def}
\end{align}
for all $\btau\in\REAL^{2\times2}_{\textrm{sym}}$.
In this definition, $\matI$ and $\TRACE(\cdot)$ are the identity
matrix and the trace operator; $\lambda$ and $\mu$ are the first and
second Lam\'e coefficients, which we assume to be in $\LINF(\Omega)$
and nonnegative.
The compressional (P) and shear (S) wave velocities of the medium are
respectively obtained through the relations $c_P = \sqrt{(\lambda +
  2\mu)/\rho}$ and $c_S = \sqrt{\mu/\rho}$.

\medskip
\noindent
Let $\Vv=\big[\HONEgm(\Omega)\big]^2$ be the space of $H^1$ vector
valued functions with null trace on $\Gamma_D$.
We consider the two bilinear forms
$\ms(\cdot,\cdot),\,\as(\cdot,\cdot)\,:\,\Vv\times\Vv\to\REAL$ defined
as
\begin{align}
  \ms(\wv,\vv) &= \int_{\Omega}\rho\wv\cdot\vv\dV       \quad\forall\wv,\,\vv\in\Vv,\label{eq:ms:def}\\[0.5em]
  \as(\wv,\vv) &= \int_{\Omega}\bsig(\wv):\beps(\vv)\dV \quad\forall\wv,\,\vv\in\Vv,\label{eq:as:def}\\[0.5em]
  \intertext{and the linear functional $\Fs(\cdot)\,:\,\Vv\to\REAL$ defined
    as}
  \Fs(\vv) &= \int_{\Omega}\fv\cdot\vv\dV + \int_{\GamN}\gvN\cdot\vv\quad\forall\vv\in\Vv.\label{eq:Fv:def}
\end{align}

The variational formulation of the linear elastodynamics equations
reads as: \emph{For all $t\in(0,T]$ find $\uv(t)\in\Vv$ such that for
  $t=0$ it holds that $\uv(0)=\uv_0$ and $\dot{\uv}(0)=\uv_1$ and}
\begin{align}
  \ms(\ddot{\uv},\vv) + \as(\uv,\vv) &= \Fs(\vv)\qquad\forall\vv\in\Vv.
  \label{eq:weak:A}
\end{align}
As shown, for example, by Raviart and Thomas (see Theorem~8-3.1\cite{Raviart-Thomas:1983})
the variational problem \eqref{eq:weak:A} is well posed and its unique
solution satisfies
$\uv\in\CS{0}((0,T];\Vv)\cap\CS{1}((0,T];[\LTWO(\Omega)]^2)$.

\medskip
\noindent
The virtual element approximation of problem~\eqref{eq:weak:A} relies
on the virtual element space $\Vvhk$, which is a subspace of
$\big[\HONEgm(\Omega)\big]^2$.
Space $\Vvhk$ is built upon the scalar conforming
space.\cite{BeiraodaVeiga-Brezzi-Cangiani-Manzini-Marini-Russo:2013}
For the sake of completeness, we briefly review the construction of
the scalar space in the next section.
Then, we consider
the virtual element bilinear forms $\msh(\cdot,\cdot)$ and
$\ash(\cdot,\cdot)$, which approximate the bilinear forms
$\ms(\cdot,\cdot)$ and $\as(\cdot,\cdot)$,
and the virtual element functional $\Fsh(\cdot)$, which approximates
the linear functional $\Fs(\cdot)$.
The definition of $\msh(\cdot,\cdot)$, $\ash(\cdot,\cdot)$, and
$\Fsh(\cdot)$ and the discussion of their properties is addressed in
the next section.

\medskip
\noindent
The semi-discrete virtual element approximation of~\eqref{eq:weak:A} read
as: \emph{For all $t\in(0,T]$ find $\uvh(t)\in\Vvhk$ such that for
  $t=0$ it holds that $\uvh(0)=(\uv_{0})_{\INTP}$ and
  $\dot{\uv}_{\hh}(0)= (\uv_{1})_{\INTP}$ and}
\begin{align}
  \msh(\ddot{\uv}_h,\vvh) 
  + \ash(\uvh,\vvh) = \Fsh(\vvh)
  \quad\forall\vvh\in\Vvhk.
  \label{eq:VEM:semi-discrete}
\end{align}
Here, 
$\uvh(t)$ is the virtual element approximation of $\uv$ and $\vvh$ the
generic test function in $\Vvhk$, while $(\uv_{0})_{\INTP}$ and
$(\uv_{1})_{\INTP}$ are the virtual element interpolants of the
initial solution functions $\uv(0)$ and $\dot{\uv}(0)$, respectively.

Time integration is performed by applying the leap-frog time marching
scheme~\cite{Quarteroni-Sacco-Saleri:2007} to the second derivative in
time $\ddot{\uv}_{\hh}$.
To this end, we subdivide the interval $(0,T]$ into $N_T$ subintervals
of amplitude $\Delta t=T\slash{N_T}$ and at every time level
$t^n=n\Delta t$ we consider the variational problem for $n\geq1$:
\begin{align}
  \msh(\uvh^{n+1},\vvh) 
  - 2\msh(\uvh^{n},\vvh) 
  +  \msh(\uvh^{n-1},\vvh) 
  + \Delta t^2\ash(\uvh^{n},\vvh)  
  = \Delta t^2\Fsh^{n}(\vvh)
  \quad\forall\vvh\in\Vhk,
  \label{eq:VEM:fully-discrete}
\end{align}
and initial step
\begin{align*}
  \msh(\uvh^{1},\vvh) 
  - \msh(\uv_{0},\vvh) 
  - \Delta t  \msh(\uv_{1},\vvh) 
  + \frac{\Delta t^2}{2}\ash(\uv_{0},\vvh)  
  = \frac{\Delta t^2}{2}\Fsh^{0}(\vvh)
  \quad\forall\vvh\in\Vhk.
\end{align*}
The leap-frog scheme is second-order accurate, explicit and
conditionally stable.~\cite{Quarteroni-Sacco-Saleri:2007}
It is straightforward to show that these properties are inherited by
the fully-discrete scheme~\eqref{eq:VEM:fully-discrete}.
\section{Virtual Element Approximation}
\label{sec3:vem}

The VEM proposed in this paper is a vector extension of the VEM
previously developed for scalar elliptic and parabolic
problems,~\cite{BeiraodaVeiga-Brezzi-Cangiani-Manzini-Marini-Russo:2013,Ahmad-Alsaedi-Brezzi-Marini-Russo:2013,AyusodeDios-Lipnikov-Manzini:2016,Cangiani-Manzini-Sutton:2017}
that we shortly review in this section.
First, we introduce the family of mesh decompositions of the
computational domain and the mesh regularity assumptions needed to
prove the stability and convergence of the method.
Then, we formulate the conforming virtual element spaces of various
degree $k$ and present the degrees of freedom that are unisolvent in
such spaces.
Finally, we present the definition of the virtual element bilinear
forms and discuss their properties.

\subsection{Mesh definition and regularity assumptions}
\label{subsec:vem:meshes}
Let $\mathcal{T}=\{\Th\}_{\hh}$ be a family of decompositions of
$\Omega$ into nonoverlapping polygonal elements $\P$ with
nonintersecting boundary $\partial\P$, center of gravity $\xvP$,
two-dimensional measure (area) $\mP$, and diameter
$\hP=\sup_{\xv,\yv\in\P}\abs{\xv-\yv}$.
The subindex $\hh$ that labels each mesh $\Th$ is the maximum of the
diameters $\hP$ of the elements of that mesh.
The boundary of $\P$ is formed by straight edges $\E$ and the midpoint
and length of each edge are denoted by $\xvE$ and $\hE$, respectively.

We denote the unit normal vector to the elemental boundary
$\partial\P$ by $\norP$, and the unit normal vector to edge $\E$ by
$\norE$.
Vector $\norP$ points out of $\P$ and the orientation of $\norE$
is fixed \emph{once and for all} in every mesh

\medskip 
Now, we state the mesh regularity assumptions that are required for
the convergence analysis.

\medskip
\noindent
\ASSUM{A0}~\textbf{Mesh regularity assumptions}.

\begin{itemize}
  \medskip
\item 
  There exists a positive constant $\varrho$ independent of $\hh$
  (and, hence, of $\Th$) such that for every polygonal element
  $\P\in\Th$ it holds that
  \begin{description}
  \item[$(i)$] $\P$ is star-shaped with respect to a disk with radius
    $\ge\varrho\hP$;
  \item[$(ii)$] for every edge $\E\in\partial\P$ it holds that
    $\hE\geq\varrho\hP$.
  \end{description}
\end{itemize}

\medskip
\begin{remark}
  Star-shapedness property $(i)$ implies that the elements are
  \emph{simply connected} subsets of $\REAL^{2}$.
  Scaling property $(ii)$ implies that the number of edges in each
  elemental boundary is uniformly bounded over the whole mesh family
  $\mathcal{T}$.
\end{remark}

\medskip
To conclude this section, we note that the above mesh assumptions are
very general and, as observed from the first publications on the
VEM,\cite{BeiraodaVeiga-Brezzi-Cangiani-Manzini-Marini-Russo:2013}
allow us to formulate the VEM on grids of polygonal elements having
very general geometric shapes, e.g., nonconvex elements or elements
with hanging nodes.
However, virtual element formulations under weaker mesh assumptions
have been the object of recent
investigations.\cite{Beirao-Lovadina-Russo:2016,Brenner-Sung:2018,Cao-Chen:2018}
%

\subsection{Virtual element space, degrees of freedom and projection
  operators}
The global virtual element space is defined as
\begin{align}
  \Vvhk := \Big\{\vv\in\Vv\,:\,\vv_{|_\P}\in\Vvhk(\P)\,\textrm{for~every~}\P\in\Th\Big\}.
\end{align}
The construction of the local virtual element space is carried out
along these three steps: $(i)$ we select the set of degrees of freedom
that uniquely characterizes the functions of the local space; $(ii)$
we introduce the elliptic projector onto the subspace of polynomials;
$(iii)$ we define the virtual element functions as the solution of a
differential problem, also using the elliptic projector.

\medskip
\noindent
\begin{figure}[!t]
  \centering
  \begin{tabular}{ccc}
    \includegraphics[scale=0.25]{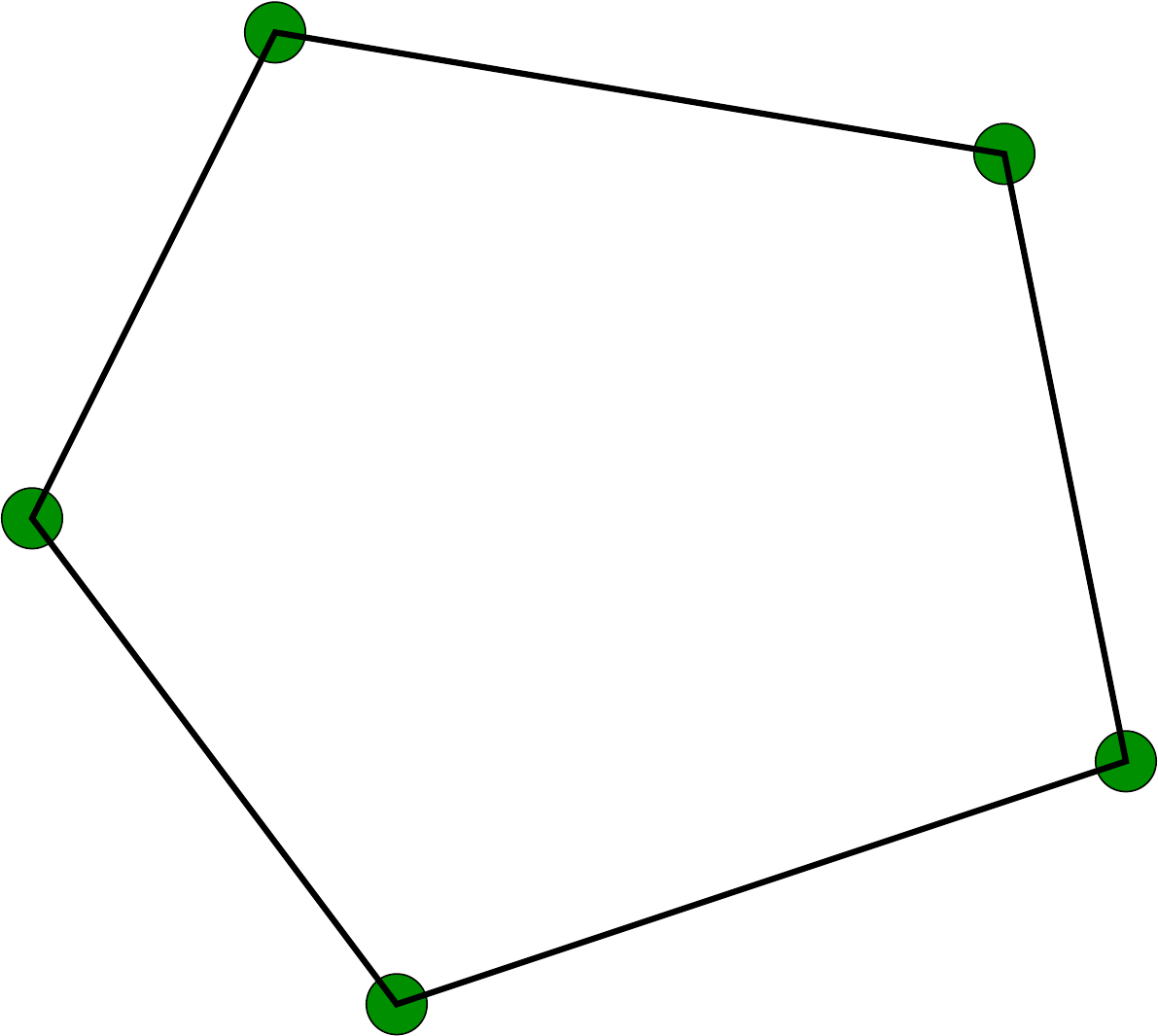}&\qquad\quad
    \includegraphics[scale=0.25]{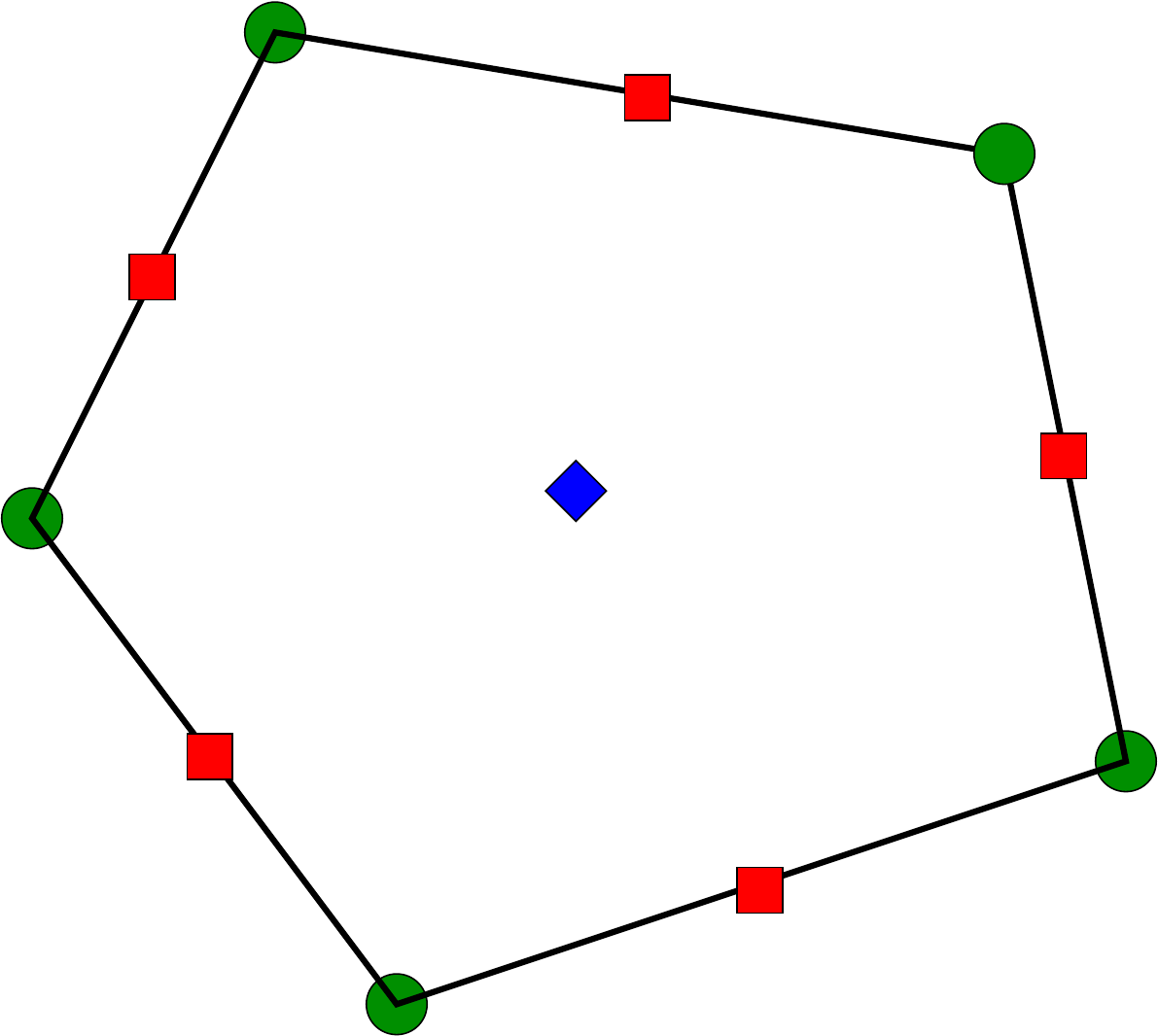}&\qquad\quad
    \includegraphics[scale=0.25]{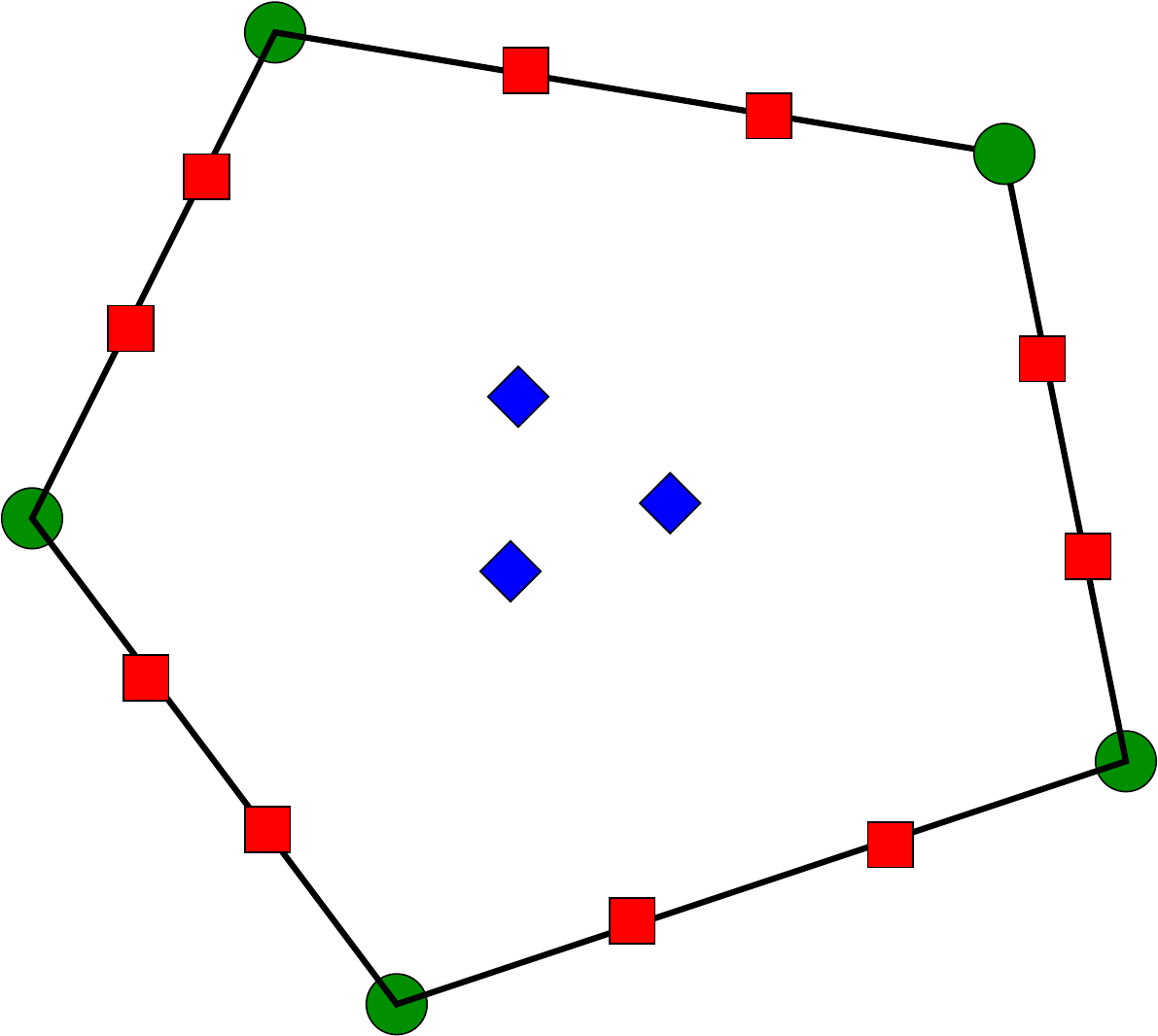}\\[0.75em]  
    \textbf{(k=1)} & \textbf{(k=2)} & \textbf{(k=3)}\\[1.5em]
  \end{tabular}
  \caption{The degrees of freedom of the scalar conforming virtual
    element spaces $\Vhk(\P)$ defined on the pentagonal cell $\P$ for
    $k=1,2,3$.}
  \label{fig:dofs}
\end{figure}
Let us move to the first step.
\begin{itemize}
\item Each virtual element function $\vsh$ is uniquely characterized
  by
  \begin{description}
  \item[]\textbf{(C1)} the values of $\vsh$ at the vertices of $\P$;
  \item[]\textbf{(C2)} the moments of $\vsh$ of order up to $k-2$ on
    each one-dimensional edge $\E\in\partial\P$:
    \begin{equation}\label{eq:CF:dofs:01}
      \frac{1}{\mE}\int_{\E}\vsh\,\ms\dS,
      \,\,\forall\ms\in\calM_{k-1}(\E),\,
      \forall\E\in\partial\P;
    \end{equation}
  \item[]\textbf{(C3)} the moments of $\vsh$ of order up to $k-2$ on
    $\P$:
    \begin{equation}\label{eq:CF:dofs:02}
      \frac{1}{\mP}\int_{\P}\vsh\,\ms\dS,
      \,\,\forall\ms\in\calM_{k-2}(\P).
    \end{equation}
  \end{description}
\end{itemize}
Figure~\ref{fig:dofs} shows the degrees of freedom of the three scalar
conforming virtual element spaces defined on a pentagonal cell for
$k=1,2,3$.
Since we assume that $\Vvhk(\P)=\big[\Vhk(\P)\big]^2$, $\Vhk(\P)$
being the local scalar conforming virtual element space, the degrees
of freedom of each component of the vector-valued functions $\vvh$ are
those described above.

\medskip
In the second step, we introduce the elliptic projection operator
$\Pin{k}:\HONE(\P)\cap\CS{0}(\overline{\P})\to\PS{k}(\P)$, so that the
elliptic projection of a function $\vsh$ is the polynomial of degree
$k$ that satisfies the variational problem given by
\begin{align}
  \int_{\P}\nabla\Pin{k}\vsh\cdot\nabla\qs\dV = \int_{\P}\nabla\vsh\cdot\nabla\qs\dV
  \qquad\forall\qs\in\PS{k}(\P),
\end{align}
with the additional condition that
\begin{align}
  \int_{\partial\P}\Pin{1}\vsh\dV = \int_{\partial\P}\vsh\dV &\qquad\textrm{for~}k=1,\\[0.5em]
  \int_{\P}\Pin{k}\vsh\dV        = \int_{\P}\vsh\dV        &\qquad\textrm{for~}k\geq2.
\end{align}
The crucial property is that $\Pin{k}\vsh$ is computable using only
the information on $\vsh$ provided by its degrees of freedom, i.e.,
the values of the linear functionals \textbf{(C1)}-\textbf{(C3)}.

\medskip 
In the third and final step, we define the conforming virtual element
space of order $k\geq1$ by
\begin{align}
  \Vhk(\P) := \Big\{
  \vsh\in\HONE(\P)\,:\,
  &
  \vs_{\hh|_{\partial\P}}\in\CS{}(\partial\P),\,
  \vs_{\hh|_\E}\in\PS{k}(\E)\,\,\forall\E\in\partial\P,\,\nonumber\\[0.25em]
  &\Delta\vsh\in\PS{k}(\P),\,
  \big(\vsh-\Pin{k}\vsh,\mu_{\hh}\big)_{\P}=0
  \quad\forall\mu_{\hh}\in\PS{k}(\P)\backslash{\PS{k-2}(\P)}
  \Big\}.
\end{align}

\begin{remark}
  A few remarkable facts characterize these scalar functional spaces:
  \begin{description}
  \item[$(i)$] their definition resorts to the \emph{enhancement
      strategy}\cite{Ahmad-Alsaedi-Brezzi-Marini-Russo:2013} for the
    conforming VEM;
  \item[$(ii)$] the degrees of freedom \textbf{(C1)}-\textbf{(C3)} are
    unisolvent~\cite{BeiraodaVeiga-Brezzi-Cangiani-Manzini-Marini-Russo:2013}
    in $\Vhk(\P)$;
  \item[$(iii)$] the scalar polynomial space $\PS{k}(\P)$ is a subset
    of $\Vhk(\P)$;
  \item[$(iv)$] the $\LTWO$-orthogonal projections
    $\Piz{k}\vsh\in\PS{k}(\P)$ and
    $\Piz{k-1}\nabla\vsh\in\big[\PS{k-1}(\P)\big]^{2}$ are computable
    for all $\vsh\in\Vhk$ using only the degrees of freedom of $\vsh$.\cite{Ahmad-Alsaedi-Brezzi-Marini-Russo:2013} 
  \end{description}
\end{remark}

\medskip
From $(iv)$, we readily find that for all $\vvh\in\Vvhk$ the
$\LTWO$-orthogonal projections $\Piz{k}\vvh\in\big[\PS{k}(\P)\big]^2$
and $\Piz{k-1}\beps(\vvh)\in\PS{k-1,\textrm{sym}}^{2\times2}(\P)$ are
also computable (without any approximation) using only the degrees of
freedom of $\vvh$.
In particular, the latter one is the solution of the finite
dimensional variational problem:
\begin{align}
  \int_{\P}\Piz{k-1}(\beps(\vvh)):\beps^{p}\dV =
  \int_{\P}\beps(\vvh):\beps^{p}\dV
  \qquad\forall\beps^{p}\in\PS{k-1,\textrm{sym}}^{2\times2}(\P),
\end{align}
i.e., the $\LTWO$-projection of the symmetric gradient $\beps(\vvh)$
onto $\PS{k-1,\textrm{sym}}^{2\times2}(\P)$, that (we recall) is the
space of symmetric $2\times2$-sized tensor-valued polynomials of
degree up to $k-1$.

\medskip 
Finally, the degrees of freedom of the global space $\Vvhk$ are
provided by collecting all the local degrees of freedom.
We note that the vertex values are the same for all the elements to
which a vertex belongs.
Similarly, the edge values at internal edges (i.e., edges shared by
two mesh elements) are the same for the two elements to which an edge
belongs.
The unisolvence of such degrees of freedom for the global space
$\Vvhk$ is an immediate consequence of the unisolvence of the local
degrees of freedom for the elemental spaces $\Vhk(\P)$.

\subsection{Virtual element bilinear forms}
In the virtual element setting, we define the bilinear forms
$\msh(\cdot,\cdot)$ and $\ash(\cdot,\cdot)$ as the sum of elemental
contributions, which are denoted by $\mshP(\cdot,\cdot)$ and
$\ashP(\cdot,\cdot)$, respectively:
\begin{align}
  \msh(\cdot,\cdot)&\,:\,\Vvhk\times\Vvhk\to\REAL,
  \quad\textrm{with}\quad\msh(\vvh,\wvh) = \sum_{\P\in\Th}\mshP(\vvh,\wvh),\nonumber\\[0.25em]
  \ash(\cdot,\cdot)&\,:\,\Vvhk\times\Vvhk\to\REAL,
  \quad\textrm{with}\quad\ash(\vvh,\wvh) = \sum_{\P\in\Th}\ashP(\vvh,\wvh).
\end{align}

The local bilinear form $\mshP(\cdot,\cdot)$ is given by
\begin{align}
  \mshP(\vvh,\wvh) 
  = \int_{\P}\rho\Piz{k}\vvh\cdot\Piz{k}\wvh\dV 
  + \SP_{m}(\vvh,\wvh),
  \label{eq:mshP:def}
\end{align}
where $\SP_{m}(\cdot,\cdot)$ is the local stabilization term.
The bilinear form $\mshP$ depends on the orthogonal projections
$\Piz{k}\vvh$ and $\Piz{k}\wvh$, which are computable from the degrees
of freedom of $\vvh$ and $\wvh$, respectively, see the previous
section.
The local form $\SP_{m}(\cdot,\cdot)\,:\,\Vvhk\times\Vvhk\to\REAL$ can
be \emph{any} symmetric and coercive bilinear form that is computable
from the degrees of freedom and for which there exist two strictly
positive real constants $\sigma_*$ and $\sigma^*$ such that
\begin{align}
  \sigma_*\msP(\vvh,\vvh)\leq\SP_{m}(\vvh,\vvh)\leq\sigma^*\msP(\vvh,\vvh)
  \quad\vvh\in\textrm{ker}\big(\Piz{k}\big)\cap\Vvhk(\P).
\end{align}
We can define computable stabilizations $\SP_{m}(\cdot,\cdot)$ by
resorting to the two-dimensional implementations of the effective
choices for the scalar case proposed in the
literature.\cite{Mascotto:2018,Dassi-Mascotto:2018}
The one used in our implementation of the method is discussed in
subsection~\ref{subsec:vem:implementation}.

The discrete bilinear form $\mshP(\cdot,\cdot)$ satisfies the following two
fundamental properties:
\begin{description}
\item[-] {\emph{$k$-consistency}}: for all $\vvh\in\Vvhk$ and for all
  $\qv\in\big[\PS{k}(\P)\big]^2$ it holds
  \begin{align}
    \label{eq:msh:k-consistency}
    \mshP(\vvh,\qv) = \msP(\vvh,\qv);
  \end{align}
\item[-] {\emph{stability}}: there exist two positive constants
  $\mu_*,\,\mu^*$, independent of $\hh$ and $\P$, such that
  \begin{align}
    \label{eq:msh:stability}
    \mu_*\msP(\vvh,\vvh)
    \leq\mshP(\vvh,\vvh)
    \leq\mu^*\msP(\vvh,\vvh)\quad\forall\vvh\in\Vhk.
  \end{align}
  The constants $\mu_*$ and $\mu^*$ may depend on upper and
    lower bounds of mass density factor $\rho$, cf.~\eqref{eq:ms:def},
    and the polynomial degree $k$.
\end{description}

\medskip
The local bilinear form $\ashP$ is given by
\begin{align}
  \ashP(\vvh,\wvh) 
  = \int_{\P}\matD\Piz{k-1}(\beps(\vvh)):\Piz{k-1}(\beps(\wvh))\dV
  + \SP_a(\vvh,\wvh),
  \label{eq:ashP:def}
\end{align}
where $\SP_{a}(\cdot,\cdot)$ is the local stabilization term.
The bilinear form $\ashP$ depends on the orthogonal projections
$\Piz{k-1}\nabla\vvh$ and $\Piz{k-1}\nabla\wvh$, which are computable
from the degrees of freedom of $\vvh$ and $\wvh$, respectively, see
the previous section.
$\SP_{a}(\cdot,\cdot)\,:\,\Vvhk\times\Vvhk\to\REAL$ and can be
\emph{any} symmetric and coercive bilinear form that is computable
from the degrees of freedom and for which there exist two strictly
positive real constants $\overline{\sigma}_*$ and
$\overline{\sigma}^*$ such that
\begin{align}
  \overline{\sigma}_*\asP(\vvh,\vvh)
  \leq\SP_{m}(\vvh,\vvh)\leq
  \overline{\sigma}^*\asP(\vvh,\vvh)
  \quad\vvh\in\textrm{ker}\big(\Piz{k}\big)\cap\Vvhk(\P).
\end{align}
Note that $\SP_a(\cdot,\cdot)$ must scale like $\asP(\cdot,\cdot)$,
i.e., as $\mathcal{O}(1)$.
We can define computable stabilizations $\SP_{a}(\cdot,\cdot)$ by
resorting to the two-dimensional implementations of the effective
choices for the scalar case proposed in the
literature.~\cite{Mascotto:2018,Dassi-Mascotto:2018}
The one used in our implementations of the method is discussed in
subsection~\ref{subsec:vem:implementation}.
  
\medskip
The discrete bilinear form $\ashP(\cdot,\cdot)$ satisfies the two
fundamental properties:
\begin{description}
\item[-] {\emph{$k$-consistency}}: for all $\vvh\in\Vvhk$ and for all
  $\qv\in\big[\PS{k}(\P)\big]^2$ it holds
  \begin{align}
    \label{eq:k-consistency}
    \ashP(\vvh,\qv) = \asP(\vvh,\qv);
  \end{align}
\item[-] {\emph{stability}}: there exist two positive constants
  $\alpha_*,\,\alpha^*$, independent of $\hh$ and $\P$, such that
  \begin{align}
    \label{eq:ash:stability}
    \alpha_*\asP(\vvh,\vvh)
    \leq\ashP(\vvh,\vvh)
    \leq\alpha^*\asP(\vvh,\vvh)\quad\forall\vvh\in\Vhk.
  \end{align}
  The constants $\alpha_*$ and $\alpha^*$ may depend on the Lam\'e
  coefficients $\mu$ and $\lambda$, and the polynomial degree
  $k$~\cite{BeiraodaVeiga-Chernov-Mascotto-Russo:2016}.
\end{description} 

\medskip
\begin{remark}
  We will use the stability of both $\mshP(\cdot,\cdot)$ and
  $\ashP(\cdot,\cdot)$ to prove that the semi-discrete virtual element
  approximation is stable in time, i.e., the approximate solution
  $\uvh(t)$ for all $t\in(0,T]$ depends continuously on the initial
  solutions $\uv_0$ and $\uv_1$ and the source term $\fv$.
\end{remark}

\medskip
\begin{remark}
  The stability of the symmetric bilinear forms $\mshP(\cdot,\cdot)$
  and $\ashP(\cdot,\cdot)$ implies that both are inner products on
  $\Vhk$, and, hence, that they are continuous.
  In fact, the Cauchy-Schwarz inequality and the local stability of
  $\msh$ imply that
  \begin{align}
    \msh(\vvh,\wvh)
    &\leq\big(\msh(\vvh,\vvh)\big)^{\frac{1}{2}}\big(\msh(\wvh,\wvh)\big)^{\frac{1}{2}}
    \leq\mu^*\big(\ms(\vvh,\vvh)\big)^{\frac{1}{2}}\big(\ms(\wvh,\wvh)\big)^{\frac{1}{2}}
    \nonumber\\[0.5em]
    &\leq\mu^*\norm{\rho}{\infty}\norm{\vvh}{0}\norm{\wvh}{0}
    \label{eq:msh:continuity}
  \end{align}
  for all $\vvh\in\Vvhk$.
  Similarly, the Cauchy-Schwarz inequality and the local stability of
  $\ash$ imply that
  \begin{align}
    \ash(\vvh,\wvh)
    \leq\big(\ash(\vvh,\vvh)\big)^{\frac{1}{2}}\big(\ash(\wvh,\wvh)\big)^{\frac{1}{2}}
    \leq\alpha^*\big(\as(\vvh,\vvh)\big)^{\frac{1}{2}}\big(\as(\wvh,\wvh)\big)^{\frac{1}{2}}
    =\alpha^*{ \RED{\snorm{\vvh}{1}\snorm{\wvh}{1}} }
  \end{align}
  for all $\vvh\in\Vvhk$.
\end{remark}

\subsection{Approximation of  the right-hand side}
We approximate the right-hand side~\eqref{eq:VEM:semi-discrete} of the
semi-discrete formulation (and, consequently,
\eqref{eq:VEM:fully-discrete} of the full discrete formulation) as
follows:
\begin{align}
  \Fsh(\vvh) 
  = \int_{\Omega}\fv\cdot\Piz{k-2}(\vvh)\dV 
  + \sum_{\E\in\GamN}\int_{ \RED{\E} }\gvN\cdot{ \RED{\vvh} }\dS
  \quad\forall\vvh\in\Vvhk.
  \label{eq:Fsh:def}
\end{align}
The linear functional $\Fsh(\cdot)$ is clearly computable since
\RED{ $\vvh|_\E$ is a polynomial and } $\Piz{k}(\vvh)$ is
computable from the degrees of freedom of $\vvh$.
Moreover, when $\gvN=0$ using the stability of the projection operator
and the Cauchy-Schwarz inequality, we note that
\begin{align}
  \ABS{\Fsh(\vvh)}
  \leq\left\vert
    \int_{\Omega}\fv(t)\cdot\Piz{k-2}(\vvh)\dV
  \right\vert
  \leq \NORM{\fv(t)}{0}\NORM{\Piz{k-2}(\vvh)}{0}
  \leq \NORM{\fv(t)}{0}\NORM{\vvh}{0}
  \quad\forall t\in[0,T].
  \label{eq:Fsh:stab}
\end{align}
We will use~\eqref{eq:Fsh:stab} in the proof of the stability of the
semi-discrete virtual element approximation.

\subsection{The hitchhiker's guide of the VEM for the elastodynamics equation}
\label{subsec:vem:implementation}
In this subsection, we present the implementation details that are
practically useful and the basic steps to reduce the implementation of
the VEM to the calculation of a few small elemental matrices; \BLUE{cf. also Algorithm~1
where the sketch of the pseudocode is presented}.
In fact, the implementation of the VEM relies on the
$\LTWO$-orthogonal projection matrices for scalar shape functions and
their gradients.\cite{BeiraodaVeiga-Brezzi-Marini-Russo:2014}


We build the mass and stiffness matrices of the VEM by applying
definitions~\eqref{eq:mshP:def} and~\eqref{eq:ashP:def} to the
vector-valued shape functions generating $\Vvhk(\P)$.
Let $\bvarphi_i$ be the $i$-th ``canonical'' vector-valued basis
function of the global virtual element space $\Vvhk$.
We define the \emph{mass matrix} $\Mv=(\matM_{ij})$ and the stiffness
matrix $\Kv=(\matK_{ij})$ by
$\matM_{ij}=\msh(\bvarphi_j,\bvarphi_i)$ and
$\matK_{ij}=\ash(\bvarphi_j,\bvarphi_i)$, respectively.
The stability condition~\eqref{eq:msh:stability} implies that
$\norm{\bvarphi_i}{0}^2\lesssim\msh(\bvarphi_i,\bvarphi_i)\lesssim\norm{\bvarphi_i}{0}^2$
for every $i$.
Therefore, mass matrix $\Mv$ is strictly positive definite (and
symmetric) and, hence, nonsingular.
Similarly, the stability condition~\eqref{eq:ash:stability} implies
that
$\snorm{\bvarphi_i}{1}^2\lesssim\ash(\bvarphi_i,\bvarphi_i)\lesssim\snorm{\bvarphi_i}{1}^2$.
Therefore, the stiffness matrix $\Kv$ is non-negative
definite (and symmetric).

\medskip
As discussed in the previous subsections, the virtual element space of
two-dimensional vector-valued functions $\vvh\in\Vvhk(\P)$ is built by
taking the two components of $\vvh$ in the scalar virtual element
space $\Vhk(\P)$.
Let $\phi_i$ be the shape function of $\Vhk(\P)$ that is associated
with the $i$-th degree of freedom, so that by definition, its $i$-th
degree of freedom is equal to one, while all other degrees of freedom
are equal to zero.
Here, we consider the index $i$ (and $j$ in the next formulas) as
running from $1$ to $N^{\dofs}$, where $N^{\dofs}$ is the dimension of
$\Vhk(\P)$.
Using this convention, the dimension of the vector virtual element
space $\Vvhk(\P)$ is actually $2N^{\dofs}$.
Accordingly, the set of ``canonical'' shape functions that generate
$\Vvhk(\P)$ is given by vector-valued functions of the form
$\bphiUP_i=(\phi_i,0)^T$ and $\bphiDW_i=(0,\phi_i)^T$.

\medskip
For the exposition's sake, we simplify the notation for the orthogonal
projections.
More precisely, we use the ``hat'' symbol over $\phi_i$, e.g.,
$\widehat{\phi_i}$, to denote the $\LTWO$-projection of $\phi_i$ onto
the polynomials of degree $k$.
We also denote the partial derivatives of $\phi_i$ along the $x$ and
$y$ direction by $\px\phi_i$ and $\py\phi_i$, respectively, and, with
a small abuse of notation, their $\LTWO$-orthogonal projections onto
the polynomials of degree $k-1$ by $\pxh{\phi_i}$ and $\pyh{\phi_i}$
As discussed previously, all these projections are computable from the
degrees of freedom of $\phi_i$.\cite{BeiraodaVeiga-Brezzi-Marini-Russo:2014}

\medskip
Let $\Mv=\Mv^c+\Mv^s$ and $\Kv=\Kv^c+\Kv^s$ be the mass and stiffness
matrices, that we write as the sum of the consistence term, i.e.,
matrices $\Mv^c$ and $\Kv^c$, and stability term, i.e., matrices
$\Mv^s$ and $\Kv^s$.
In the rest of this section, we detail the construction of each one of
these four matrix terms.

\medskip
The splitting of the polygonal shape functions in vector-valued
functions like $\bphiUP_i$ and $\bphiDW_i$, where only one components
is actually nonzero, simplifies the expression of the mass matrix
significantly.
Matrix $\Mv^{c}$ is, indeed, block diagonal, each block has size
$N^{\dofs}\times N^{\dofs}$, and its $(ij)$-th entry is given by
\begin{align}
  \Mv^c_{ij} = \int_{\P}\widehat{\phi}_{i}\,\widehat{\phi}_{j}\dV.
\end{align}
Let $\bPi{k}$ be the projection matrix of the set of scalar shape
functions $\{\phi_i\}_{i=1}^{N^{\dofs}}$.
Matrix $\bPi{k}$ has size $N^{\poly}\times N^{\dofs}$, where
$N^{\poly}$ is the dimension of $\PS{k}(\P)$.
The coefficients of the expansion of $\widehat{\phi_i}$ on the monomial basis
are on the $i$-th columns of $\bPi{k}$, so that:
\begin{align}
  \widehat{\phi_i}(x,y) =
  \sum_{\alpha=1}^{N^{k}}\ms_{\alpha}(x,y)\big(\bPi{k}\big)_{\alpha,i}.
\end{align}
Using this polynomial expansion we find that 
\begin{align}
  \Mv^c_{ij} = \sum_{\alpha,\beta=1}^{N^{\poly}} \Qv_{\alpha,\beta}\big(\bPi{k}\big)_{\alpha,i}\big(\bPi{k}\big)_{\beta,j},
\end{align}
where $\Qv$ is the mass matrix of the monomials,
\begin{align}
  \Qv_{\alpha,\beta} = \int_{\P}\ms_{\beta}(x,y)\,\ms_{\alpha}(x,y)\,\dV.
  \label{eq:Qv:def}
\end{align}
The equivalent matrix form is
\begin{align}
  \Mv^{c} = \big(\bPi{k}\big)^T\,\Qv\,\bPi{k}.
  \label{eq:Mvc:def}
\end{align}
The stability matrix used in this work is obtained from the
stabilization bilinear form
\begin{align}
  \SP_{m}(\vvh,\wvh) =
  \overline{\rho}\hP^2
  \sum_{\ell=1}^{2N^{\dofs}}\textrm{dof}_{\ell}(\vvh)\,\textrm{dof}_{\ell}(\wvh),
\end{align}
where $\overline{\rho}$ is the cell-average of $\rho$ over $\P$.
We recall that $\ell$ runs from $1$ to $2N^{\dofs}$ since $N^{\dofs}$
is the number of degrees of freedom of the scalar virtual element
space.
Using $\bphiUP_i=(\phi_i,0)^T$ and $\bphiDW_i=(0,\phi_i)^T$ for the
vector basis functions, the stability part of the mass matrix is
provided by the formula:
\begin{align}
  \Mv^s_{ij} =
  \overline{\rho}\hP^2
  \sum_{\ell=1}^{2N^{\dofs}}\bigg[
  &
  \textrm{dof}_{\ell}\big( (1-\Piz{k})\bphiUP_i \big)\,\textrm{dof}_{\ell}\big( (1-\Piz{k})\bphiUP_j \big) 
  \nonumber\\[0.5em]
  &+\textrm{dof}_{\ell}\big( (1-\Piz{k})\bphiDW_i \big)\,\textrm{dof}_{\ell}\big( (1-\Piz{k})\bphiDW_j \big)
  \bigg].
\end{align}
Using $\bphiUP_i=(\phi_i,0)^T$ and $\bphiDW_i=(0,\phi_i)^T$ for the
vector basis functions, the stability part of the mass matrix is
provided by the formula:
\begin{align}
  \Mv^s =
  \overline{\rho}\hP^2
  \left[
    \begin{array}{cc}
      \big(\Iv-\bPi{k}\big)^T\,\big(\Iv-\bPi{k}\big) & 0 \\[0.5em]
      0 & \big(\Iv-\bPi{k}\big)^T\,\big(\Iv-\bPi{k}\big)
    \end{array}
  \right].
  \label{eq:Mvs:def}
\end{align}
The block-diagonal structure above is induced by our choice of using
the vector basis functions $\bphiUP_i=(\phi_i,0)^T$ and
$\bphiDW_i=(0,\phi_i)^T$.

\medskip
The situation is more complex for the stiffness matrix, where the
splitting ``$\UP-\DW$'' induces the $2\times 2$ splitting:
\begin{align}
  \Kv^c = 
  \left[
    \begin{array}{ll}
      \Kv^{c,\UP,\UP} &\quad \Kv^{c,\UP,\DW} \\[0.5em]
      \Kv^{c,\DW,\UP} &\quad \Kv^{c,\DW,\DW}
    \end{array}
  \right].
  \label{eq:Kvc:def}
\end{align}
To detail each one of these four submatrices, consider first a generic
vector-valued field $\wv=(\wsx,\wsy)^T$.
From the standard definition of the tensor fields $\beps(\wv)$ and
$\bsig(\wv)$, we immediately find that:
\begin{align}
  \beps(\wv) = 
  \left[
    \begin{array}{cc}
      \px\wsx                      & \small{\frac{1}{2}}(\px\wsy+\py\wsx) \\[0.5em]
      \frac{1}{2}(\px\wsy+\py\wsx) & \py\wsy   
    \end{array}
  \right],
\end{align}
and, according to~\eqref{eq:bsig:def},
\begin{align}
  \bsig(\wv) = 
  \left[
    \begin{array}{cc}
      (2\mu+\lambda)\px\wsx+\lambda\py\wsy & \mu(\px\wsy+\py\wsx) \\[0.5em]
      \mu(\px\wsy+\py\wsx)                 & \lambda\px\wsx + (2\mu+\lambda)\py\wsy
    \end{array}
  \right].
\end{align}
Then, we take
$\wv\in\big\{\bphiUP_i,\,\bphiDW_i\big\}_{i=1}^{N^{\dofs}}$, so that
\begin{align}
  \beps(\bphiUP_i) = 
  \left[
    \begin{array}{cc}
      \px\phi_i            & \frac{1}{2}\py\phi_i \\[0.5em]
      \frac{1}{2}\py\phi_i & 0   
    \end{array}
  \right],\qquad
  \beps(\bphiDW_i) = 
  \left[
    \begin{array}{cc}
      0                    & \frac{1}{2}\px\phi_i \\[0.5em]
      \frac{1}{2}\px\phi_i & \py\phi_i   
    \end{array}
  \right],
\end{align}
and
\begin{align}
  \bsig(\bphiUP_i) = 
  \left[
    \begin{array}{cc}
      (2\mu+\lambda)\px\phi_i & \mu\py\phi_i \\[0.5em]
      \mu\py\phi_i            & \lambda\px\phi_i
    \end{array}
  \right],\qquad
  \bsig(\bphiDW_i) = 
  \left[
    \begin{array}{cc}
      \lambda\py\phi_i & \mu\px\phi_i \\[0.5em]
      \mu\px\phi_i     & (2\mu+\lambda)\py\phi_i
    \end{array}
  \right].
\end{align}
Using such definitions, a straightforward calculation immediately
provides us the formulas for the stiffness submatrices:
\begin{align}
  \begin{array}{lll}
    \Kv^{c,\UP,\UP}_{ij}
    &
    = \int_{\P} \widehat{\bsig(\bphiUP_j)}\colon\widehat{\beps(\bphiUP_i)}\,\dV
    &= (2\mu+\lambda)\int_{\P} \pxh{\phi_j}\,\pxh{\phi_i}\,\dV 
    + \mu\int_{\P} \pyh{\phi_j}\,\pyh{\phi_i}\,\dV,\\[1.em]
    \Kv^{c,\UP,\DW}_{ij} 
    &
    = \int_{\P} \widehat{\bsig(\bphiUP_j)}\colon\widehat{\beps(\bphiDW_i)}\,\dV
    &= \mu\int_{\P} \pyh{\phi_j}\,\pxh{\phi_i}\,\dV
    + \lambda\int_{\P} \pxh{\phi_j}\,\pyh{\phi_i}\,\dV,\\[1.em]
    \Kv^{c,\DW,\UP}_{ij} 
    &
    = \int_{\P} \widehat{\bsig(\bphiDW_j)}\colon\widehat{\beps(\bphiUP_i)} \,\dV
    &= \lambda\int_{\P} \pyh{\phi_j}\,\pxh{\phi_i}\,\dV
    + \mu\int_{\P} \pxh{\phi_j}\,\pyh{\phi_i}\,\dV,\\[1.em]
    \Kv^{c,\DW,\DW}_{ij}
    &
    = \int_{\P} \widehat{\bsig(\bphiDW_j)}\colon\widehat{\beps(\bphiDW_i)} \,\dV
    &= \mu\int_{\P} \pxh{\phi_j}\,\pxh{\phi_i}\,\dV
    + (2\mu+\lambda)\int_{\P} \pyh{\phi_j}\,\pyh{\phi_i}\,\dV. 
  \end{array}
    \label{eq:Kvc:def-1}
\end{align}
A thorough inspection of these formulas reveals that we only need the
two projection matrices $\bPix{k-1}$ and $\bPiy{k-1}$ such that
\begin{align}
  \pxh{\phi_i}(x,y) = \sum_{\alpha=1}^{N^{k-1}}\ms_{\alpha}(x,y)\big(\bPix{k-1}\big)_{\alpha,i},
  \qquad
  \pyh{\phi_i}(x,y) = \sum_{\alpha=1}^{N^{k-1}}\ms_{\alpha}(x,y)\big(\bPiy{k-1}\big)_{\alpha,i},
\end{align}
and the four additional matrices involving the derivatives of
monomials up to the degree $k$:
\begin{align}
  \Qv^{xx}_{\alpha,\beta} = \int_{\P}\px\ms_{\beta}\px\ms_{\alpha},\dV,\qquad\quad
  \Qv^{xy}_{\alpha,\beta} = \int_{\P}\px\ms_{\beta}\py\ms_{\alpha},\dV,\label{eq:Qv-xx:def}\\[1em]
  \Qv^{yx}_{\alpha,\beta} = \int_{\P}\py\ms_{\beta}\px\ms_{\alpha},\dV,\qquad\quad
  \Qv^{yy}_{\alpha,\beta} = \int_{\P}\py\ms_{\beta}\px\ms_{\alpha},\dV.\label{eq:Qv-yy:def}
\end{align}
Using such matrices we can reformulate the entries of the four
subblocks of matrix $\Kv^c$ as follows:
\begin{align}
  \begin{array}{lll}
    \Kv^{c,\UP,\UP}_{ij}
    &= (2\mu+\lambda) \sum_{\alpha,\beta} \Qv^{xx}_{\alpha,\beta}\,\big(\bPix{k-1}\big)_{\alpha,i}\,\big(\bPix{k-1}\big)_{\beta,j}
    &+ \mu             \sum_{\alpha,\beta} \Qv^{yx}_{\alpha,\beta}\,\big(\bPiy{k-1}\big)_{\alpha,i}\,\big(\bPiy{k-1}\big)_{\beta,j},\\[1.em]
    \Kv^{c,\UP,\DW}_{ij} 
    &= \mu   \sum_{\alpha,\beta} \Qv^{xy}_{\alpha,\beta}\,\big(\bPix{k-1}\big)_{\alpha,i}\,\big(\bPiy{k-1}\big)_{\beta,j}
    &+ \lambda\sum_{\alpha,\beta} \Qv^{yx}_{\alpha,\beta}\,\big(\bPiy{k-1}\big)_{\alpha,i}\,\big(\bPix{k-1}\big)_{\beta,j},\\[1.em]
    \Kv^{c,\DW,\UP}_{ij} 
    &= \lambda\sum_{\alpha,\beta} \Qv^{xy}_{\alpha,\beta}\,\big(\bPix{k-1}\big)_{\alpha,i}\,\big(\bPiy{k-1}\big)_{\beta,j}
    &+ \mu     \sum_{\alpha,\beta} \Qv^{yx}_{\alpha,\beta}\,\big(\bPiy{k-1}\big)_{\alpha,i}\,\big(\bPix{k-1}\big)_{\beta,j},\\[1.em]
    \Kv^{c,\DW,\DW}_{ij}
    &= \mu          \sum_{\alpha,\beta} \Qv^{xx}_{\alpha,\beta}\,\big(\bPix{k-1}\big)_{\alpha,i}\,\big(\bPix{k-1}\big)_{\beta,j}
    &+ (2\mu+\lambda)\sum_{\alpha,\beta} \Qv^{yy}_{\alpha,\beta}\,\big(\bPiy{k-1}\big)_{\alpha,i}\,\big(\bPiy{k-1}\big)_{\beta,j}.
  \end{array}
\end{align}
The equivalent compact matrix form is:
\begin{align}
  \begin{array}{lll}
    \Kv^{c,\UP,\UP}
    &= (2\mu+\lambda) \big(\bPix{k-1}\big)^T\,\Qv^{xx}\,\bPix{k-1}
    &+ \mu            \big(\bPiy{k-1}\big)^T\,\Qv^{yy}\,\bPiy{k-1},\\[1.em]
    \Kv^{c,\UP,\DW} 
    &= \mu    \big(\bPix{k-1}\big)^T\,\Qv^{xy}\,\bPiy{k-1}
    &+ \lambda\big(\bPiy{k-1}\big)^T\,\Qv^{yx}\,\bPix{k-1},\\[1em]
    \Kv^{c,\DW,\UP}
    &= \lambda\big(\bPix{k-1}\big)^T\,\Qv^{xy}\,\bPiy{k-1}
    &+ \mu    \big(\bPiy{k-1}\big)^T\,\Qv^{yx}\,\bPix{k-1},\\[1em]
    \Kv^{c,\DW,\DW}
    &= \mu           \big(\bPix{k-1}\big)^T\,\Qv^{xx}\,\bPix{k-1}
    &+ (2\mu+\lambda)\big(\bPiy{k-1}\big)^T\,\Qv^{yy}\,\bPiy{k-1}.
  \end{array}
  \label{eq:Kvc:def-2}
\end{align}

\medskip
The stability matrix used in this work is obtained from the
stabilization bilinear form
\begin{align}
  \SP_{a}(\vvh,\wvh) =
  \max( 2\overline{\mu}, \overline{\lambda} )
  \sum_{\ell=1}^{2N^{\dofs}}\textrm{dof}_{\ell}(\vvh)\,\textrm{dof}_{\ell}(\wvh),
\end{align}
where $\overline{\mu}$ and $\overline{\lambda}$ are the cell averages
of $\mu$ and $\lambda$, respectively.
Since we assume that the Lam\'e coefficients are constant,
$\overline{\mu}$ and $\overline{\lambda}$ are respectively equal to
$\mu$ and $\lambda$.
We recall that index $\ell$ runs from $1$ to $2N^{\dofs}$ because
$N^{\dofs}$ is the number of degrees of freedom of the scalar virtual
element space.
Using $\bphiUP_i=(\phi_i,0)^T$ and $\bphiDW_i=(0,\phi_i)^T$ for the
vector basis functions, the stability part of the mass matrix is
provided by the formula:
\begin{align}
  \Kv^s_{ij} =
  \max( 2\overline{\mu}, \overline{\lambda} )
  \sum_{\ell=1}^{N^{\dofs}}\bigg[
  &
  \textrm{dof}_{\ell}\left( \left(1-\Pi^{\nabla,\P}_{k}\right) \bphiUP_i \right)
  \textrm{dof}_{\ell}\left( \left(1-\Pi^{\nabla,\P}_{k}\right) \bphiUP_j \right) \nonumber\\
  &+
  \textrm{dof}_{\ell}\left( \left(1-\Pi^{\nabla,\P}_{k}\right) \bphiDW_i \right)
  \textrm{dof}_{\ell}\left( \left(1-\Pi^{\nabla,\P}_{k}\right) \bphiDW_j \right)
  \bigg].
  \label{eq:stiffness:stab:bilinear:form}
\end{align}
The stability matrix can be written in block-diagonal form, and in
this work we consider
\begin{align}
  \Kv^s =
  \max( 2\overline{\mu}, \overline{\lambda} )
  \left[
    \begin{array}{cc}
      \big(\Iv-\bPin{k}\big)^T\,\big(\Iv-\bPin{k}\big) & 0 \\[0.5em]
      0 & \big(\Iv-\bPin{k}\big)^T\,\big(\Iv-\bPin{k}\big)
    \end{array}
  \right],
  \label{eq:stiffness:stab:matrix}
\end{align}
where $\bPin{k}$ is the elliptic projection matrix for the scalar
case.
An alternative formulation can be obtained by considering the
orthogonal projector $\Pi^{0,\P}_{k}$ instead of the elliptic
projector $\Pi^{\nabla,\P}_{k}$
in~\eqref{eq:stiffness:stab:bilinear:form},\ and, consistently, the
orthogonal projection matrix $\bPi{k}$ instead of $\bPin{k}$
in~\eqref{eq:stiffness:stab:matrix}.
\section{Stability and convergence analysis for the semi-discrete problem}
\label{sec4:convergence}

The main results of this section are stated in
Theorems~\ref{theorem:semi-discrete:stability},
\ref{theorem:semi-discrete:convergence}
and~\ref{theorem:L2:convergence} below, that prove the
stability and convergence in the mesh dependent energy norm that will
be introduced in~\eqref{eq:three-bar-norm} and the convergence in the
$\LTWO(\Omega)$-norm, respectively.
For exposition's sake, we set $\rho=1$ in~\eqref{eq:pblm:strong:A},
\eqref{eq:ms:def} and~\eqref{eq:mshP:def} and $\gvN=0$
in~\eqref{eq:pblm:strong:C}, \eqref{eq:Fv:def} and~\eqref{eq:Fsh:def}.

\subsection{Technicalities and preliminary results}
\label{subsec:technicalities}
To carry out the analysis of this section and derive \emph{a priori}
estimates, we need the error estimates for piecewise polynomial
approximations and interpolation in the virtual element space $\Vvhk$
that are stated in the two following lemmas.

\medskip
\begin{lemma}\label{lemma:hp:convergence:1}
  Let $\PS{k}(\Th)$ be the space of piecewise discontinuous polynomials of
  degree up to $k$ defined on mesh $\Th$.
  Under the mesh regularity assumption \ASSUM{A0}, for all
  $\uv\in\HS{m+1}(\Omega)$, $m\in\INTG$, there exists a vector-valued
  field $\uv_\pi\in\big[\PS{k}(\Th)\big]^2$ such that
  \begin{align}
    \begin{array}{rll}
      \NORM{ \uv-\uv_\pi }{0}    &\lesssim\frac{\hh^{\mu+1}}{k^{m+1}}\NORM{\uv}{m+1} & \quad\mu=\min(k,m),\,m\geq0,\\[1.00em]
      \ABS{ \uv-\uv_\pi }_{1,\hh} &\lesssim\frac{\hh^{\mu}}{k^{m}}\NORM{\uv}{m+1}    & \quad\mu=\min(k,m),\,m\geq1.
      \label{eq:hp:convergence:1}
    \end{array}
  \end{align}  
\end{lemma}
\begin{proof}
  The assertion of the lemma is proved
  in Reference~\cite{BeiraodaVeiga-Chernov-Mascotto-Russo:2016}.
\end{proof}

\medskip
\begin{lemma}\label{lemma:hp:convergence:2}
  Under the mesh regularity assumption \ASSUM{A0}, for all
  $\uv\in\HS{m+1}(\Omega)$, $m\in\INTG$, there exists a virtual
  element interpolant $\uvI\in\Vvhk$ such that
  \begin{align}
    \begin{array}{rll}
      \NORM{\uv-\uvI}{0} &\lesssim\frac{\hh^{\mu+1}}{k^{m}}\,\NORM{\uv}{m+1}  & \quad\mu=\min(k,m),\,m\geq0,\\[0.75em]
      \ABS{\uv-\uvI}_{1}  &\lesssim\frac{\hh^{\mu}}{k^{m}}\,\NORM{\uv}{m+1}   & \quad\mu=\min(k,m),\,m\geq1.
    \end{array}
    \label{eq:hp:convergence:2}
  \end{align}
\end{lemma}
\begin{proof}
  The assertion of the lemma is proved
  in Reference~\cite{BeiraodaVeiga-Chernov-Mascotto-Russo:2016}. See also ~\cite{Certik-Gardini-Manzini-Mascotto-Vacca:2019}.
\end{proof}

\begin{remark}\label{remark5}
  The $L^2$-estimate provided by Lemma~\ref{lemma:hp:convergence:2} is
  suboptimal in $k$ as we have a dependence like $k^{m}$ instead of
  $k^{m+1}$ in the fraction denominators.
   \RED{ We point out that, whenever $\Omega$ is convex and the partition
    $\Th$ is made of convex elements, following the proof of
    Proposition 1 in \cite{BeiraodaVeiga-Manzini-Mascotto:2019} and
    combining the energy estimate in \eqref{eq:hp:convergence:2}
    (where $\uvI$ has to be defined as in eq. (20) of
    \cite{BeiraodaVeiga-Manzini-Mascotto:2019}) with the standard
    Aubin-Nitsche trick it is possible to prove that the following
    holds
    \begin{align}
      \begin{array}{rll} \NORM{\uv-\uvI}{0}
        \lesssim\frac{\hh^{\mu+1}}{k^{m+1}}\,\NORM{\uv}{m+1} &
        \quad\mu=\min(k,m),\,m\geq0.\\[0.75em]
      \end{array}
    \end{align}
    provided that $hp$ optimal $L^2$-error estimates hold for the
    continuous piecewise polynomial approximation $\varphi_k^h$
    introduced in eq. (4.10) of
    \cite{Beirao-Chernov-Mascotto-Russo:2018}.  }
\end{remark}

\subsection{Stability}
\label{subsec:stability}
The semi-discrete virtual element approximation of the time-dependent
linear elastodynamics problem in variational form is stable and
convergent, cf. Theorems~\ref{theorem:semi-discrete:stability} and
\ref{theorem:semi-discrete:convergence} below, which are the main
results of this section.
Moreover, we state Theorems~\ref{theorem:semi-discrete:stability} and
\ref{theorem:semi-discrete:convergence} below by using the
\emph{energy} norm
\begin{align}
  \TNORM{\vvh(t)}{}^2 
  = \NORM{ \rho^{\frac{1}{2}}\dot{\vv}_{\hh}(t) }{0}^2 
  + \abs{ \vvh(t) }_{1}^2,
  \qquad t\in[0,T],
  \label{eq:three-bar-norm}
\end{align}
which is defined for all $\vs\in\Vvhk$.
The local stability property of the bilinear forms $\msh(\cdot,\cdot)$
and $\ash(\cdot,\cdot)$ readily imply the equivalence relation
\begin{align}
  \msh(\dot{\vv}_{\hh},\dot{\vv}_{\hh}) + \ash(\vvh,\vvh)
  \lesssim\TNORM{\vvh(t)}{}^2\lesssim \msh(\dot{\vv}_{\hh},\dot{\vv}_{\hh}) +
  \ash(\vvh,\vvh)
  \label{eq:energy:norm:equivalence}
\end{align}
for all time-dependent virtual element functions $\vvh(t)$ with square
integrable derivative $\dot{\vv}_{\hh}(t)$.

\medskip
\begin{remark}
  \RED{ The hidden constants in~\eqref{eq:energy:norm:equivalence}
    are independent of the mesh size parameter $\hh$.
    However, they may depend on the stability parameters $\mu_*$,
    $\mu^*$, $\alpha_*$, $\alpha^*$, the regularity constant $\varrho$
    of the mesh, and the polynomial degree
    $k$.\cite{Beirao-Chernov-Mascotto-Russo:2018}
    The dependence on $k$ does not seem to have a relevant impact on
    the optimality of the convergence rates in our numerical
    experiments in Section~\ref{sec5:numerical}.  }
\end{remark}

\medskip
\begin{theorem}\label{theorem:semi-discrete:stability}
  Let $\fv\in\LTWO( (0,T]; [\LTWO(\Omega)]^2 )$ and let
  $\uvh\in\CS{2}((0,T]; \Vvhk)$ be the solution of
  \eqref{eq:VEM:semi-discrete}.
  Then, it holds
  \begin{equation}
    \TNORM{\uvh(t)}{} \lesssim 
    \TNORM{ (\uv_{0})_I }{} + \int_0^t \NORM{ \fv(\tau) }{0,\Omega} d\tau. 
  \end{equation}
  The hidden constant in $\lesssim$ is independent of $\hh$, but may
  depend on the model parameters and approximation constants and the
  polynomial degree $k$.
\end{theorem}
\begin{proof}
  We substitute $\vvh=\dot{\uv}_{\hh}(t)$
  in~\eqref{eq:VEM:semi-discrete} and, for all $t\in(0,T]$, we obtain
  \begin{align}
    \msh(\ddot{\uv}_{\hh},\dot{\uv}_{\hh})+\ash(\uvh,\dot{\uv}_{\hh})=\Fsh(\dot{\uv}_{\hh}).
    \label{eq:LHS}
  \end{align}
  Since both $\msh(\cdot,\cdot)$ and $\ash(\cdot,\cdot)$ are symmetric
  bilinear forms, a straightforward calculation yields
  \begin{align*}
    \frac{1}{2}\frac{d}{\dt}\big( \msh(\dot{\uv}_{\hh},\dot{\uv}_{\hh}) + \ash(\uvh,\uvh) \big) =
    \msh(\ddot{\uv}_{\hh},\dot{\uv}_{\hh})+\ash(\uvh,\dot{\uv}_{\hh}).
  \end{align*}
  We substitute this expression in the left-hand side
  of~\eqref{eq:LHS}, we integrate in time the resulting equation from
  $0$ to the intermediate time $t$, and using the definition of norm
  $\TNORM{\,\cdot\,}{}$ in~\eqref{eq:three-bar-norm} and the
  equivalence relation~\eqref{eq:energy:norm:equivalence}, we find
  that
  \begin{align*}
    \TNORM{\uvh(t)}{}^2
    &\lesssim
    \msh(\dot{\uv}_{\hh}(t),\dot{\uv}_{\hh}(t)) + \ash(\uvh(t),\uvh(t))
    \\[0.5em]
    &= \msh(\dot{\uv}_{\hh}(0),\dot{\uv}_{\hh}(0)) + \ash(\uvh(0),\uvh(0))
    + 2\int_{0}^{t}\Fsh(\dot{\uv}_{\hh}(\tau))d\tau
    \\[0.5em]
    &\lesssim
    \TNORM{\uvh(0)}{}^2
    +\int_{0}^{t}\Fsh(\dot{\uv}_{\hh}(\tau))d\tau.
  \end{align*}
  Since $\uvh(0)=(\uv_{0})_{\INTP}$, and using~\eqref{eq:Fsh:stab}
  (with $\vvh=\dot{\uv}_{\hh}$), we find that
  \begin{align*}
    \TNORM{\uvh(t)}{}^2
    \lesssim\TNORM{ (\uv_{0})_{\INTP}}{}^2  + \int_{0}^{t} \Fsh\big(\dot{\uv}_{\hh}(\tau)\big) d\tau
    \lesssim\TNORM{ (\uv_{0})_{\INTP} }{}^2 + \int_0^t \NORM{\fv(\tau)}{0}\,\NORM{\dot{\uv}_{\hh}(\tau)}{0}\,d\tau.
  \end{align*}
  The thesis follows on applying Lemma~A5, p.~157 of Reference\cite{Brezis:1973}.
\end{proof}

\subsection{Convergence analysis in the energy norm}
\label{subsec:convergence:energy:norm}
In this section, we prove the convergence of the semi-discrete virtual
element approximation in the energy norm \eqref{eq:three-bar-norm}.
\emph{A priori} error estimates of the approximation error are derived
from Theorem~\ref{theorem:semi-discrete:convergence} as a corollary,
which is reported at the end of the section, by using approximation
results for discontinuous polynomial and virtual element spaces.

\medskip
\begin{theorem}\label{theorem:semi-discrete:convergence}
  Let $\uv\in\CS{2}\big( (0,T]; [H^{m+1}(\Omega)]^2 \big)$,
  $m\in\INTG$, be the exact solution of problem~\eqref{eq:weak:A}.
  Let $\uvh\in\Vvhk$ be the solution of the semi-discrete
  problem~\eqref{eq:VEM:semi-discrete} under the mesh regularity
  assumption \ASSUM{A0}.
  Then, for all $t\in[0,T]$ and all discontinuous polynomial
  approximations $\uv_{\pi}(t)$ of $\uv(t)$, it holds that
  \begin{align}
    \TNORM{\uv(t)-\uvh(t)}{}\lesssim
    \sup_{\tau\in[0,T]}\Gs_0(\tau)
    +\int_{0}^{t}\Gs_1(\tau)d\tau,
    \label{eq:semi-discrete:convergence}
  \end{align}
  where 
  \begin{align}
    \Gs_0(\tau) &= \NORM{\dot{\uv}(\tau)-\dot{\uv}_{\INTP}(\tau)}{0} + \snorm{\uv(\tau)-\uvI(\tau)}{1} + \RED{ \snorm{\uv(\tau)-\uv_{\pi}(\tau)}{1,h} },
    \label{eq:semi-discrete:convergence:G0}
    \\[0.5em]
    \Gs_1(\tau) &= 
    \NORM{\ddot{\uv}(\tau)-\ddot{\uv}_{\INTP}(\tau)}{0} + \NORM{\ddot{\uv}(\tau)-\ddot{\uv}_{\pi}(\tau)}{0} +
    \snorm{\dot {\uv}(\tau)-\dot {\uv}_{\INTP}(\tau)}{1} + \RED{ \snorm{\dot {\uv}(\tau)-\dot {\uv}_{\pi}(\tau)}{1,h} } \nonumber\\[0.5em]
    &+  
    \sup_{\vvh\in\Vvhk\backslash{\REAL^{\dims}}}\frac{\ABS{\Fs(\vvh)-\Fsh(\vvh)}}{\snorm{\vvh}{1}}.
    \label{eq:semi-discrete:convergence:G1}
  \end{align}
  The hidden constant in $\lesssim$ is independent of
  $\hh$, but may depend on the model parameters, the approximation
  constants, and the polynomial degree $k$, and the final observation time $T$.
\end{theorem}
\begin{proof}
  Since $\Vvhk$ is a subspace of $\Vv$, we can take $\vvh\in\Vvhk$ as
  test function in~\eqref{eq:weak:A} and substract from
  \eqref{eq:VEM:semi-discrete} to find the error equation:
  \begin{align}
    \ms(\ddot{\uv}(t),\vvh) - \msh(\ddot{\uv}_{\hh}(t),\vvh) 
    + \as(\uv(t),\vvh) - \ash(\uv_{\hh}(t),\vvh) 
    = \Fs(\vvh) - \Fsh(\vvh),
    \label{eq:proof:energy:05}
  \end{align}
  which holds for all $\vvh\in\Vvhk$.
  Next, we rewrite this equation as
  $\TERM{T}{1}+\TERM{T}{2}=\TERM{T}{3}$, with the definitions:
  \begin{align*}
    \TERM{T}{1} &:= \ms(\ddot{\uv},\vvh) - \msh(\ddot{\uv}_{\hh},\vvh),\nonumber\\[0.5em]
    \TERM{T}{2} &:= \as(\uv,\vvh) - \ash(\uvh,\vvh),\nonumber\\[0.5em]
    \TERM{T}{3} &:= \Fs(\vvh) - \Fsh(\vvh),
  \end{align*}
  where we dropped out the explicit dependence on $t$ to simplify the
  notation.
  We analyze each term separately.
  First, we rewrite $\TERM{T}{1}$ as
  \begin{align*}
    \TERM{T}{1} 
    = \msh(\ddot{\uv}_{\INTP} -\ddot{\uv}_{\hh},\vvh) 
    + \ms (\ddot{\uv}        -\ddot{\uv}_{\pi},\vvh)
    - \msh(\ddot{\uv}_{\INTP} -\ddot{\uv}_{\pi},\vvh) 
  \end{align*}
  by adding and subtracting $\ddot{\uv}_{\INTP}$ and
  $\ddot{\uv}_{\pi}$ to the arguments of $\ms(\cdot,\cdot)$ and
  $\msh(\cdot,\cdot)$ and noting that
  $\ms(\ddot{\uv}_{\pi},\vvh)=\msh(\ddot{\uv}_{\pi},\vvh)$ for all
  $\vvh\in\Vvhk$.
  We also rewrite $\TERM{T}{2}$ as
  \begin{align*}
    \TERM{T}{2} 
    = \ash(\uvI-\uvh,\vvh) 
    + \as (\uv -\uv_{\pi},\vvh)
    - \ash(\uvI-\uv_{\pi},\vvh) 
  \end{align*}
  by adding and subtracting $\uvI$ and $\uv_{\pi}$ to the arguments of
  $\as(\cdot,\cdot)$ and $\ash(\cdot,\cdot)$ and noting that
  $\as(\uv_{\pi},\vvh)=\ash(\uv_{\pi},\vvh)$ for all $\vvh\in\Vvhk$.
  Let $\evh=\uvI-\uvh$.
  It holds that $\evh(0)=\dot{\ev}_{\hh}(0)=0$ since
  $\uvh(0)=\big(\uv_0(0)\big)_{\INTP}=\uvI(0)$ and
  $\dot{\uv}_{\hh}(0)=\big(\uv_1(0)\big)_{\INTP}=\dot{\uv}_{\INTP}(0)$.
  Then, using the definition of $\evh$, we reconsider the error
  equation
  \begin{align}
    \TERM{T}{1} + \TERM{T}{2} 
    &= \msh(\ddot{\ev}_{\hh},\vvh) + \ash(\evh,\vvh) 
    +  \ms (\ddot{\uv}        -\ddot{\uv}_{\pi},\vvh)
    -  \msh(\ddot{\uv}_{\INTP} -\ddot{\uv}_{\pi},\vvh)\nonumber\\[0.5em] 
    &+ \as (\uv -\uv_{\pi},\vvh)
    -  \ash(\uvI-\uv_{\pi},\vvh)
    = \Fs(\vvh) - \Fsh(\vvh).
    \label{eq:proof:energy:10}
  \end{align}
  Assume that $\vvh\neq0$ and consider the inequalities:
  \begin{align}
    \ABS{ \Fs(\vvh) - \Fsh(\vvh) }
    &= \frac{ \ABS{\Fs(\vvh) - \Fsh(\vvh)} }{ \snorm{\vvh}{1} }\,\snorm{\vvh}{1}
    \leq \left(\sup_{\vvh\in\Vvhk\backslash{\{0\}}} \frac{ \ABS{\Fs(\vvh) - \Fsh(\vvh)} }{ \snorm{\vvh}{1} }\right)\,\snorm{\vvh}{1}. 
    \label{eq:proof:energy:10:a}
  \end{align}
  Note that it holds:
  \begin{align}
    \msh(\ddot{\ev}_{\hh},\dot{\ev}_{\hh}) + \ash(\evh,\dot{\ev}_{\hh})=\frac{1}{2}\frac{d}{\dt}\big( \msh(\dot{\ev}_{\hh},\dot{\ev}_{\hh}) + \ash(\evh,\evh) \big),
    \label{eq:proof:energy:15:a}\\[0.5em]
    \ABS{ \Fs(\dot{\ev}_{\hh}) - \Fsh(\dot{\ev}_{\hh}) } \leq \left(\sup_{\vvh\in\Vvhk\backslash{\{0\}}}\frac{ \ABS{\Fs(\vvh) - \Fsh(\vvh)} }{ \snorm{\vvh}{1} }\right)\,\TNORM{\evh}{}. 
    \label{eq:proof:energy:15:b}
  \end{align}
  Setting $\vvh=\dot{\ev}_{\hh}(t)$ on the left-hand side
  of~\eqref{eq:proof:energy:10} and
  employing~\eqref{eq:proof:energy:15:a}-\eqref{eq:proof:energy:15:b}
  together with~\eqref{eq:proof:energy:05}, we obtain, after
  rearranging the terms, that:
  \begin{align}
    &\frac{1}{2}\frac{d}{\dt}\big( \msh(\dot{\ev}_{\hh},\dot{\ev}_{\hh}) + \ash(\evh,\evh) \big)
    \leq 
    {}-\ms (\ddot{\uv}        -\ddot{\uv}_{\pi},\dot{\ev}_{\hh})
    + \msh(\ddot{\uv}_{\INTP} -\ddot{\uv}_{\pi},\dot{\ev}_{\hh}) 
    \nonumber\\[0.5em]
    &\qquad
    - \as (\uv -\uv_{\pi},\dot{\ev}_{\hh})
    + \ash(\uvI-\uv_{\pi},\dot{\ev}_{\hh})
    \nonumber\\[0.5em]
    &\qquad
    + \left(\sup_{\vvh\in\Vvhk\backslash{\{0\}}} \frac{ \ABS{ \Fs(\vvh) - \Fsh(\vvh) } }{ \snorm{\vvh}{1} }\right)\,\TNORM{\evh}{}. 
    \label{eq:proof:energy:25}
  \end{align}
  To ease the notation, we denote the last term above
  by $\TERM{R}{1}(t)$.
  We integrate in time from $0$ to $t$ both sides
  of~\eqref{eq:proof:energy:25}, note that the initial term is zero
  since $\evh(0)=\dot{\ev}_{\hh}(0)=0$ and
  use~\eqref{eq:energy:norm:equivalence}
  \begin{align}
    \TNORM{\evh(t)}{}^2
    &\leq \msh\big(\dot{\ev}_{\hh}(t),\dot{\ev}_{\hh}(t)) + \ash(\evh(t),\evh(t)\big) \big)
    \nonumber\\[0.5em]
    &\leq \int_{0}^{t}\Big(
    \TERM{R}{1}(\tau)
    - \ms \big(\ddot{\uv}(\tau)        -\ddot{\uv}_{\pi}(\tau),\dot{\ev}_{\hh}(\tau)\big)
    + \msh\big(\ddot{\uv}_{\INTP}(\tau) -\ddot{\uv}_{\pi}(\tau),\dot{\ev}_{\hh}(\tau)\big) 
    \nonumber\\[0.5em]
    &\phantom{ \leq \int_{0}^{t}\Big( }
    - \as \big(\uv (\tau)-\uv_{\pi}(\tau),\dot{\ev}_{\hh}(\tau)\big)
    + \ash\big(\uvI(\tau)-\uv_{\pi}(\tau),\dot{\ev}_{\hh}(\tau)\big)
    \Big)\,d\tau.
    \label{eq:proof:energy:30}
  \end{align}
  Then, we integrate by parts the integral that contains
  $\as(\cdot,\cdot)$ and $\ash(\cdot,\cdot)$, and again use the fact
  that $\evh(0)=\dot{\ev}_{\hh}(0)=0$, to obtain
  \begin{align}
    \TNORM{\evh(t)}{}^2
    &\leq \int_{0}^{t}\Big(
    \TERM{R}{1}(\tau)
    + \Big[ 
    {}-\ms \big(\ddot{\uv}(\tau)        -\ddot{\uv}_{\pi}(\tau),\dot{\ev}_{\hh}(\tau) \big)
    + \msh\big(\ddot{\uv}_{\INTP}(\tau) -\ddot{\uv}_{\pi}(\tau),\dot{\ev}_{\hh}(\tau) \big)\Big]
    \nonumber\\[0.5em]
    &\phantom{ \leq \int_{0}^{t}\Big( }
    + \Big[ 
    \as \big( \dot{\uv}(\tau)       -\dot{\uv}_{\pi}(\tau),  \evh(\tau) \big)
    -\ash\big( \dot{\uv}_{\INTP}(\tau)-\dot{\uv}_{\pi}(\tau),\evh(\tau) \big) \Big]
    \Big)\,d\tau
    \nonumber\\[0.5em]
    &\phantom{ \leq \int_{0}^{t}\Big( }
    + \Big[ 
    {}-\as \big( \uv(t) -\uv_{\pi}(t), \evh(t) \big)
    + \ash\big( \uvI(t)-\uv_{\pi}(t), \evh(t) \big) \Big]
    \nonumber\\[0.5em]
    &= \int_{0}^{t}\Big(
    {}\TERM{R}{1}(\tau)
    + \TERM{R}{2}(\tau)
    + \TERM{R}{3}(\tau)
    \Big)\,d\tau
    + \TERM{R}{4}(t),
    \label{eq:proof:energy:35}
  \end{align}
  where terms $\TERM{R}{\ell}$, $\ell=2,3,4$, match with the squared
  parenthesis.
  For the next development, we do not need an upper bound of term
  $\TERM{R}{1}$.
  Instead, we have to bound the other three terms in the right-hand
  side of~\eqref{eq:proof:energy:35}.
  To bound $\TERM{R}{2}$ we use the continuity of $\ms(\cdot,\cdot)$
  and $\msh(\cdot,\cdot)$ and the definition of the energy norm
  $\TNORM{\,\cdot\,}{}$ given in~\eqref{eq:three-bar-norm}:
  \begin{align}
    \ABS{ \TERM{R}{2} } 
    & \leq
    \ABS{ \ms (\ddot{\uv}        -\ddot{\uv}_{\pi}, \dot{\ev}_{\hh} ) } +
    \ABS{ \msh(\ddot{\uv}_{\INTP} -\ddot{\uv}_{\pi}, \dot{\ev}_{\hh} ) }
    \lesssim \big( \NORM{\ddot{\uv}-\ddot{\uv}_{\pi}}{0} + \NORM{\ddot{\uv}_{\INTP}-\ddot{\uv}_{\pi}}{0} \big)\,\NORM{ \dot{\ev}_{\hh} }{0}
    \nonumber\\[0.5em]
    &\lesssim \big( \NORM{\ddot{\uv}-\ddot{\uv}_{\pi}}{0} + \NORM{\ddot{\uv}_{\INTP}-\ddot{\uv}_{\pi}}{0} \big)\,\TNORM{ \evh }{}.
    \label{eq:proof:energy:40}
  \end{align}
  Similarly, to bound $\TERM{R}{3}$ we use the continuity of
  $\as(\cdot,\cdot)$ and $\ash(\cdot,\cdot)$ and the definition of the
  energy norm $\TNORM{\,\cdot\,}{}$ given
  in~\eqref{eq:three-bar-norm}:
  \begin{align}
    \ABS{ \TERM{R}{3} } 
    & \leq
    \ABS{ \as (\dot{\uv}        -\dot{\uv}_{\pi}, \evh ) } +
    \ABS{ \ash(\dot{\uv}_{\INTP} -\dot{\uv}_{\pi}, \evh ) }
    \lesssim \big( \snorm{\dot{\uv} -\dot{\uv}_{\pi}}{1,\hh} + \snorm{\dot{\uv}_{\INTP}-\dot{\uv}_{\pi}}{1,\hh} \big)\,\snorm{ \evh }{1}
    \nonumber\\[0.5em]
    &\lesssim \big( \snorm{\dot{\uv}-\dot{\uv}_{\pi}}{1,\hh} + \snorm{\dot{\uv}_{\INTP}-\dot{\uv}_{\pi}}{1,\hh} \big)\,\TNORM{ \evh }{}.
    \label{eq:proof:energy:45}
  \end{align}
  Finally, to bound $\TERM{R}{4}$ we first use the continuity of
  $\as(\cdot,\cdot)$ and $\ash(\cdot,\cdot)$, and the right inequality
  in~\eqref{eq:energy:norm:equivalence}; then, we apply the Young
  inequality, so that
  \begin{align}
    \ABS{ \TERM{R}{4} } 
    & \leq
    {}\ABS{ \as ( \uv-\uv_{\pi}, \evh ) }
    + \ABS{ \ash( \uvI-\uv_{\pi}, \evh ) }
    \lesssim \big( \snorm{\uv-\uv_{\pi}}{1,\hh} + \snorm{\uvI-\uv_{\pi}}{1,\hh} \big)\,\snorm{\evh}{1}
    \nonumber\\[0.5em]
    &\lesssim \big( \snorm{\uv-\uv_{\pi}}{1,\hh} + \snorm{\uvI-\uv_{\pi}}{1,\hh} \big)\,\TNORM{\evh}{}
    \lesssim \frac{1}{2\epsilon}\big( \snorm{\uv-\uv_{\pi}}{1,\hh} + \snorm{\uvI-\uv_{\pi}}{1,\hh} \big)^2 
    + \frac{\epsilon}{2}\TNORM{\evh}{}^2.
    \label{eq:proof:energy:50}
  \end{align}
  Using bounds~\eqref{eq:proof:energy:40}, \eqref{eq:proof:energy:45},
  and~\eqref{eq:proof:energy:50} in \eqref{eq:proof:energy:35}, we
  find the inequality
  \begin{align*}
    \TNORM{\evh(t)}{}^2
    \lesssim \widetilde{\Gs}_0^2(t)
    +\int_{0}^{t}\Gs_1(\tau)\TNORM{\evh(\tau)}{}\,d\tau
    \lesssim \left(\sup_{\tau\in[0,T]}\Gs_0(\tau)\right)^2
    +\int_{0}^{t}\Gs_1(\tau)\TNORM{\evh(\tau)}{}\,d\tau,
  \end{align*}
  where 
  $\widetilde{\Gs}_0^2(t)=\big(\snorm{\uv(\tau)-\uvI(\tau)}{1,\hh}+\snorm{\uv(\tau)-\uv_{\pi}(\tau)}{1,\hh}\big)^2$,
  and $\Gs_0(t)$ and $\Gs_1(t)$ are the time-dependent functions
  defined
  in~\eqref{eq:semi-discrete:convergence:G0}-\eqref{eq:semi-discrete:convergence:G1}.
  Again, an application of Lemma~A5, p.~157 of Reference\cite{Brezis:1973}
  yields
  \begin{align*}
    \TNORM{\evh(t)}{}\lesssim\sup_{\tau\in[0,T]}\Gs_0(\tau)+\int_{0}^{t}\Gs_1(\tau)d\tau.
  \end{align*}
  The theorem follows on using the triangular inequality
  $$\TNORM{\uv(t)-\uvh(t)}{}\leq\TNORM{\uv(t)-\uvI(t)}{}+\TNORM{\uvI(t)-\uvh(t)}{}$$
  and noting that $\TNORM{\uv(t)-\uvI(t)}{}$ is absorbed in
  $\sup_{\tau\in[0,T]}\Gs_0(\tau)$.
\end{proof}


\medskip
\begin{corollary}
  Under the conditions of
  Theorem~\ref{theorem:semi-discrete:convergence}, for
  $\fv\in\LTWO\big((0,T);\big[\HS{m-1}(\Omega)\big]^2\big)$ we have
  that
  \begin{align}
    &\sup_{0<t\leq T}\TNORM{\uv(t)-\uvh(t)}{}
    \lesssim
    \frac{\hh^\mu}{k^m}\sup_{0<t\leq T}\Big(\NORM{\dot{\uv}(t)}{m+1} + \NORM{\uv(t)}{m+1}\Big)\nonumber\\[0.5em]
    &\qquad+\int_{0}^{T}\left(
      \frac{\hh^{\mu+1}}{k^{m}}\left( \NORM{ \ddot{\uv}(\tau) }{m+1} + \NORM{ \dot{\uv}(\tau) }{m+1}\right) 
      + \frac{\hh^{\mu}}{k^{m}}\big ( \NORM{ \ddot{\uv}(\tau) }{m+1} + \NORM{ \dot{\uv}(\tau) }{m+1} \big) 
    \right)\,d\tau\nonumber\\[0.5em]
    &\qquad+\int_{0}^{T}\hh\Norm{ \big(I-\Pi^0_{k-2}\big)\fv(\tau) }{0}\,d\tau,
    \label{eq:corollary:estimate}
  \end{align}
  where $\mu=\min(k,m)$.
  The hidden constant in ``$\lesssim$`` is independent of $\hh$, but
  may depend on the model parameters and approximation constants, 
  \RED{the polynomial degree $k$}, and the final observation time $T$.
\end{corollary}
\begin{proof}
  Note that 
  \begin{align}
    \Fsh(\vvh) 
    = \int_{\Omega}\fv\cdot\Piz{k-2}\vvh\,\dV
    = \int_{\Omega}\Piz{k-2}\fv\cdot\vvh\,\dV.
  \end{align}
  For all $\vvh\in\Vvhk$ it holds that
  \begin{align}
    \ABS{ \Fs(\vvh)-\Fsh(\vvh) }
    &= \left\vert \int_{\Omega}(I-\Piz{k-2})\fv\cdot\vvh\,\dV \right\vert
    = \left\vert \int_{\Omega}(I-\Piz{k-2})\fv\cdot(I-\Piz{0})\vvh\,\dV \right\vert
    \nonumber\\[0.5em]
    &= \norm{(I-\Piz{k-2})\fv}{0}\,\norm{(I-\Piz{0})\vvh}{0}
    \lesssim\hh\norm{(I-\Piz{k-2})\fv}{0}\,\snorm{\vvh}{1}.
  \end{align}
  Hence,
  \begin{align}
    \sup_{\vvh\in\Vvhk\backslash{\{0\}}} \frac{ \ABS{ \Fs(\vvh) - \Fsh(\vvh) } }{ \snorm{\vvh}{1} }
    \leq \hh\norm{(I-\Piz{k-2})\fv}{0}.
  \end{align}
  Estimate~\eqref{eq:corollary:estimate} follows on applying this
  inequality and the results of Lemmas~\ref{lemma:hp:convergence:1}
  and \ref{lemma:hp:convergence:2}
  to~\eqref{eq:semi-discrete:convergence:G0}
  and~\eqref{eq:semi-discrete:convergence:G1}, using the resulting
  estimates in~\eqref{eq:semi-discrete:convergence}, and taking the
  supremum on the time interval $[0,T]$.
\end{proof}

\subsection{Convergence analysis in the $\LTWO$ norm}
\label{subsec:convergence:L2:norm}
The main result of this section is the following theorem that proves
the $\LTWO$-convergence of the conforming VEM.
The strategy we use in the proof is inspired by
Reference~\cite{Baker-Dougalis-Karakashian:1980}, using also the
substantial modifications required to set it up in the virtual element
framework of Reference~\cite{Adak-Natarajan:2019}.

\medskip
\begin{theorem}
  \label{theorem:L2:convergence}
  Let $\uv$ be the exact solution of problem~\eqref{eq:weak:A} under
  the assumption that domain $\Omega$ is $\HTWO$-regular and
  $\uvh\in\Vvhk$ the solution of the virtual element method stated
  in~\eqref{eq:VEM:semi-discrete} under the mesh assumptions of
  Section~\ref{subsec:vem:meshes}.
  If
  $\uv,\dot{\uv},\ddot{\uv}\in\LTWO\big(0,T;\big[\HS{m+1}(\Omega)\cap\HONEzr(\Omega)\big]^2\big)$,
  with integer $m\geq0$, then the following estimate holds for almost
  every $t\in[0,T]$ by setting $\mu=\min(m,k)$:
  \begin{align}
    \NORM{\uv(t)-\uvh(t)}{0} 
    \lesssim& \NORM{\uvh(0)-\uv_0}{0} + \NORM{ \dot{\uv}_{\hh}(0)-\uv_{1} }{0} 
    + \frac{\hh^{\mu+1}}{ \RED{k^{m+1}} }\Big(
    \NORM{\ddot{\uv}}{\LTWO( 0,T; [\HS{m+1}(\Omega)]^2 )} 
    \nonumber\\ & 
    +
    \NORM{\dot {\uv}}{\LTWO( 0,T; [\HS{m+1}(\Omega)]^2 )} + 
    \NORM{      \uv }{\LTWO( 0,T; [\HS{m+1}(\Omega)]^2 )} 
    \Big) + 
    \int_{0}^{T}\Norm{\big(1-\Piz{k-2}\big)\fv(\tau)}{0}^2d\tau.
  \end{align}
  The hidden constant in ``$\lesssim$`` is independent of $\hh$, but
  may depend on the model parameters and approximation constants
  $\varrho$, $\mu^*$, and \RED{the polynomial degree $k$},
  and the final observation time $T$.
\end{theorem}

\medskip
To prove this theorem, we need the energy projection operator
$\Ph:\big[\HONEzr(\Omega)\big]^2\to\Vvhk$, which is such that
$\Ph\uv\in\Vvhk$, for every $\uv\in\big[\HONEzr(\Omega)\big]^2$, is
the solution of the variational problem:
\begin{align}
  \ash(\Ph\uv,\vvh) = \as(\uv,\vvh)
  \quad\forall\vvh\in\Vvhk.
  \label{eq:Ph:def}
\end{align}
The energy projection $\Ph\uv$ is a virtual element approximation of
the exact solution $\uv$, and the accuracy of the approximation is
characterized by the following lemma.

\medskip
\begin{lemma}
  \label{lemma:Ph:error:estimate}
  Let $\uv\in\big[\HS{m+1}(\Omega)\cap\HONEzr(\Omega)]^2$ be the
  solution of problem~\eqref{eq:weak:A} under the mesh assumptions of
  Section~\ref{subsec:vem:meshes}.
  Then, there exists a unique function $\Ph\uv\in\Vvhk$ such that
  \begin{align}
    \snorm{\uv-\Ph\uv}{\RED{1}}\lesssim\frac{\hh^{\mu}}{k^m}\NORM{\uv}{m+1}
    \label{eq:Ph:H1:estimate}
  \end{align}
  with $\mu=\min(k,m)$ and $m\geq1$.
  Moreover, if domain $\Omega$ is $\HTWO$-regular, it holds that
  \begin{align}
    \norm{\uv-\Ph\uv}{0}\lesssim\frac{\hh^{\mu+1}}{ \RED{k^{m+1}} }\NORM{\uv}{m+1}.
    \label{eq:Ph:L2:estimate}
  \end{align}
\end{lemma}
\ifJournal
\else
\begin{proof}
  The argument we use in this proof is similar to that used in the
  proof of Lemma 3.1 in Reference~\cite{Vacca-BeiraodaVeiga:2015} for
  the conforming virtual element approximation of a scalar parabolic
  problem.
  First, we note that the bilinear form $\ash(\cdot,\cdot)$ is
  continuous and coercive on $\Vvhk\times\Vvhk$, the linear functional
  $\as(\uv,\cdot)$ is continuous on $\Vvhk$, and the Lax-Milgram Lemma
  implies that the solution $\Ph\uv$ to problem~\eqref{eq:Ph:def}
  exists and is unique.
  Then, to prove estimate~\eqref{eq:Ph:H1:estimate}, we introduce the
  virtual element interpolate $\uvI$ of $\uv$, which is the function
  in $\Vvhk$ that has the same degrees of freedom of $\uv$ and
  satisfies the error inequality given in
  Lemma~\ref{lemma:hp:convergence:2}.
  A straightforward application of the triangular inequality yields:
  \begin{align}
    \snorm{\uv-\Ph\uv}{1} \leq \snorm{\uv-\uvI}{1} + \snorm{\uvI-\Ph\uv}{1},
    \label{eq:Ph:proof:00}
  \end{align}
  We estimate the first term on the right by using the second inequality
  in~\eqref{eq:hp:convergence:2}.
  Instead, to estimate the second term we need the following
  developments.
  Let $\delta_{\hh}=\Ph\uv-\uvI$ and $\uvp$ any piecewise polynomial
  approximation of degree (at most) $k$ that satisfies
  inequality~\eqref{eq:hp:convergence:1} on each element $\P$.
  Using the $k$-consistency property, stability, and the continuity of
  the bilinear forms $\ash$ and $\as$ yield the development chain:
  \begin{align*}
    \alpha_*\snorm{\delta_{\hh}}{{\color{red}1}}^2
    &= \alpha_*\as(\delta_{\hh},\delta_{\hh})
    \leq \ash(\delta_{\hh},\delta_{\hh})
    = \ash(\Ph\uv,\delta_{\hh}) - \ash(\uvI,\delta_{\hh})
    = \as (\uv,\delta_{\hh}) - \ash(\uvI,\delta_{\hh})\\[0.5em]
    &= \sum_{\P\in\Th}\Big( \asP(\uv,\delta_{\hh}) - \ashP(\uvI,\delta_{\hh}) \Big)\\
    &= \sum_{\P\in\Th}\Big( \asP(\uv-\uvp,\delta_{\hh}) - \ashP(\uvI-\uvp,\delta_{\hh}) \Big)\\
    &\leq \sum_{\P\in\Th}\Big( \snorm{\uv-\uvp}{1,\P} + \alpha^*\snorm{\uvI-\uvp}{1,\P} \Big)\snorm{\delta_{\hh}}{1,\P}\\
    &\lesssim \max(1,\alpha^*)\Big( \snorm{\uv-\uvp}{1,\hh} + \snorm{\uvI-\uvp}{1,\hh} \Big)\snorm{\delta_{\hh}}{{\color{red}1}}.
  \end{align*}
  Therefore, it holds that
  \begin{align*}
    \snorm{\Ph\uv-\uvI}{{\color{red}1}} 
    \lesssim\Big(  \snorm{\uv-\uvp}{1,\hh} + \snorm{\uvI-\uvp}{1,\hh} \Big)
    \lesssim \Big( 2\snorm{\uv-\uvp}{1,\hh} + \snorm{\uvI-\uv }{1}     \Big).
  \end{align*}
  Inequality~\eqref{eq:Ph:H1:estimate} follows on substituting this
  relation in~\eqref{eq:Ph:proof:00} and
  using~\eqref{eq:hp:convergence:1}.
  
  \medskip
  To derive the estimate in the $\LTWO$-norm, we consider the weak
  solution
  $\psiv\in\big[\HTWO(\Omega)\big]^2\cap\big[\HONEgm(\Omega)\big]^2$
  to the auxiliary elliptic equation:
  \begin{align*}
    -\nabla\cdot\bsig(\psiv) &= \uv-\Ph\uv\phantom{\mathbf{0}} \quad\textrm{in~}\Omega,\\
    \psiv                    &= \mathbf{0}\phantom{\uv-\Ph\uv} \quad\textrm{on~}\Gamma,
  \end{align*}
  where $\bsig(\psiv)=\nabla\psiv$.
  As $\Omega$ is an $\HTWO$-regular domain, solution $\psiv$ satisfies
  the following stability result:
  \begin{align}
    \norm{\psiv}{2}\leq\mathcal{C}\norm{\uv-\Ph\uv}{0}.
  \end{align}
  Let $\psivI\in\Vvhk$ be the virtual element interpolant of $\psiv$
  that satisfies the interpolation error estimate given in
  Lemma~\ref{lemma:hp:convergence:2} (with $m=1$).
  We integrate by parts and use the definition of the energy
  projection $\Ph$:
  \begin{align}
    \norm{\uv-\Ph\uv}{0}^2 
    &= \scal{\uv-\Ph\uv,\uv-\Ph\uv}
    = \scal{\uv-\Ph\uv,-\nabla\cdot\bsig(\psiv)}
    = \as(\uv-\Ph\uv,\psiv)
    \nonumber\\[-0.25em]
    &= \as(\uv-\Ph\uv,\psiv-\psivI) + \as(\uv-\Ph\uv,\psivI)
    \nonumber\\
    &= \TERM{T}{1} + \TERM{T}{2}.
  \end{align}
  The proof continues by estimating each term $\TERM{T}{i}$, $i=1,2$,
  separately.
  The first term is bounded as follows:
  \begin{align*}
    \abs{\TERM{T}{1}} 
    &= \abs{\as(\uv-\Ph\uv,\psiv-\psivI)}
    \leq \snorm{\uv-\Ph\uv}{{\color{red}1}}\,\snorm{\psiv-\psivI}{{\color{red}1}}
    \lesssim \frac{\hh^{\mu}}{k^{m}}\NORM{\uv}{m+1}\,\frac{\hh}{k}
    \norm{\psiv}{2}\\
    &\lesssim\frac{\hh^{\mu+1}}{k^{m+1}}\NORM{\uv}{m+1}\norm{\uv-\Ph\uv}{0}
  \end{align*}
  where we used the estimate in the energy
  norm~\eqref{eq:Ph:H1:estimate} derived previously and interpolation
  error estimate~\eqref{eq:hp:convergence:2}.
  For the second term, first we use the consistency and stability
  property to transform $\TERM{T}{2}$ as follows:
  \begin{align*}
    \TERM{T}{2}
    &= \as(\uv,\psivI) - \as(\Ph\uv,\psivI)
    = \sum_{\P\in\Th}\left( \ashP(\Ph\uv,\psivI) - \asP(\Ph\uv,\psivI) \right)\\
    &= \sum_{\P\in\Th}\left( \ashP(\Ph\uv-\uvp,\psivI-\Piz{1}\psiv) - \asP(\Ph\uv-\uvp,\psivI-\Piz{1}\psiv) \right).
  \end{align*}
  Then, we add and subtract $\uv$ and $\psiv$ and use
  estimates~\eqref{eq:hp:convergence:1}
  and~\eqref{eq:hp:convergence:2} to obtain
  \begin{align*}
    \abs{\TERM{T}{2}}
    &\leq \max(1,\alpha^*)\sum_{\P\in\Th}\snorm{\Ph\uv-\uvp}{1,\P}\snorm{\psivI-\Piz{1}\psiv}{1,\P}\\
    &\leq \max(1,\alpha^*)\sum_{\P\in\Th}\left(\snorm{\Ph\uv-\uv}{1,\P}+\snorm{\uv-\uvp}{1,\P}\right)\left(\snorm{\psivI-\psiv}{1,\P}+\snorm{\psiv-\Piz{1}\psiv}{1,\P}\right)\\
    &\lesssim \sum_{\P\in\Th}\frac{\hP^{\mu}}{k^{m}}\snorm{\uv}{m+1,\P}\,\frac{\hP}{k}\snorm{\psiv}{2,\P}
    \lesssim \frac{\hh^{\mu+1}}{k^{m+1}}\snorm{\uv}{m+1}\norm{\psiv}{2}\\
    &\lesssim \frac{\hh^{\mu+1}}{k^{m+1}}\snorm{\uv}{m+1}\norm{\uv-\Ph\uv}{0},
  \end{align*}
  and the bound of $\TERM{T}{2}$ is derived by using in the final step
  the $\HTWO$-regularity of $\psiv$.
  The estimate in the $\LTWO$-norm is finally proved by collecting the
  estimates of the two terms $\TERM{T}{i}$, $i=1,2$.  \ENDPROOF
\end{proof}

  
\fi

\medskip
\textbf{\textit{Proof of Theorem~\ref{theorem:L2:convergence}.}}
We use the energy projection to split the approximation error as
follows: $\uv(t)-\uvh(t)=\rhov(t)-\etav(t)$, with
$\rhov(t)=\uv(t)-\Ph\uv(t)$ and $\etav(t)=\uvh(t)-\Ph\uv(t)$.
Since $\Vvhk$ is a subspace of $\big[\HONEzr(\Omega)\big]^2$, to
derive the error equation, we test~\eqref{eq:weak:A}
and~\eqref{eq:VEM:semi-discrete} against $\vvh\in\Vvhk$
\begin{align*}
  &\ms (\ddot{\uv},  \vvh) + \as (\uv, \vvh) = \Fs (\vvh),\\
  &\msh(\ddot{\uv}_h,\vvh) + \ash(\uvh,\vvh) = \Fsh(\vvh),
\end{align*}
and take the difference 
\begin{align}
  \msh(\ddot{\uv}_h,\vvh)+\ash(\uvh,\vvh)
  -\big(\ms(\ddot{\uv},\vvh)+\as(\uv,\vvh)\big)
  =\Fsh(\vvh)-\Fs(\vvh).
  \label{eq:L2estimate:proof:100}
\end{align}
We add and subtract $\Ph\ddot{\uv}$ and $\Ph\uv$ in the virtual
element bilinear forms $\msh$ and $\ash$ and rearrange the terms
containing $\ddot{\uv}$ and $\uv$ to the right-hand side to obtain:
\begin{align}
  \msh(\ddot{\uv}_h-\Ph\ddot{\uv},\vvh)
  +\ash(\uvh-\Ph\uv,\vvh)
  &=\Fsh(\vvh)-\Fs(\vvh)
  +\ms(\ddot{\uv},\vvh)+\as(\uv,\vvh)
  \nonumber\\ & \phantom{=}
  -\msh(\Ph\ddot{\uv},\vvh)
  -\ash(\Ph\uv,\vvh).
  \label{eq:L2estimate:proof:110}
\end{align}
Since~\eqref{eq:Ph:def} implies that
$\as(\uv,\vvh)-\ash(\Ph\uv,\vvh)=0$, and using the notation
$\ddot{\etav}=\ddot{\uv}_{\hh}-\Ph\ddot{\uv}$ and $\etav=\uvh-\Ph\uv$,
we obtain
\begin{align}
  \msh(\ddot{\etav},\vvh) +\ash(\etav,\vvh)
  =\Fsh(\vvh)-\Fs(\vvh)
  +\ms(\ddot{\uv},\vvh)
  -\msh(\Ph\ddot{\uv},\vvh).
  \label{eq:L2estimate:proof:120}
\end{align}
Hereafter in this proof, we will assume that
\begin{align}
  \vvh(t) = \int^{\xi}_{t}\etav(\tau)d\tau
  \label{eq:L2estimate:proof:130}
\end{align}
for every $t$ and $\xi\in[0,T]$.
The function $\vvh(t)$ given by~\eqref{eq:L2estimate:proof:130}
obviously belongs to the virtual element space $\Vvhk$ as it is a
linear superposition of virtual element functions in such a space and,
thus, can be used as a test function.
Since now $\vvh$ depends on time $t$, we are allowed to consider its
time derivatives.
In particular, we observe that the straightforward calculation
\begin{align*}
  \frac{d}{\dt}\left(\vvh\frac{d}{\dt}(\uv-\uvh)\right) = 
  \frac{d}{\dt}\left(\vvh\frac{d}{\dt}(\rhov-\etav)\right) = 
  \frac{d\vvh}{\dt}\frac{d\rhov}{\dt}-\frac{d\vvh}{\dt}\frac{d\etav}{\dt} +
  \vvh\frac{d^2\rhov}{\dt^2}-\vvh\frac{d^2\etav}{\dt^2}
\end{align*}
implies the identity
\begin{align}
  \msh(\ddot{\etav},\vvh) =
  -\msh(\dot{\etav},\dot{\vv}_{\hh})
  -\frac{d}{\dt}\msh(\dot{\uv}-\dot{\uv}_{\hh},\vvh)
  +\msh(\ddot{\rhov},\vvh)
  +\msh(\dot{\rhov},\dot{\vv}_{\hh}).
  \label{eq:L2estimate:proof:140}
\end{align}
Therefore, using~\eqref{eq:L2estimate:proof:140}
in~\eqref{eq:L2estimate:proof:120} yields
\begin{align}
  -\msh(\dot{\etav},\dot{\vv}_{\hh})
  +\ash(\etav,\vvh)
  &=
  \Fsh(\vvh)-\Fs(\vvh)
  +\ms(\ddot{\uv},\vvh)
  -\msh(\Ph\ddot{\uv},\vvh)
  \nonumber\\[0.5em]
  &\phantom{=}
  +\frac{d}{\dt}\msh(\dot{\uv}-\dot{\uv}_{\hh},\vvh)
  -\msh(\ddot{\rhov},\vvh)
  -\msh(\dot{\rhov},\dot{\vv}_{\hh}).
  \label{eq:L2estimate:proof:150}
\end{align}
We rewrite the approximation error on the source term on the
right-hand side of~\eqref{eq:L2estimate:proof:150} as follows:
\begin{align}
  \Fsh(\vvh(t))-\Fs(\vvh(t)) = \ms\big( \fvh(t)-\fv(t), \vvh(t) \big),
\end{align}
where $\fvh(t)=\Piz{k-2}\fv$ according to~~\eqref{eq:Fsh:def}.
Consider the integral quantities\cite{Adak-Natarajan:2019}
\begin{align}
  \calA_1(t) = \int_{0}^{t}\big(\,\fvh(\tau)-\fv(\tau)\,\big)d\tau,\qquad
  \calA_2(t) = \int_{0}^{t}\ddot{\uv}(\tau)d\tau,\qquad
  \calA_3(t) = \int_{0}^{t}\Ph\ddot{\uv}(\tau)d\tau,
  \label{eq:L2estimate:proof:160}
\end{align}
and note that 
\begin{align}
  &\Fsh(\vvh(t))-\Fs(\vvh(t)) = \frac{d}{\dt}\ms (\calA_1,\vvh) - \ms (\calA_1,\dot{\vv}_{\hh}),   \label{eq:L2estimate:proof:170:a}\\[0.5em]
  &\ms(\ddot{\uv},\vvh)       = \frac{d}{\dt}\ms (\calA_2,\vvh) - \ms (\calA_2,\dot{\vv}_{\hh}),   \label{eq:L2estimate:proof:170:b}\\[0.5em]
  &\msh(\Ph\ddot{\uv},\vvh)   = \frac{d}{\dt}\msh(\calA_3,\vvh) - \msh(\calA_3,\dot{\vv}_{\hh}).   \label{eq:L2estimate:proof:170:c}
\end{align}
Hence,
using~\eqref{eq:L2estimate:proof:170:a},~\eqref{eq:L2estimate:proof:170:b}
and~\eqref{eq:L2estimate:proof:170:c}
in~\eqref{eq:L2estimate:proof:150} yields
\begin{align}
  &-\msh(\dot{\etav},\dot{\vv}_{\hh})
  +\ash(\etav,\vvh)
  =
  \sum_{i=1}^{2}\left[ \frac{d}{\dt}\ms (\calA_i,\vvh) - \ms (\calA_i,\dot{\vv}_{\hh}) \right]
  - \left[ \frac{d}{\dt}\msh(\calA_3,\vvh) - \msh(\calA_3,\dot{\vv}_{\hh}) \right]
  \nonumber\\[0.5em] & \qquad\qquad\qquad
  +\frac{d}{\dt}\msh(\dot{\uv}-\dot{\uv}_{\hh},\vvh)
  -\msh(\ddot{\rhov},\vvh)
  -\msh(\dot{\rhov},\dot{\vv}_{\hh}).
  \label{eq:L2estimate:proof:180}
\end{align}
To prove the assertion of the theorem we integrate both sides of
\eqref{eq:L2estimate:proof:180} with respect to $t$ between $0$ and
$\xi$; then, we estimate a lower bound for the left-hand side (LHS)
and an upper bound for the right-hand side (RHS).

\medskip
To estimate a lower bound for the left-hand side
of~\eqref{eq:L2estimate:proof:180}, we note that
$\dot{\vv}_{\hh}=-\etav$, $\vvh(\xi)=0$ and
$\ash(\vvh(0),\vvh(0))\geq0$ from the coercivity of $\ash$.
Thus, we estimate the left-hand side as follows:
\begin{align}
  &\int_{0}^{\xi}\big[\textrm{LHS of Eq.~\eqref{eq:L2estimate:proof:180}}\big]\dt
  =\int_{0}^{\xi}\Big[\msh(\dot{\etav},\etav) - \ash(\dot{\vv}_{\hh},\vvh)\Big]\dt
  \nonumber\\[0.5em]
  &\qquad\qquad= \int_{0}^{\xi}\frac{1}{2}\frac{d}{\dt}\Big[\msh(\etav,\etav) - \ash(\vvh,\vvh)\Big]\dt
  \nonumber\\[0.75em]
  &\qquad\qquad=\frac{1}{2}\Big[\msh(\etav(\xi),\etav(\xi))-\msh(\etav(0),\etav(0))\Big]
  -\frac{1}{2}\Big[\ash(\vvh (\xi),\vvh (\xi))-\ash(\vvh(0), \vvh(0) )\Big]
  \nonumber\\[0.75em]
  &\qquad\qquad\geq \frac{1}{2}\left( \msh(\etav(\xi),\etav(\xi)) - \msh(\etav(0),\etav(0)) \right).
  \label{eq:L2estimate:proof:200}
\end{align}

\medskip
To estimate an upper bound for the right-hand side
of~\eqref{eq:L2estimate:proof:180}, we note that $\calA_i(0)=0$ for
$i=1,2,3$ and use again the fact $\vvh(\xi)=0$.
So, for $i=1,2$, a direct integration yield
\begin{align}
  \int_{0}^{\xi}\Big[\,\frac{d}{\dt}\ms(\calA_{i},\vvh)-\ms(\calA_{i},\dot{\vv}_{\hh})\,\Big]\dt 
  &= 
  \ms(\calA_{i}(\xi),\vvh(\xi)) - \ms(\calA_{i}(0),\vvh(0)) 
  -\int_{0}^{\xi}\ms(\calA_{i},\dot{\vv}_{\hh})\dt 
  \nonumber\\[0.5em]&= 
  -\int_{0}^{\xi}\ms(\calA_{i},\dot{\vv}_{\hh})\dt
  \label{eq:A12:integrate}\\[-0.5em]
  \intertext{and, similarly,}
  \int_{0}^{\xi}\Big[\,\frac{d}{\dt}\msh(\calA_3,\vvh)-\msh(\calA_3,\dot{\vv}_{\hh})\,\Big]\dt 
  &= 
  \msh(\calA_3(\xi),\vvh(\xi)) - \msh(\calA_3(0),\vvh(0)) 
  -\int_{0}^{\xi}\msh(\calA_3,\dot{\vv}_{\hh})\dt
  \nonumber\\[0.5em]&= 
  -\int_{0}^{\xi}\msh(\calA_3,\dot{\vv}_{\hh})\dt.
  \label{eq:A3:integrate}
\end{align}
Using \eqref{eq:A12:integrate} and \eqref{eq:A3:integrate} in the
right-hand side of~\eqref{eq:L2estimate:proof:180} and rearranging the
terms yield
\begin{align*}
  \int_{0}^{\xi}\big[\textrm{RHS of Eq.~\eqref{eq:L2estimate:proof:180}}\big]\dt
  =& 
  -\int_{0}^{\xi}\ms(\calA_{1},\dot{\vv}_{\hh})\dt 
  + \int_{0}^{\xi} \big(\, \msh(\calA_3,\dot{\vv}_{\hh}) - \ms(\calA_{2},\dot{\vv}_{\hh}) \,\big)\dt
  \nonumber\\[0.5em]&
  +\Big[\,\msh(\dot{\uv}(\xi)-\dot{\uv}_{\hh}(\xi),\vvh(\xi)) - \msh(\dot{\uv}(0)-\dot{\uv}_{\hh}(0),\vvh(0))\,\Big]
  \nonumber\\[0.5em]&
  -\int_{0}^{\xi}\Big(\,\msh(\ddot{\rhov},\vvh)+\msh(\dot{\rhov},\dot{\vv}_{\hh})\,\Big)\dt  
  = \TERM{L}{1}(\xi) + \TERM{L}{2}(\xi) + \TERM{L}{3}(\xi) + \TERM{L}{4}(\xi).
\end{align*}
The proof continues by bounding the four terms $\TERM{L}{i}(\xi)$,
$i=,1,2,3,4$, separately.

\PGRAPH{Estimate of term $\TERM{L}{1}$}
To estimate $\TERM{L}{1}$, we first note that $\ms(\cdot,\cdot)$ is an
inner product, so we can apply the Cauchy-Schwarz and the Young
inequalities to derive the following bound, which holds for any
$t\in[0,\xi]$:
\begin{align}
  \ABS{\ms(\calA_{1}(t),\dot{\vv}_{\hh}(t))}
  \leq \NORM{\calA_1(t)}{0}\,\NORM{\dot{\vv}_{\hh}(t)}{0}
  \leq 
  \frac12\NORM{\calA_1(t)}{0}^2+\frac12\NORM{\dot{\vv}_{\hh}(t)}{0}^2.
  \label{eq:L1:bound:10}
\end{align}
We estimate term $\NORM{\calA_1(t)}{0}$ by applying Jensen's
inequality and Fubini's theorem to exchange the integration order
\begin{align}
  \NORM{\calA_1(t)}{0}^2
  &=   \int_{\Omega}\bigg\vert \int_{0}^{t} \big(\,\fv(\tau)-\Piz{k-2}(\fv(\tau)\,\big)d\tau \bigg\vert^2\!\dV
  \leq \int_{\Omega}\,T\bigg( \int_{0}^{t}\ABS{ \big(I-\Piz{k-2}\big)\fv(\tau) }^2d\tau\bigg) \dV
  \nonumber\\[0.5em]
  &=   T\int_{0}^{t}\bigg(\int_{\Omega}\ABS{ \big(I-\Piz{k-2}\big)\fv(\tau) }^2\dV\bigg)d\tau
  \leq T\int_{0}^{t}\NORM{ \big(I-\Piz{k-2}\big)\fv(\tau) }{0}^2d\tau.
  \nonumber\\[0.5em]
  &\leq T\int_{0}^{T}\NORM{ \big(I-\Piz{k-2}\big)\fv(\tau) }{0}^2d\tau.
  \label{eq:L1:bound:20}
\end{align}
We use \eqref{eq:L1:bound:20} in~\eqref{eq:L1:bound:10} and the
resulting inequality in the definition of $\TERM{L}{1}$; then, we note
that the last integral in~\eqref{eq:L1:bound:20} is on the whole
interval $[0,T]$, and is, thus, independent on $t$.
Since $\xi\leq\Ts$ and by using the Young's inequality, we readily
find that
\begin{align*}
  \ABS{\TERM{L}{1}(\xi)}
  &\leq \ABS{ \int_{0}^{\xi} \big( \ms(\calA_{1}(t),\dot{\vv}_{\hh}(t)) \big)\dt }
  \leq \int_{0}^{\xi}\ABS{\ms(\calA_{1}(t),\dot{\vv}_{\hh}(t))}\dt
  \leq 
  \frac12\int_{0}^{\xi}\NORM{\calA_1(t)}{0}^2\dt +
  \frac12\int_{0}^{\xi}\NORM{\dot{\vv}_{\hh}(t)}{0}^2\dt
  \nonumber\\[0.5em]
  &\leq
  \frac{T^2}{2}\int_{0}^{T}\NORM{ \big(1-\Piz{k-2}\big)\fv(\tau) }{0}^2d\tau +
  \frac{1}{2}\int_{0}^{\xi}\NORM{\dot{\vv}_{\hh}(t)}{0}^2\dt.
\end{align*}

\PGRAPH{Estimate of term $\TERM{L}{2}$}
Let $\ddot{\uv}_{\pi}$ be \emph{any} piecewise polynomial
approximation of degree at most $k$ of $\ddot{\uv}$ that satisfies
Lemma~\ref{lemma:hp:convergence:1}.
Now, consider the piecewise polynomial function defined on mesh $\Th$
that is given by
\begin{align*}
  (\calA_{2})_{\pi} = \int_{0}^{t}\ddot{\uv}_{\pi}(\tau)d\tau.
\end{align*}
Then, we start the estimate of the integral argument of $\TERM{L}{2}$
by using consistency property~\eqref{eq:msh:k-consistency} to add and
substract $(\calA_{2})_{\pi}$:
\begin{equation}
  \begin{array}{lll}
    &\abs{ \ms(\calA_{2},\dot{\vv}_{\hh})-\msh(\calA_{3},\dot{\vv}_{\hh}) }
    = \abs{ \ms(\calA_{2}-(\calA_{2})_{\pi},\dot{\vv}_{\hh})-\msh(\calA_{3}-(\calA_{2})_{\pi},\dot{\vv}_{\hh}) }
    &\quad\mbox{\big[add and subtract $\calA_{2}$ \big]}
    \nonumber\\[0.5em]
    &\qquad= \abs{ \ms(\calA_{2}-(\calA_{2})_{\pi},\dot{\vv}_{\hh})
      -  \msh(\calA_{3}-\calA_{2},\dot{\vv}_{\hh})
      -  \msh(\calA_{2}-(\calA_{2})_{\pi},\dot{\vv}_{\hh}) }
    &\quad\mbox{\big[use triangular inequality\big]}
    \nonumber\\[0.5em]
    &\qquad\leq 
    \abs{ \ms(\calA_{2}-(\calA_{2})_{\pi},\dot{\vv}_{\hh}) }
    +\abs{ \msh(\calA_{3}-\calA_{2},\dot{\vv}_{\hh}) }
    +\abs{ \msh(\calA_{2}-(\calA_{2})_{\pi},\dot{\vv}_{\hh}) }
    &\quad\mbox{\big[use~\eqref{eq:msh:continuity} and continuity of $\ms$\big]}
    \nonumber\\[0.5em]
    &\qquad\leq 
    \Big( 
    (1+\mu^*)\NORM{\calA_{2}-(\calA_{2})_{\pi}}{0}
    +\mu^*\NORM{\calA_{3}-\calA_{2}}{0}
    \Big)\,\NORM{\dot{\vv}_{\hh}}{0}
    &\quad\mbox{\big[note that $\mu^*<1+\mu^*$\big]}
    \nonumber\\[0.5em]
    &\qquad\leq 
    (1+\mu^*) \Big( \NORM{\calA_{2}-(\calA_{2})_{\pi}}{0}+\NORM{\calA_{3}-\calA_{2}}{0} \Big)\,\NORM{\dot{\vv}_{\hh}}{0}
    &\quad\mbox{\big[use Young's inequality\big]}
    \nonumber\\[0.5em]
    &\qquad\leq 
    \frac{(1+\mu^*)}{2} \Big( \NORM{\calA_{2}-(\calA_{2})_{\pi}}{0}+\NORM{\calA_{3}-\calA_{2}}{0} \Big)^2 
    + \frac{(1+\mu^*)}{2}\NORM{\dot{\vv}_{\hh}}{0}^2
    &\quad\mbox{\big[use $(a+b)^2\leq 2a^2+2b^2$\big]} 
    \nonumber\\[0.5em]
    &\qquad\leq 
    (1+\mu^*)\Big( \NORM{\calA_{2}-(\calA_{2})_{\pi}}{0}^2+\NORM{\calA_{3}-\calA_{2}}{0}^2 \Big)
    + \frac{(1+\mu^*)}{2}\NORM{\dot{\vv}_{\hh}}{0}^2.
  \end{array}
\end{equation}
To estimate $\NORM{\calA_{2}(t)-(\calA_{2}(t))_{\pi}}{0}^2$, we start
from the definition of $\calA_{2}(t)$ and $(\calA_{2}(t))_{\pi}$:
\begin{align}
  \begin{array}{lll}
    &\NORM{\calA_{2}(t)-(\calA_{2}(t))_{\pi}}{0}^2
    =\NORM{ \int_{0}^{t} \Big(\ddot{\uv}(\tau)-\ddot{\uv}_{\pi}(\tau)\Big)d\tau }{0}^2
    &\quad\mbox{\big[use definition of $\LTWO(\Omega)$-norm\big]}
    \nonumber\\[0.75em]
    &\qquad=\int_{\Omega}\ABS{ \int_{0}^{t} \Big(\ddot{\uv}(\tau)-\ddot{\uv}_{\pi}(\tau)\Big)d\tau }^2\dV
    &\quad\mbox{\big[use Jensen's inequality\big]}
    \nonumber\\[1.00em]
    &\qquad\leq \int_{\Omega} \bigg( T\int_{0}^{t} \ABS{ \ddot{\uv}(\tau)-\ddot{\uv}_{\pi}(\tau) }^2d\tau \bigg)\dV
    &\quad\mbox{\big[apply Fubini's Theorem\big]}
    \nonumber\\[1.00em]
    &\qquad\leq T\int_{0}^{t} \bigg(  \int_{\Omega}\ABS{ \ddot{\uv}(\tau)-\ddot{\uv}_{\pi}(\tau) }^2\dV \bigg)d\tau
    &\quad\mbox{\big[use definition of $\LTWO(\Omega)$-norm\big]}
    \nonumber\\[1.00em]
    &\qquad=T\int_{0}^{t} \NORM{ \ddot{\uv}(\tau)-\ddot{\uv}_{\pi}(\tau) }{0}^2d\tau
    &\quad\mbox{\big[apply Lemma~\ref{lemma:hp:convergence:1} \big]}
    \nonumber\\[1.00em]
    &\qquad\lesssim \Ts\frac{\hh^{2(\mu+1)}}{k^{2(m+1)}}\int_{0}^{t} \SNORM{\ddot{\uv}(\tau)}{m+1}^2d\tau
    &\quad\mbox{\big[use definition of $\LTWO(0,T;[\HS{m+1}(\Omega)]^2)$-norm\big]}
    \nonumber\\[1.00em]
    &\qquad\lesssim \Ts\frac{\hh^{2(\mu+1)}}{k^{2(m+1)}}\NORM{ \ddot{\uv} }{\LTWO(0,T; [\HS{m+1}(\Omega)]^2 )}^2.
  \end{array}
\end{align}
Similarly, to estimate $\NORM{\calA_{3}(t)-\calA_{2}(t)}{0}^2$, we
start from the definition of $\calA_{3}(t)$ and $\calA_{2}(t)$:
\begin{align}
  \begin{array}{lll}
    &\NORM{\calA_{3}(t)-\calA_{2}(t)}{0}^2
    =\NORM{ \int_{0}^{t} \Big(\Ph\ddot{\uv}(\tau)-\ddot{\uv}(\tau)\Big)d\tau }{0}^2
    &\quad\mbox{\big[use definition of $\LTWO(\Omega)$-norm\big]}
    \nonumber\\[1.00em]
    &\qquad=\int_{\Omega} \ABS{ \int_{0}^{t} \Big(\Ph\ddot{\uv}(\tau)-\ddot{\uv}(\tau)\Big)d\tau }^2\dV
    &\quad\mbox{\big[use Jensen's inequality\big]}
    \nonumber\\[1.00em]
    &\qquad\leq \int_{\Omega} \bigg( T\int_{0}^{t} \ABS{ \Ph\ddot{\uv}(\tau)-\ddot{\uv}(\tau) }^2d\tau \bigg)\dV
    &\quad\mbox{\big[apply Fubini's Theorem\big]}
    \nonumber\\[1.00em]
    &\qquad\leq T\int_{0}^{t} \bigg(  \int_{\Omega}\ABS{ \Ph\ddot{\uv}(\tau)-\ddot{\uv}(\tau) }^2\dV \bigg)d\tau
    &\quad\mbox{\big[use definition of $\LTWO(\Omega)$-norm\big]}
    \nonumber\\[1.00em]
    &\qquad= T\int_{0}^{t} \NORM{ \Ph\ddot{\uv}(\tau)-\ddot{\uv}(\tau) }{0}^2d\tau
    &\quad\mbox{\big[apply Lemma~\ref{lemma:Ph:error:estimate} \big]}
    \nonumber\\[1.00em]
    &\qquad\lesssim \Ts\frac{\hh^{2(\mu+1)}}{\RED{k^{2(m+1)}}}\int_{0}^{t} \SNORM{\ddot{\uv}(\tau)}{m+1}^2d\tau
    &\quad\mbox{\big[use definition of $\LTWO(0,T;[\HS{m+1}(\Omega)]^2)$-norm\big]}
    \nonumber\\[1.00em]
    &\qquad\lesssim \Ts\frac{\hh^{2(\mu+1)}}{\RED{k^{2(m+1)}}}\NORM{ \ddot{\uv} }{\LTWO(0,T; [\HS{m+1}(\Omega)]^2 )}^2.
  \end{array}
\end{align}
Using the previous estimates, and noting that the integration on the
time interval $[0,\xi]$, $\xi\leq\Ts$, produces an additional factor
$\Ts$, we obtain the final upper bound for term $\TERM{L}{2}$:
\begin{align*}
  \ABS{\TERM{L}{2}(\xi)} 
  \leq \int_{0}^{\xi} \ABS{ \msh(\calA_3,\dot{\vv}_{\hh}) - \ms(\calA_{2},\dot{\vv}_{\hh}) }\dt
  \lesssim 
  \big(1+\mu^*\big)\Ts^2\frac{\hh^{2(\mu+1)}}{\RED{k^{2(m+1)}}}\NORM{ \ddot{\uv} }{\LTWO(0,T; (\HS{m+1}(\Omega))^2 )}^2 +
  \frac{1+\mu^*}{2}\int_{0}^{\xi} \NORM{\dot{\vv}_{\hh}}{0}^2\dt.
\end{align*}

\PGRAPH{Estimate of term $\TERM{L}{3}$}
Since $\vvh(\xi)=0$ by definition, term $\TERM{L}{3}$ only depends on
the approximation error on $\dot{\uv}(0)=\uv_1$, i.e., the initial
condition for $\dot{\uv}$.
Using~\eqref{eq:msh:continuity} and Young's inequality yield:
\begin{align*}
  \ABS{\TERM{L}{3}(\xi)}
  &= \ABS{\msh(\dot{\uv}(0)-\dot{\uv}_{\hh}(0),\vvh(0))}
  \leq \mu^*\NORM{\dot{\uv}(0)-\dot{\uv}_{\hh}(0)}{0} \,\NORM{\vvh(0)}{0}
  \nonumber\\[0.5em]
  &\leq 
  \frac{(\mu^*)^2}{2}\NORM{\dot{\uv}(0)-\dot{\uv}_{\hh}(0)}{0}^2+
  \frac{(\mu^*)^2}{2}\NORM{\vvh(0)}{0}^2.
\end{align*}

\PGRAPH{Estimate of term $\TERM{L}{4}$}
We start from the definition of term $\TERM{L}{4}$ and apply the
triangular inequality to find that:
\begin{align}
  \ABS{\TERM{L}{4}(\xi)}
  = \ABS{\int_{0}^{\xi}\Big(\,\msh(\ddot{\rhov},\vvh)+\msh(\dot{\rhov},\dot{\vv}_{\hh})\,\Big)\dt}
  \leq 
  \int_{0}^{\xi}\ABS{\msh(\ddot{\rhov},\vvh)}\dt +
  \int_{0}^{\xi}\ABS{\msh(\dot{\rhov},\dot{\vv}_{\hh})}\dt.
  \label{eq:bound:L4:10}
\end{align}
Both terms on the right depend on the exact solution $\uv$ and its
approximation provided by the energy projection $\Ph\uv$.
We bound the first term by starting from~\eqref{eq:msh:continuity}:
\begin{align*}
  \begin{array}{lll}
    &\int_{0}^{\xi}\ABS{\msh(\ddot{\rhov},\vvh)}\dt
    \leq \mu^*\int_{0}^{\xi}\NORM{\ddot{\rhov}}{0}\,\NORM{\vvh}{0}\dt
    &\quad\mbox{\big[use Young's inequality\big]}
    \nonumber\\[0.5em]
    &\qquad\leq 
    \frac{\mu^*}{2}\int_{0}^{\xi}\NORM{\ddot{\rhov}}{0}^2\dt +
    \frac{\mu^*}{2}\int_{0}^{\xi}\NORM{\vvh}{0}^2\dt
    &\quad\mbox{\big[apply Lemma~\ref{lemma:Ph:error:estimate} to $\ddot{\rhov}$\big]}
    \nonumber\\[0.5em]
    &\qquad\lesssim
    \frac{\mu^*}{2} \frac{\hh^{2(\mu+1)}}{\RED{k^{2(m+1)}}} \int_{0}^{\xi}\SNORM{\ddot{\uv}}{m+1}^2\dt +
    \frac{\mu^*}{2}\int_{0}^{\xi}\NORM{\vvh}{0}^2\dt
    &\quad\mbox{\big[use definition of $\LTWO(0,T;[\HS{m+1}(\Omega)]^2)$-norm\big]}
    \nonumber\\[0.5em]
    &\qquad\lesssim
    \frac{\mu^*}{2} \frac{\hh^{2(\mu+1)}}{\RED{k^{2(m+1)}}} \NORM{\ddot{\uv}}{\LTWO(0,T;[\HS{m+1}(\Omega)]^2}^2 +
    \frac{\mu^*}{2}\int_{0}^{\xi}\NORM{\vvh}{0}^2\dt.
  \end{array}
\end{align*}
Similarly, we bound the second term by starting
from~\eqref{eq:msh:continuity}:
\begin{align*}
  \begin{array}{lll}
    &\int_{0}^{\xi}\ABS{\msh(\dot{\rhov},\dot{\vv}_{\hh})}\dt
    \leq \mu^*\int_{0}^{\xi}\NORM{\dot{\rhov}}{0}\,\NORM{\dot{\vv}_{\hh}}{0}\dt
    &\quad\mbox{\big[use Young's inequality\big]}
    \nonumber\\[0.5em]
    &\qquad\leq 
    \frac{\mu^*}{2}\int_{0}^{\xi}\NORM{\dot{\rhov}}{0}^2\dt +
    \frac{\mu^*}{2}\int_{0}^{\xi}\NORM{\dot{\vv}_{\hh}}{0}^2\dt
    &\quad\mbox{\big[apply Lemma~\eqref{lemma:Ph:error:estimate} to $\dot{\rhov}$\big]}
    \nonumber\\[0.5em]
    &\qquad\lesssim
    \frac{\mu^*}{2} \frac{\hh^{2(\mu+1)}}{\RED{k^{2(m+1)}}} \int_{0}^{\xi}\SNORM{\dot{\uv}}{m+1}^2\dt +
    \frac{\mu^*}{2}\int_{0}^{\xi}\NORM{\dot{\vv}_{\hh}}{0}^2\dt
    &\quad\mbox{\big[use definition of $\LTWO(0,T;[\HS{m+1}(\Omega)]^2)$-norm\big]}
    \nonumber\\[0.5em]
    &\qquad\lesssim
    \frac{\mu^*}{2} \frac{\hh^{2(\mu+1)}}{\RED{k^{2(m+1)}}} \NORM{\dot{\uv}}{\LTWO(0,T;[\HS{m+1}(\Omega)]^2}^2 +
    \frac{\mu^*}{2}\int_{0}^{\xi}\NORM{\dot{\vv}_{\hh}}{0}^2\dt.
  \end{array}
\end{align*}
Using these two estimates in~\eqref{eq:bound:L4:10}, we immediately
find that
\begin{align}
  \ABS{\TERM{L}{4}(\xi)}\lesssim
  \frac{\mu^*}{2} \frac{\hh^{2(\mu+1)}}{\RED{k^{2(m+1)}}} \Big(
  \NORM{\ddot{\uv}}{\LTWO(0,T;[\HS{m+1}(\Omega)]^2}^2 +
  \NORM{\dot {\uv}}{\LTWO(0,T;[\HS{m+1}(\Omega)]^2}^2 
  \Big)+
  \mu^*\int_{0}^{\xi}\NORM{\dot{\vv}_{\hh}}{0}^2\dt.
  \label{eq:bound:L4:20}
\end{align}

\PGRAPH{Estimate of $\NORM{\etav}{0}$ and conclusion of the proof}
On collecting the upper bounds of $\TERM{L}{i}(\xi)$, $i=1,2,3,4$,
absorbing the factors $\mu^*$ in the hidden constants and using the
lower bound~\eqref{eq:L2estimate:proof:200}, we have that
\begin{align}
  \mu_*\NORM{\etav(\xi)}{0}^2 
  =&\,     \mu_*\ms(\etav(\xi),\etav(\xi))
  \leq     \msh(\etav(\xi),\etav(\xi))
  \nonumber\\[0.5em]
  \lesssim
  &\,
  \msh(\etav(0),  \etav(0)) 
  + \NORM{\dot{\uv}(0)-\dot{\uv}_{\hh}(0)}{0}^2
  + \Ts^2\int_{0}^{T}\NORM{ \big(I-\Piz{k-2}\big)\fv(\tau) }{0}^2d\tau
  \nonumber\\[0.5em]
  &
  + \frac{\hh^{2(\mu+1)}}{\RED{k^{2(m+1)}}} \big( 
  \NORM{\dot{\uv}}{\LTWO(0,T;[\HS{m+1}(\Omega)]^2)}^2
  + \Ts^2\NORM{ \ddot{\uv} }{\LTWO(0,T; [\HS{m+1}(\Omega)]^2 )}^2
  \big)
  \nonumber\\[0.5em]
  &
  + \int_{0}^{\xi}\NORM{\dot{\vv}_{\hh}}{0}^2\dt
  + \int_{0}^{\xi}\NORM{\vvh}{0}^2\dt.
  \label{eq:etav:bound:10}
\end{align}
We estimate the first term in the right-hand side
of~\eqref{eq:etav:bound:10} by using
property~\eqref{eq:msh:k-consistency}, the definition of $\etav(0)$,
adding and subtracting $\uv(0)$
and estimate~\eqref{eq:Ph:L2:estimate}:
\begin{align}
  \msh(\etav(0),\etav(0))
  &\leq  \mu^*\NORM{\etav(0)}{0}^2
  =     \mu^*\NORM{\uvh(0)-\Ph\uv(0)}{0}^2
  \nonumber\\[0.5em]
  &\leq 2\mu^*\NORM{\uvh(0)-\uv(0)}{0}^2 + 2\mu^*\NORM{\uv(0)-\Ph\uv(0)}{0}^2
  \nonumber\\[0.5em]
  &\lesssim \NORM{\uvh(0)-\uv_0}{0}^2    + \frac{\hh^{2(\mu+1)}}{\RED{k^{2(m+1)}}}\SNORM{\uv_0}{m+1}^2.
  \label{eq:etav:bound:20}
\end{align}
We also estimate the last integral term in~\eqref{eq:etav:bound:10} by
using the definition of the $\LTWO(\Omega)$-norm,
definition~\eqref{eq:L2estimate:proof:130}, Jensen's inequality, and
Fubini's Theorem to change the integration order:
\begin{align}
  \NORM{\vvh(t)}{0}^2 
  &=    \int_{\Omega}\bigg| \int_{t}^{\xi}\etav(\tau)d\tau \bigg|^2\dV
  \leq \int_{\Omega}\ABS{\xi-t}\bigg(\int_{t}^{\xi}\ABS{\etav(\tau)}^2d\tau\bigg)\dV
  \nonumber\\[0.5em]
  &=    \ABS{\xi-t}\int_{t}^{\xi}\bigg(\int_{\Omega}\ABS{\etav(\tau)}^2\dV\bigg)d\tau
  \leq T\int_{0}^{\xi}\NORM{\etav(\tau)}{0}^2\,d\tau
  \label{eq:etav:bound:30}
  \intertext{and}
  \int_{0}^{\xi}\NORM{\vvh(t)}{0}^2\dt 
  &= 
  \bigg(T\int_{0}^{\xi}\NORM{\etav(\tau)}{0}^2d\tau\bigg)\,
  \bigg(\int_{0}^{\xi}\dt\bigg)
  \leq T^2\int_{0}^{\xi}\NORM{\etav(\tau)}{0}^2\,d\tau
    \label{eq:etav:bound:35}
\end{align}
Using estimates~\eqref{eq:etav:bound:20} and~\eqref{eq:etav:bound:35}
in~\eqref{eq:etav:bound:10} and recalling that
$\dot{\vv}_{\hh}=-\etav$, we find that
\begin{align}
  \NORM{\etav(\xi)}{0}^2 
  \lesssim
  &\,
  \NORM{\uv(0)-\uvh(0)}{0}^2
  + \NORM{\dot{\uv}(0)-\dot{\uv}_{\hh}(0)}{0}^2
  + \Ts^2\int_{0}^{T}\NORM{ \big(I-\Piz{k-2}\big)\fv(\tau) }{0}^2d\tau
  \nonumber\\[0.5em]
  &
  + \frac{\hh^{2(\mu+1)}}{\RED{k^{2(m+1)}}} \big( 
  \SNORM{\uv_0}{m+1}^2 
  + \NORM{\dot{\uv}}{\LTWO(0,T;[\HS{m+1}(\Omega)]^2}^2
  + \Ts^2\NORM{ \ddot{\uv} }{\LTWO(0,T; [\HS{m+1}(\Omega)]^2 )}^2
  \big)
  \nonumber\\[0.5em]
  &+ 
  (1+T^2)\int_{0}^{\xi}\NORM{\etav(t)}{0}^2\dt.
  \label{eq:etav:bound:40}
\end{align}
An application of Gronwall's inequality yields, for (almost) every
$\ts\in[0,\Ts]$, the desired upper bound on $\etav(t)$:
\begin{align}
  \NORM{\etav(t)}{0} 
  \lesssim
  \,\Cs(\Ts)\bigg(&
  \NORM{\uv(0)-\uvh(0)}{0}
  + \NORM{\dot{\uv}(0)-\dot{\uv}_{\hh}(0)}{0}
  + \int_{0}^{T}\NORM{ \big(1-\Piz{k-2}\big)\fv(\tau) }{0}d\tau
  \nonumber\\[0.5em]
  &
  +\frac{\hh^{\mu+1}}{\RED{k^{m+1}}}\big( 
  \SNORM{\uv_0}{s+1,\Omega} 
  + \NORM{\dot{\uv}}{\LTWO(0,T;[\HS{m+1}(\Omega)]^2}
  + \NORM{ \ddot{\uv} }{\LTWO(0,T;[ \HS{m+1}(\Omega])^2 )}
  \big)
  \bigg),
  \label{eq:etav:bound:50}
\end{align}
where $\Cs(\Ts)$ depends on $\Ts$, the final observation time, and 
the constant hidden in the ``$\lesssim$'' notation depends on $\mu^*$ 
and the mesh regularity constant $\varrho$, but is independent of $\hh$.
Finally, we prove the assertion of the theorem through the triangular
inequality $\NORM{\uvh-\uv}{0}\leq\NORM{\rhov}{0}+\NORM{\etav}{0}$
and, then, on using~\eqref{eq:Ph:H1:estimate} to bound $\rhov$
and~\eqref{eq:etav:bound:50} to bound $\etav$.
\ENDPROOF

\section{Numerical investigations}
\label{sec5:numerical}

In this section, we carry out number of numerical investigation to
validate the VEM by documenting its optimal accuracy and its
dissipation/dispersion properties in solving wave propagation
problems.
To this end, in section~\ref{subsec51:manuf_soln} we show the optimal
convergence properties of the VEM by using a manufactured solution on
three different mesh families, each one possessing some special
feature.
In section~\ref{subsec52:diss-disp}, we carry out a \textit{Von Neumann}
analysis by studying the dispersion and dissipation of an elementary
wave on a periodic domain tessellated in different ways.

\RED{
\begin{remark}
  All our numerical results are obtained by using the stabilization
  introduced in Section~\ref{subsec:vem:implementation}.
  This stabilization is independent of the polynomial degree $k$ and
  our numerical results are in good accordance with the expected
  results.
  However, it is known that practical tests in situations like the
  scalar Poisson problem show that for moderate-to-high values of $k$
  this might not be a good choice (especially in three-dimensions).
  An alternative choice that is surely worth mentioning is provided by
  the so called ``D-recipe''
  stabilization.~\cite{Beirao-Chernov-Mascotto-Russo:2018,Mascotto:2018,Dassi-Mascotto:2018}
\end{remark}
}

\begin{figure}
  \centering
  \begin{tabular}{ccc}
    \begin{overpic}[scale=0.2]{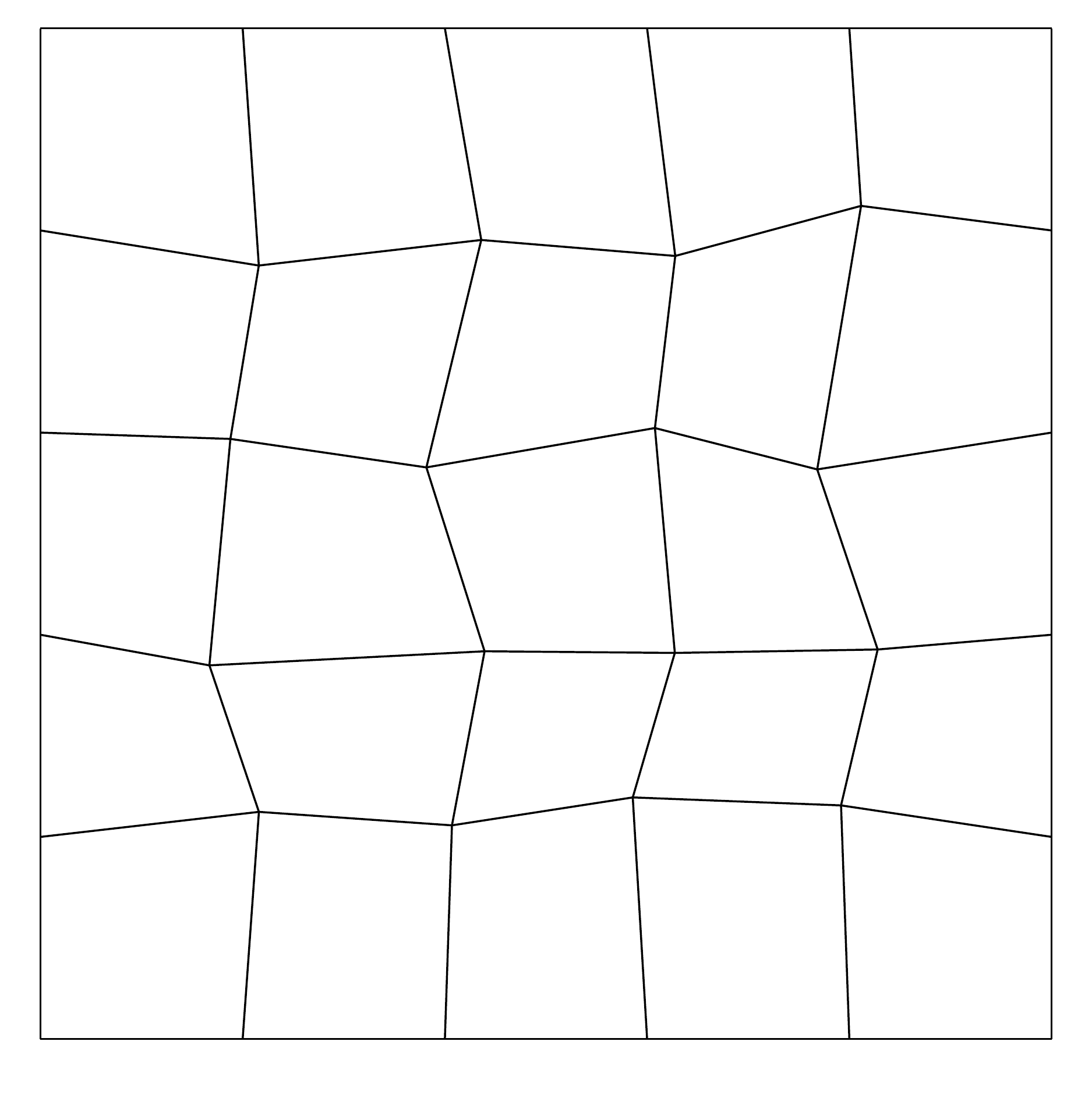}
    \end{overpic} 
    &
    \begin{overpic}[scale=0.2]{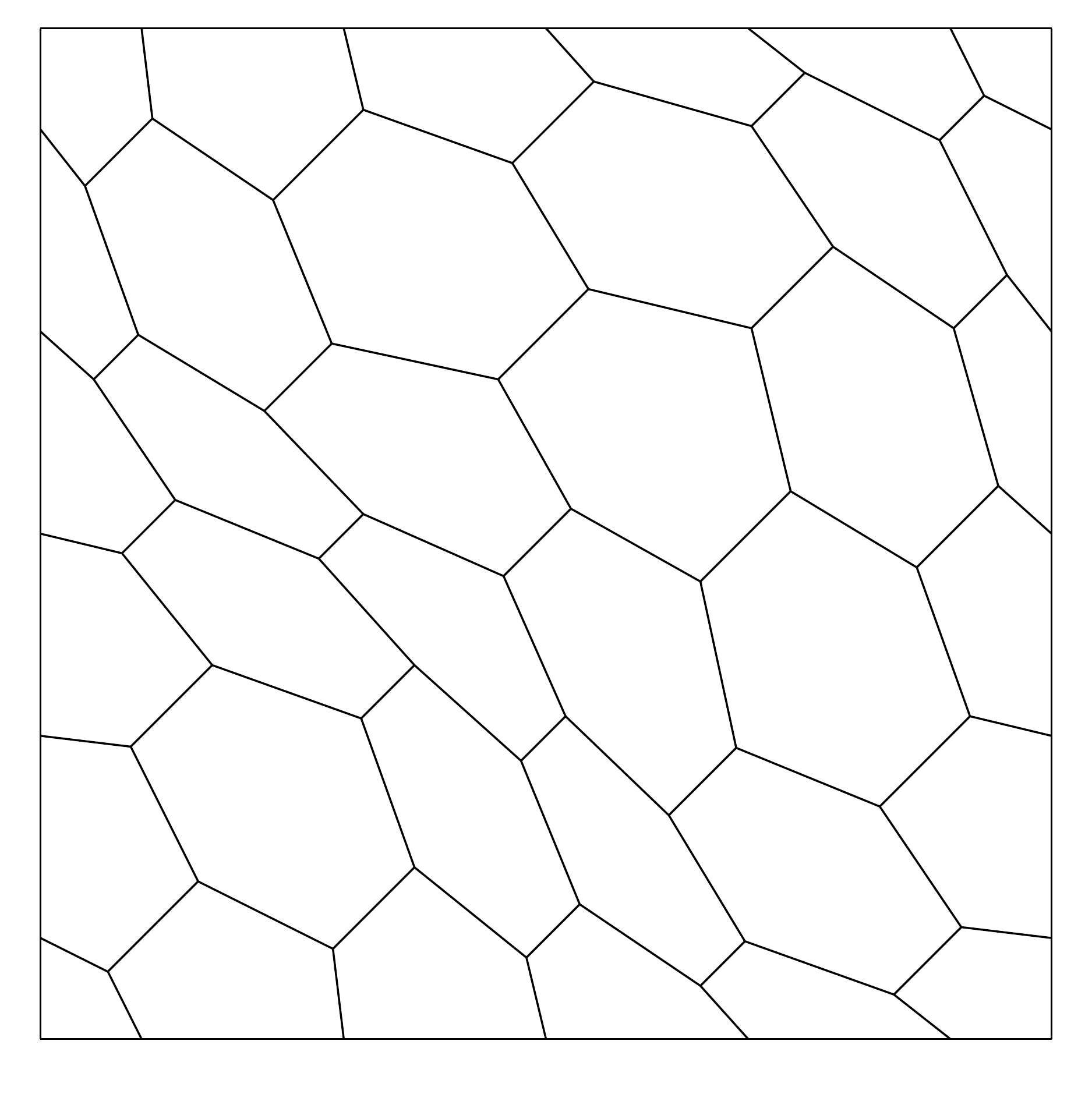}
    \end{overpic} 
    &
    \begin{overpic}[scale=0.2]{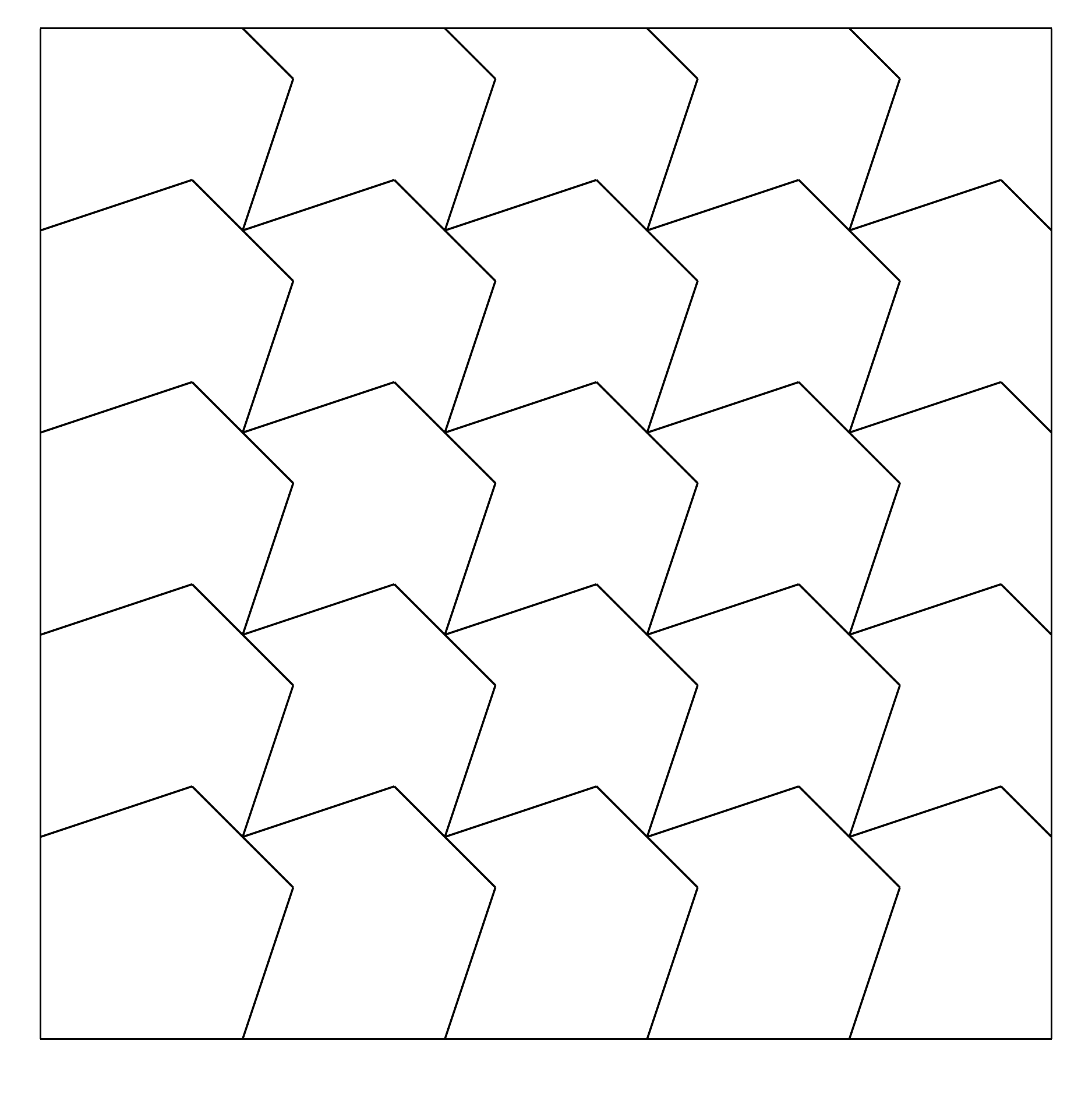}
    \end{overpic}
    \\[0.25em]
    \begin{overpic}[scale=0.2]{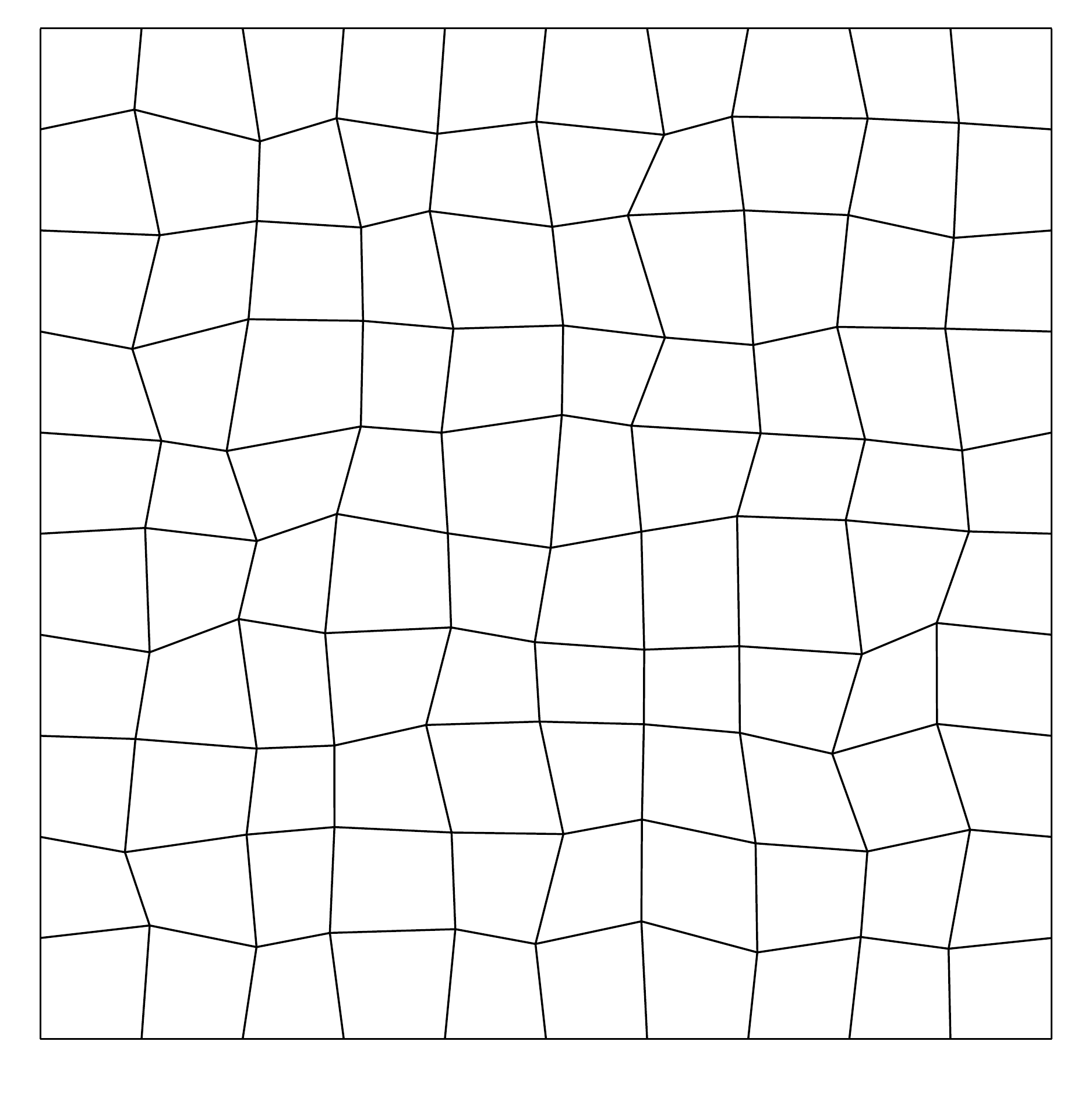}
    \end{overpic} 
    &
    \begin{overpic}[scale=0.2]{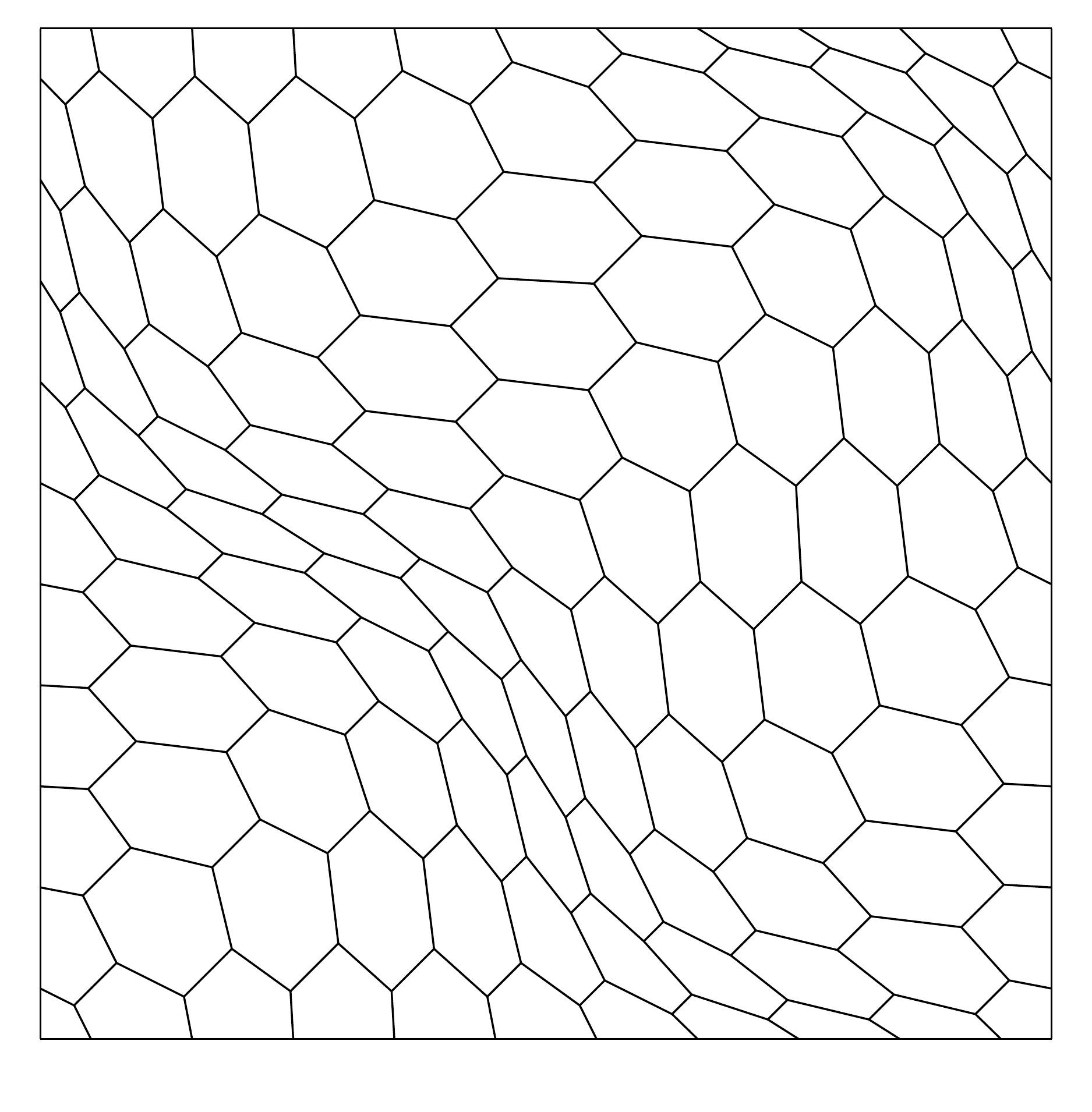}
    \end{overpic} 
    &
    \begin{overpic}[scale=0.2]{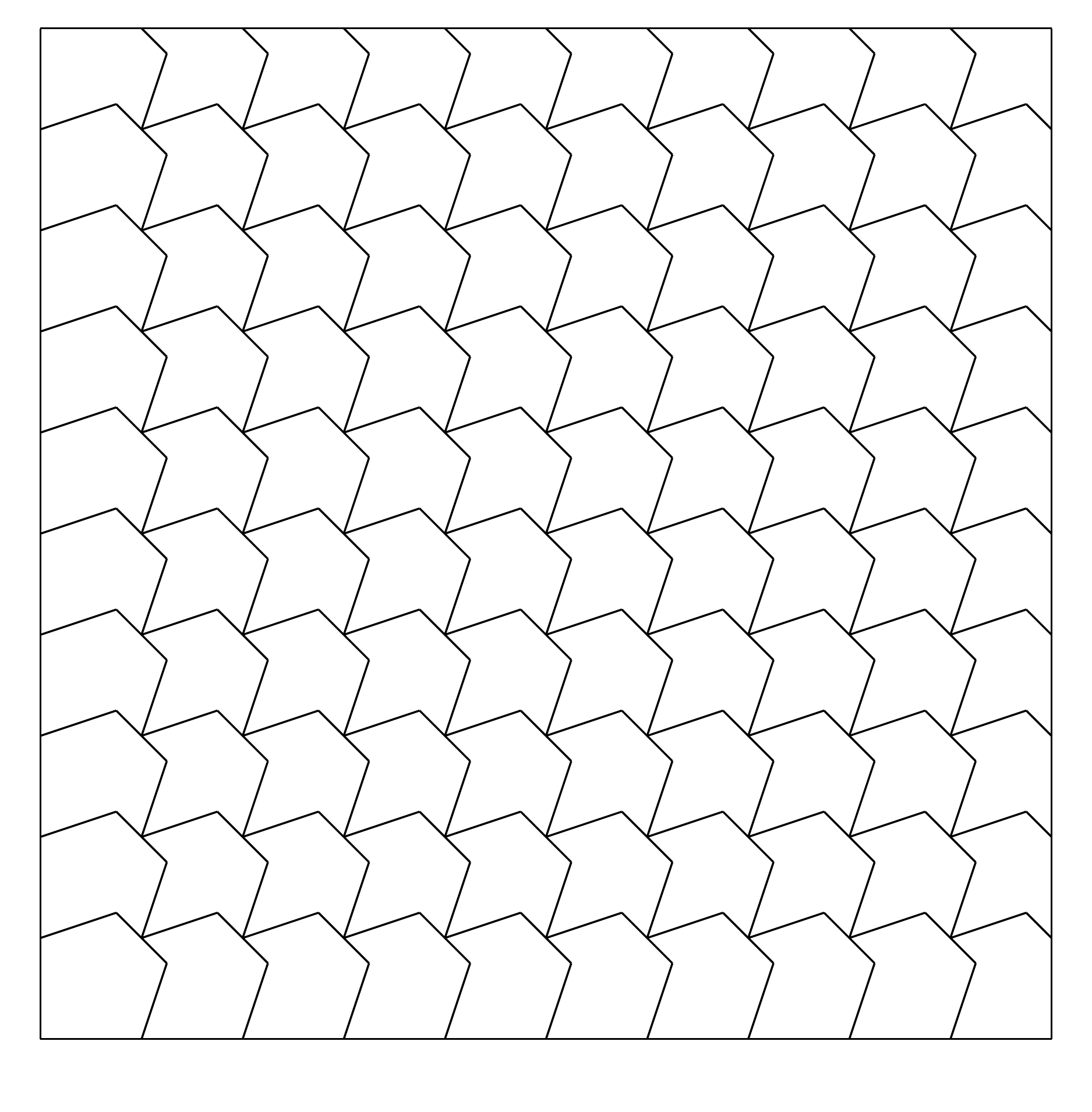}
    \end{overpic}
    \\[-0.25em]
    \textit{Mesh~1} & \textit{Mesh~2} & \textit{Mesh~3}
  \end{tabular}
  \caption{Base meshes (top row) and first refined meshes (bottom row)
    of the following mesh families from left to right: randomized
    quadrilateral mesh; mainly hexagonal mesh; nonconvex octagonal
    mesh.}
  \label{fig:Meshes}
\end{figure}

\begin{figure}
  \centering
  \begin{tabular}{cc}
    \begin{overpic}[scale=0.325]{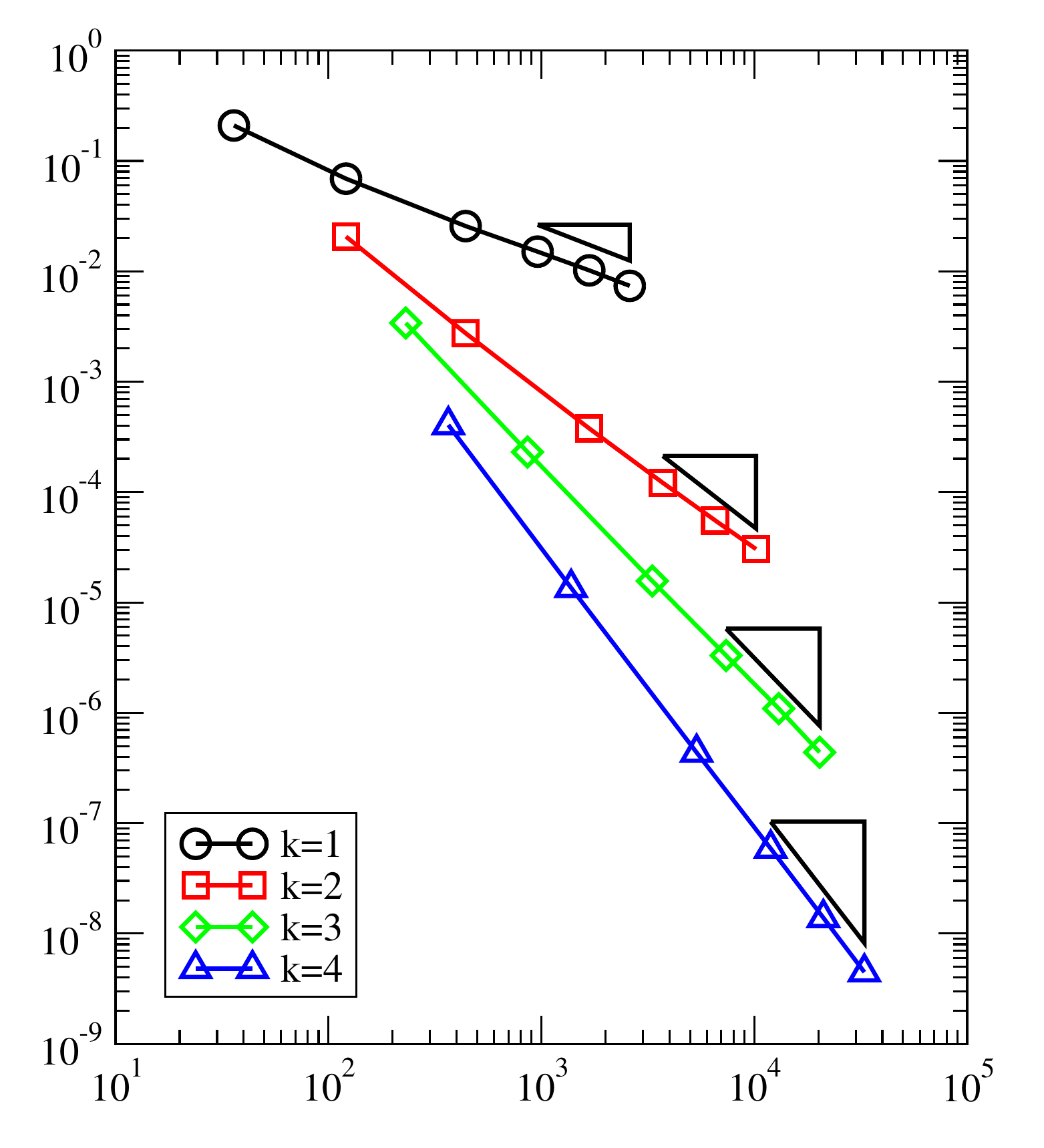}
      \put(-5,9){\begin{sideways}\textbf{$\mathbf{L^2}$ relative approximation error}\end{sideways}}
      \put(20,-2) {\textbf{\#degrees of freedom}}
      \put(50,   82.0){\textbf{1}}
      \put(57.5,62.0){\textbf{1.5}}
      \put(67,   46.5){\textbf{2}}
      \put(68,   30.0){\textbf{2.5}}
      \end{overpic} 
    &\qquad
    \begin{overpic}[scale=0.325]{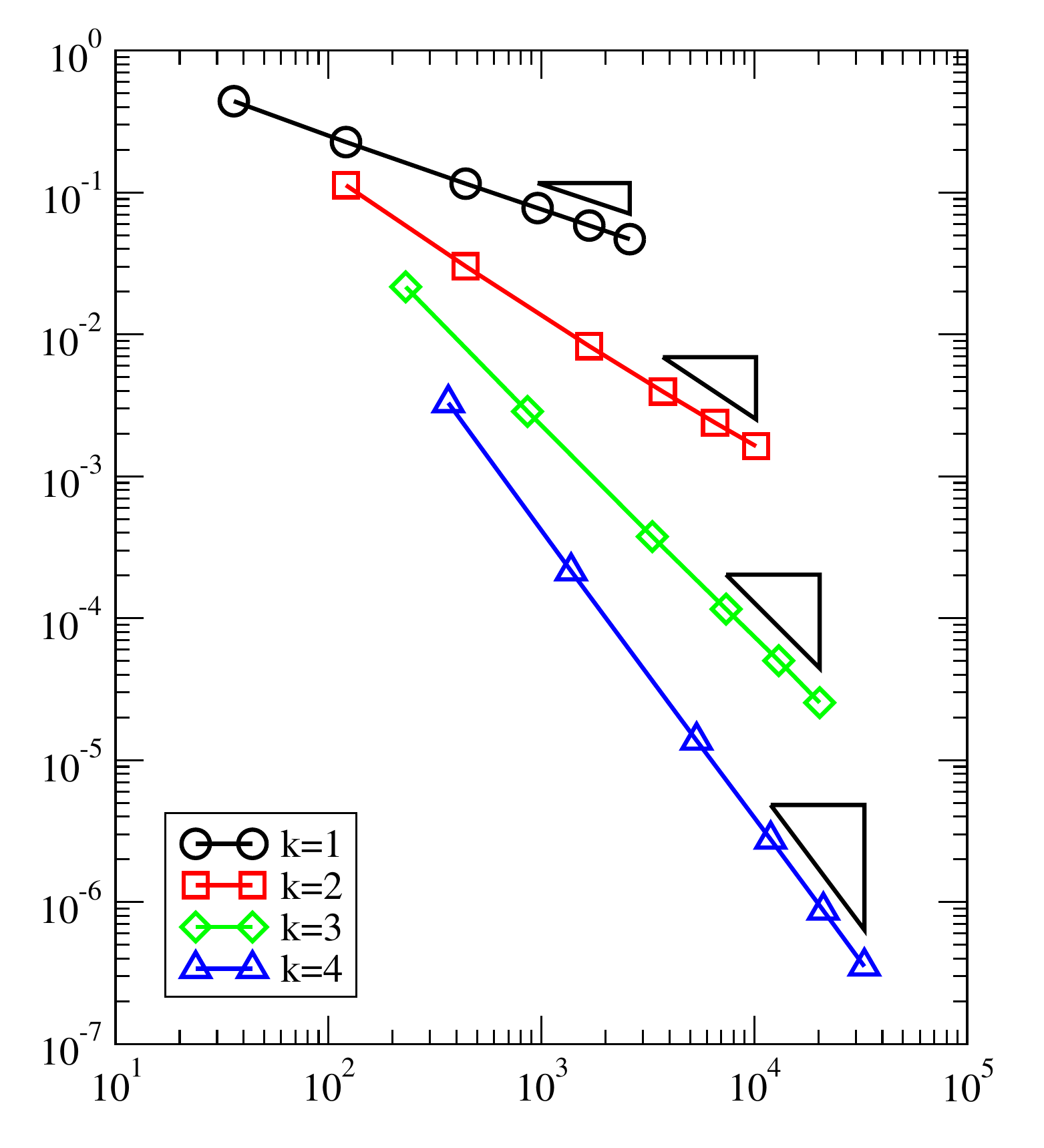}
      \put(-5,9){\begin{sideways}\textbf{$\mathbf{H^1}$ relative approximation error}\end{sideways}}
      \put(20,-2) {\textbf{\#degrees of freedom}}
      \put(47,85.0){\textbf{0.5}}
      \put(61,70.0){\textbf{1}}
      \put(64,51.0){\textbf{1.5}}
      \put(70,31.0){\textbf{2}}      
    \end{overpic}
    \\[0.5em]
  \end{tabular}
  \caption{Convergence plots for the virtual element approximation of
    Problem~\eqref{eq:pblm:strong:A}-\eqref{eq:pblm:strong:E} with
    exact solution~\eqref{eq:benchmark:solution} using
    family~\textit{Mesh~1} of randomized quadrilateral meshes. 
    Error curves are computed using the $\LTWO$ norm (left panels) and
    $\HONE$ norm (right panels) and are plot versus the 
    number of degrees of freedom.
  }
  \label{fig:quads:rates}
\end{figure}

\begin{figure}
  \centering
  \begin{tabular}{cc}
    \begin{overpic}[scale=0.325]{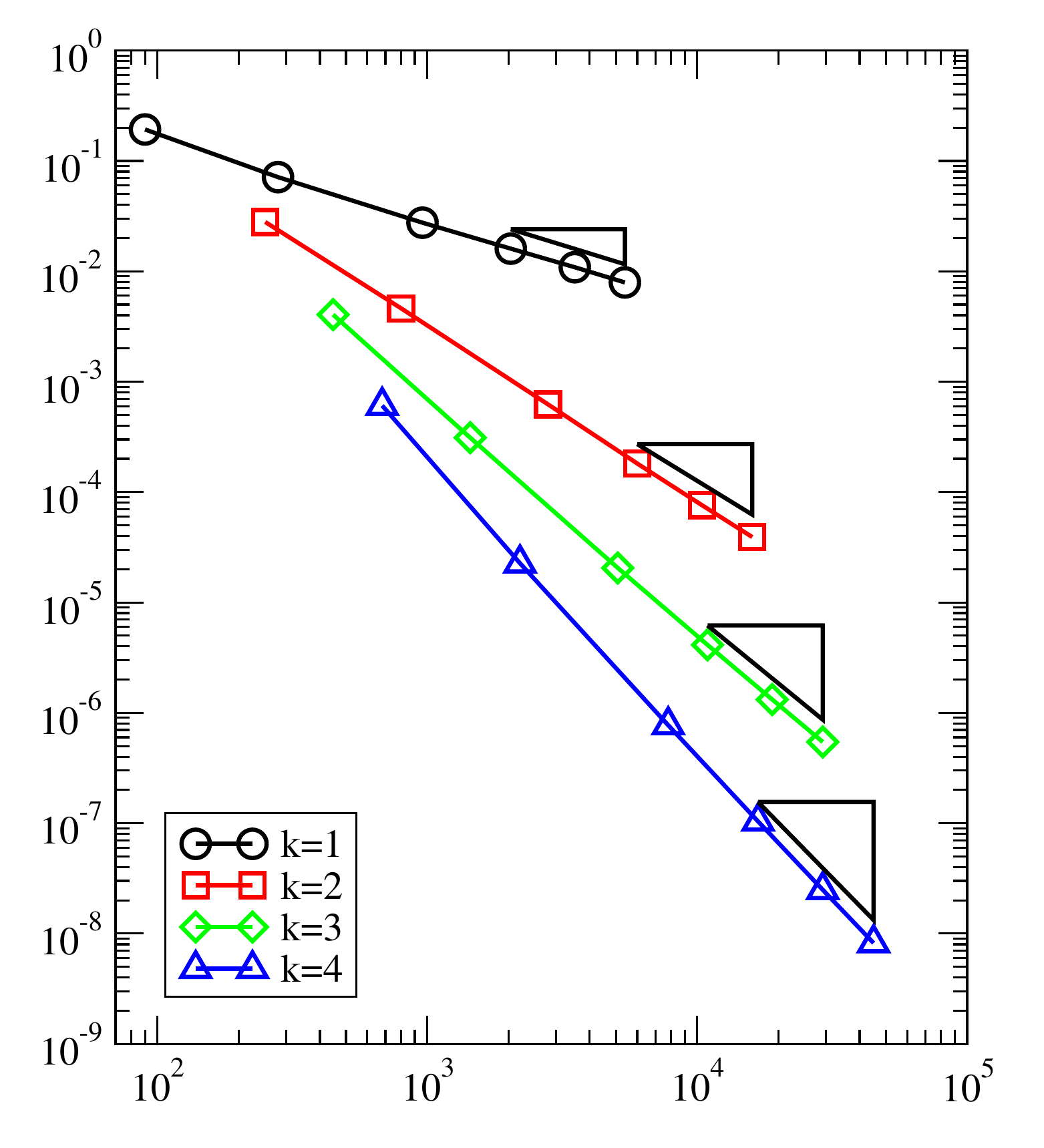}
      \put(-5,9){\begin{sideways}\textbf{$\mathbf{L^2}$ relative approximation error}\end{sideways}}
      \put(20,-2) {\textbf{\#degrees of freedom}}
      \put(48,81.5){\textbf{1}}
      \put(57,63){\textbf{1.5}}
      \put(66,46.5){\textbf{2}}
      \put(66.5,31){\textbf{2.5}}
    \end{overpic} 
    & \qquad
    \begin{overpic}[scale=0.325]{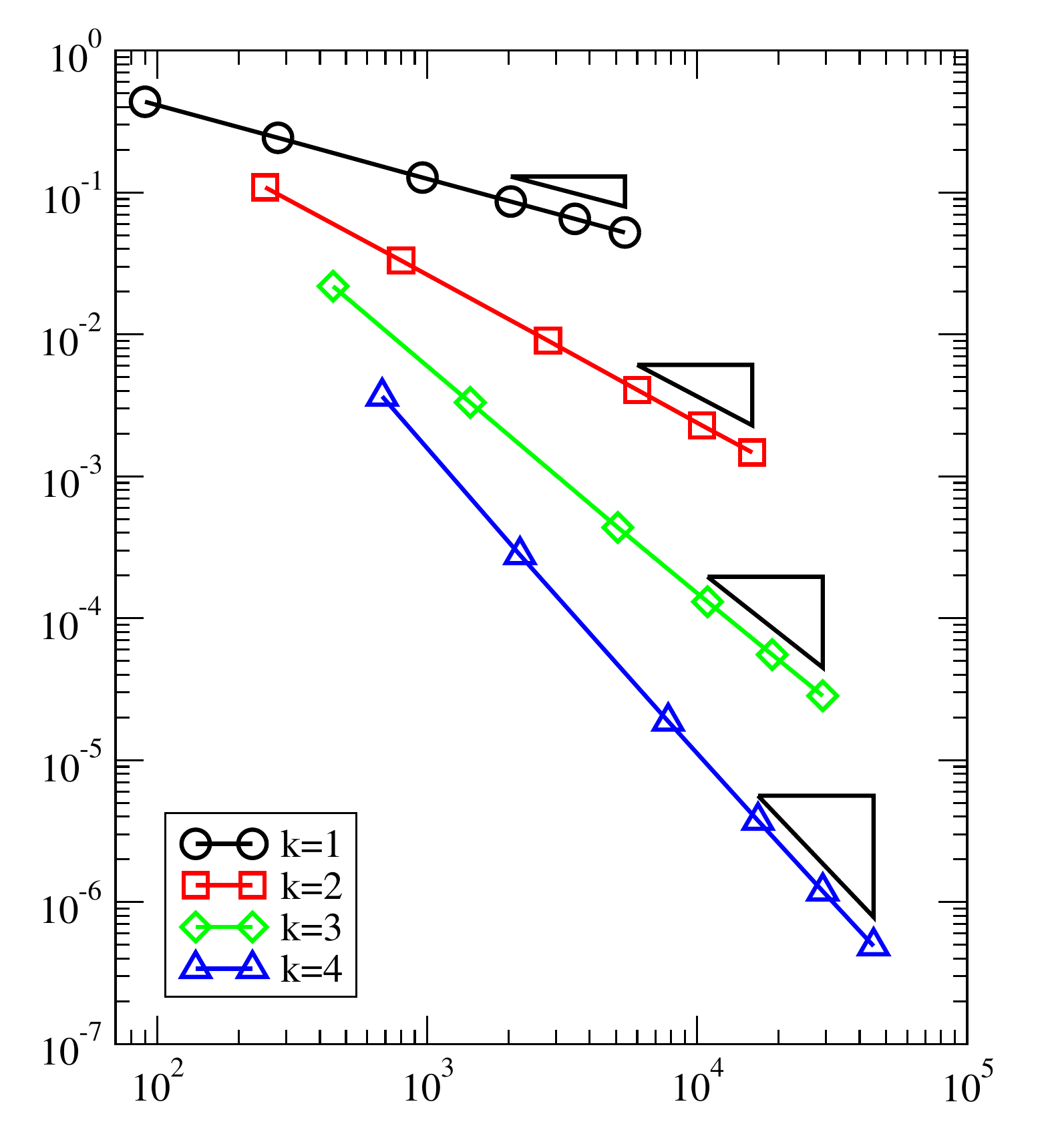}
      \put(-5,9){\begin{sideways}\textbf{$\mathbf{H^1}$ elative approximation error}\end{sideways}}
      \put(20,-2) {\textbf{\#degrees of freedom}}
     \put(47,85.5){\textbf{0.5}}
      \put(60,70.0){\textbf{1}}
      \put(63,51.0){\textbf{1.5}}
      \put(70,32.0){\textbf{2}}     
    \end{overpic}
    \\[0.5em]
  \end{tabular}
  \caption{Convergence plots for the virtual element
    approximation of
    Problem~\eqref{eq:pblm:strong:A}-\eqref{eq:pblm:strong:E} with
    exact solution~\eqref{eq:benchmark:solution} using
    family~\textit{Mesh~2} of mainly hexagonal meshes. 
    Error curves are computed using the $\LTWO$ norm (left panels) and
    $\HONE$ norm (right panels) and are plot versus
    the number of degrees of freedom.
  }
  \label{fig:hexa:rates}
\end{figure}

\begin{figure}
  \centering
  \begin{tabular}{cc}
    \begin{overpic}[scale=0.325]{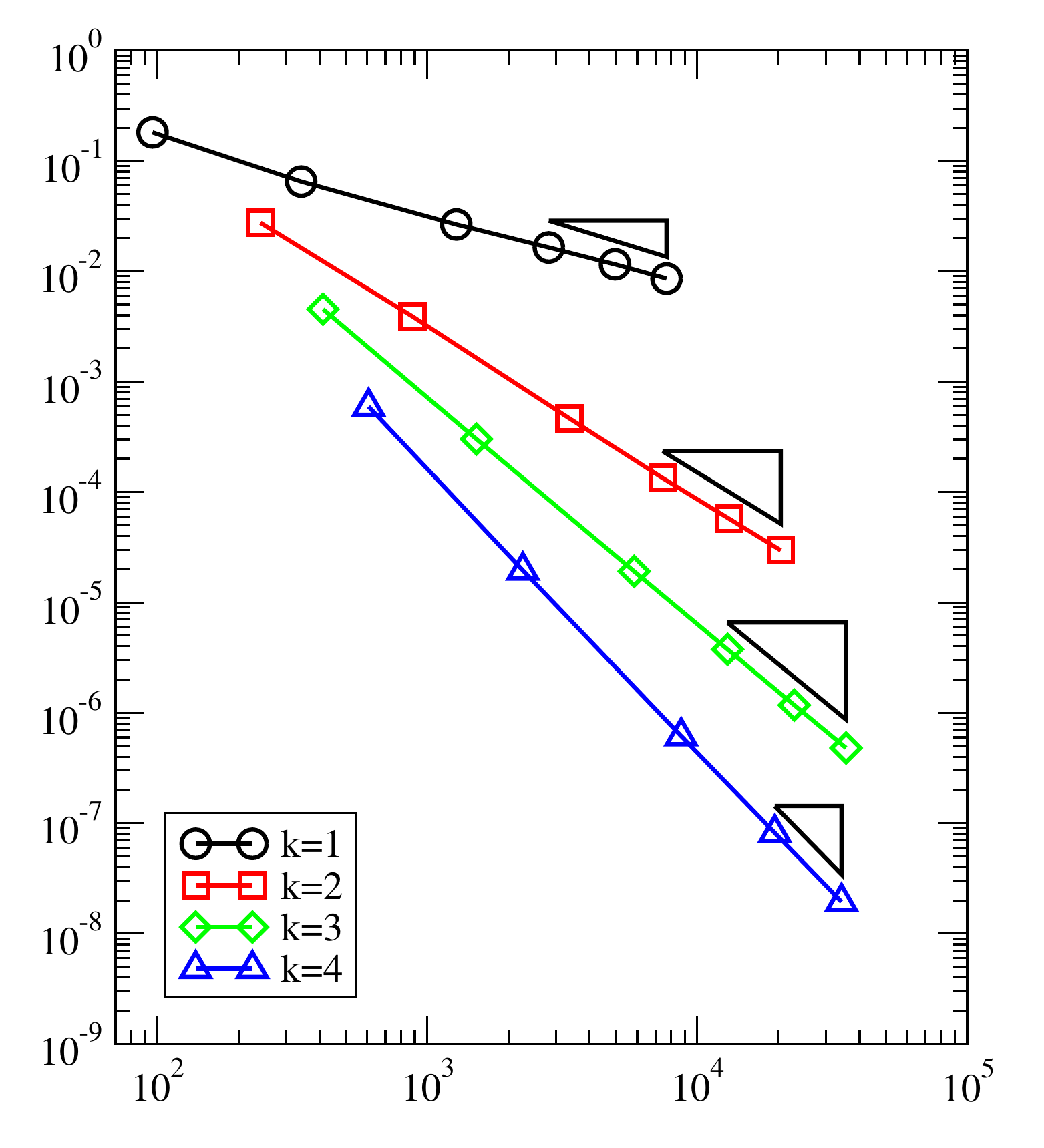}
      \put(-5,9){\begin{sideways}\textbf{$\mathbf{L^2}$ relative approximation error}\end{sideways}}
      \put(20,-2) {\textbf{\#degrees of freedom}}
      \put(52,82){\textbf{1}}
      \put(59,62){\textbf{1.5}}
      \put(67,46.5){\textbf{2}}
      \put(66.5,31){\textbf{2.5}}
    \end{overpic}
    & \qquad
    \begin{overpic}[scale=0.325]{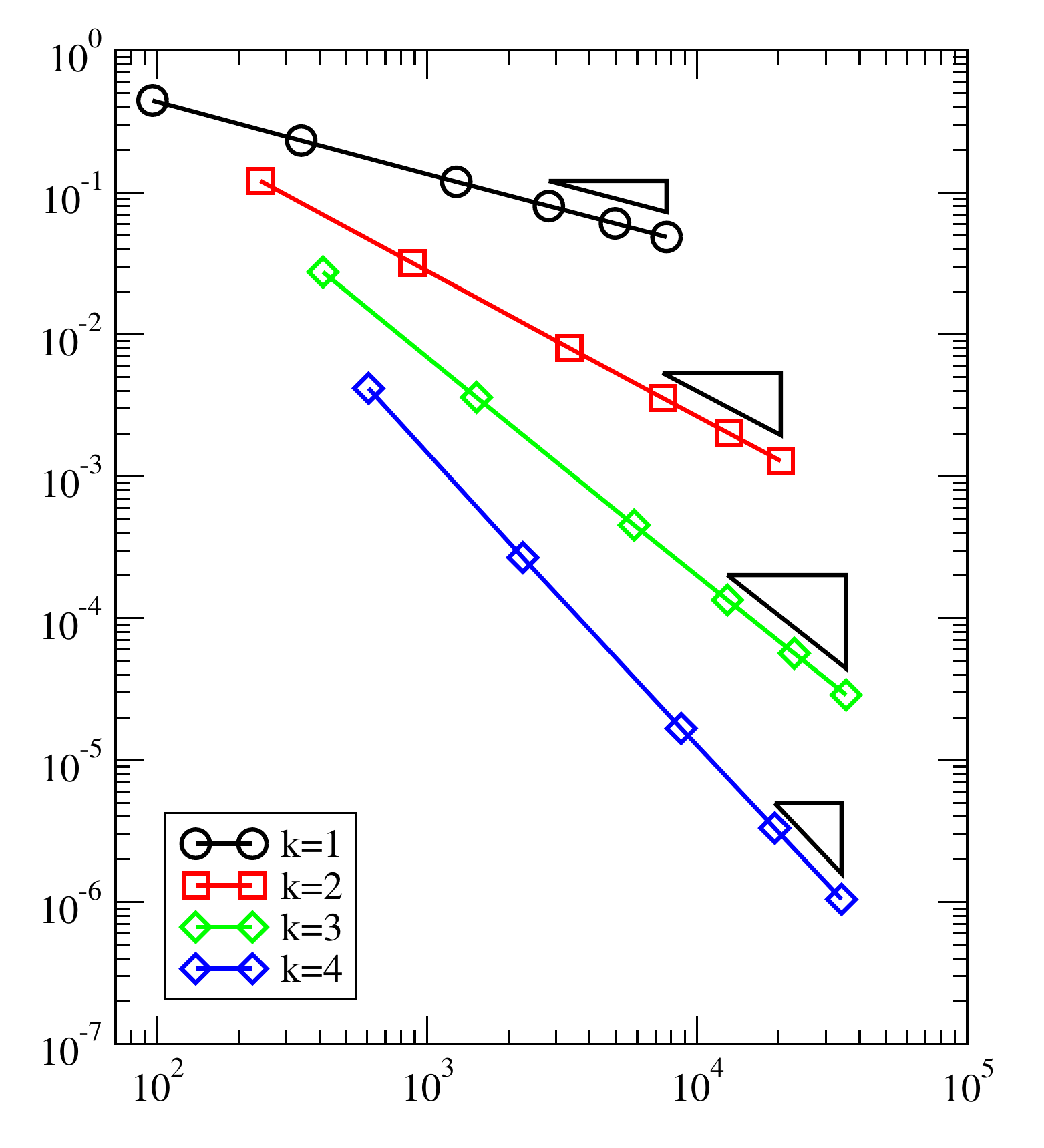}
      \put(-5,9){\begin{sideways}\textbf{$\mathbf{H^1}$ relative approximation error}\end{sideways}}
      \put(20,-2) {\textbf{\#degrees of freedom}}
      \put(49,85.5){\textbf{0.5}}
      \put(62,68.5){\textbf{1}}
      \put(65,51){\textbf{1.5}}
      \put(69,31){\textbf{2}}   
    \end{overpic}
    \\[0.5em]
  \end{tabular}
  \caption{Convergence plots for the virtual element approximation of
    Problem~\eqref{eq:pblm:strong:A}-\eqref{eq:pblm:strong:E} with
    exact solution~\eqref{eq:benchmark:solution} using
    family~\textit{Mesh~3} of nonconvex octagonal meshes.
    Error curves are computed using the $\LTWO$ norm (left panels) and
    $\HONE$ norm (right panels) and are plot versus
    the number of degrees of freedom.
  }  
  \label{fig:octa:rates}
\end{figure}

\begin{figure}
  \centering
  \begin{tabular}{cc}
    \begin{overpic}[scale=0.325]{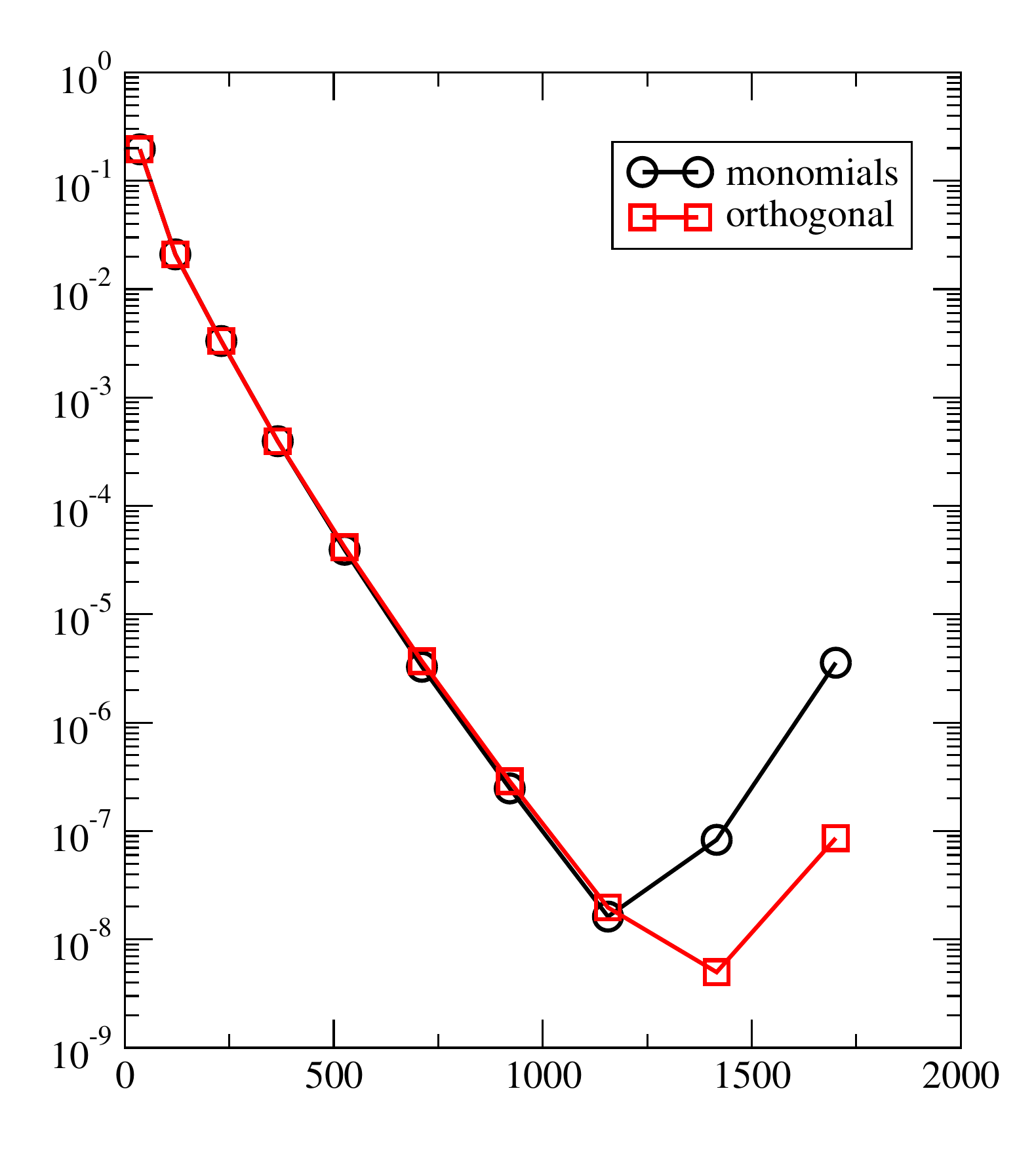}
      \put(-5,11){\begin{sideways}\textbf{$\mathbf{L^2}$ relative approximation error}\end{sideways}}
      \put(20,-2) {\textbf{\#degrees of freedom}}
    \end{overpic} 
    & \qquad
    \begin{overpic}[scale=0.325]{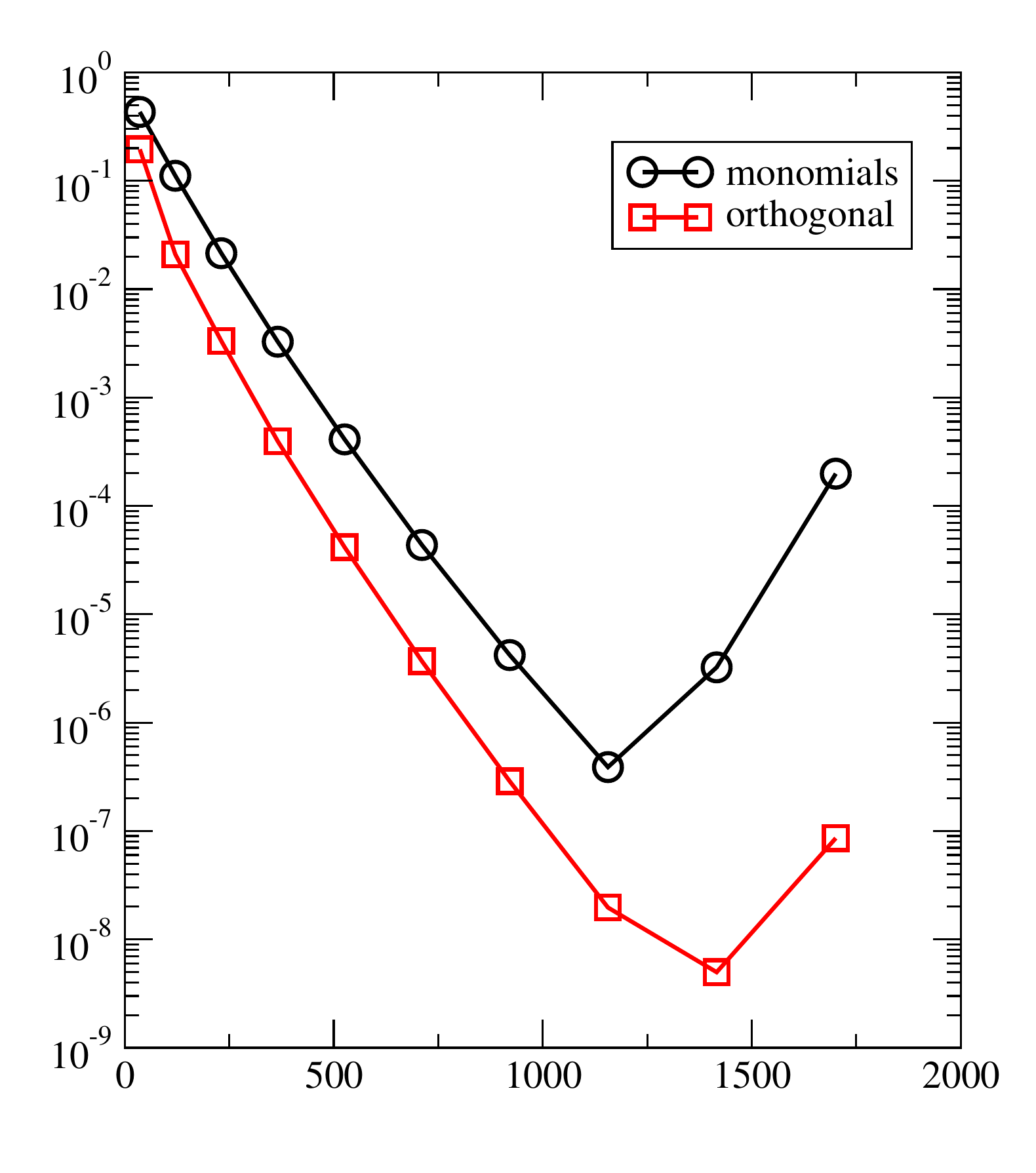}
      \put(-5,11){\begin{sideways}\textbf{$\mathbf{H^1}$ relative approximation error}\end{sideways}}
      \put(20,-2) {\textbf{\#degrees of freedom}}
    \end{overpic}
    \\[0.5em]
  \end{tabular}
  \caption{Convergence plots for the virtual element approximation of
    Problem~\eqref{eq:pblm:strong:A}-\eqref{eq:pblm:strong:E} with
    exact solution~\eqref{eq:benchmark:solution} using
    family~\textit{Mesh~1} of randomized quadrilateral meshes.
    Error curves are computed using k-refinement the $\LTWO$ norm
    (left panel) and $\HONE$ norm (right panel) and are plot versus
    the number of degrees of freedom by performing a refinement of
    type ``p'' on a $5\times 5$ mesh.
    Each plot shows the two convergence curves that are obtained using
    monomials (circles) and orthogonalized polynomials (squares.)}
  \label{fig:expo:rate}
\end{figure}

\subsection{Accuracy assessment using a manufactured solution}
\label{subsec51:manuf_soln}

In this section, we aim to confirm the optimal convergence rate of the
numerical approximation of the elastodynamic
problem~\eqref{eq:pblm:strong:A}-\eqref{eq:pblm:strong:E} provided by
the virtual element method in accordance with
Theorems~\ref{theorem:semi-discrete:convergence}
and~\ref{theorem:L2:convergence}.
In particular, we let $\Omega=(0,1)^2$ for $t\in[0,T]$, $T=1$, and
consider initial condition $\uv_{0}$, boundary condition $\gv$ and
forcing term $\fv$ determined from the exact solution:
\begin{align}
  \uv(x,y,t) = 
  \cos\left(\frac{2\pi\,t}{T}\right)
  \left(
    \begin{array}{c}
      \sin^2(\pi\xs)\sin(2\pi\ys)\\[0.5em]
      \sin(2\pi\xs)\sin^2(\pi\ys)
    \end{array}
  \right).
  \label{eq:benchmark:solution}
\end{align}
To this end, we consider three different mesh partitionings, denoted
by:
\begin{itemize}
\item \textit{Mesh~1}, randomized quadrilateral mesh;
\item \textit{Mesh~2}, mainly hexagonal mesh with continuously
  distorted cells;
\item \textit{Mesh~3}, nonconvex octagonal mesh.
\end{itemize}
The base mesh and the first refined mesh of each mesh sequence are
shown in Figure~\ref{fig:Meshes}.

These mesh sequences have been widely used in the mimetic finite
difference and virtual element
literature.\cite{BeiraodaVeiga-Lipnikov-Manzini:2011}
The discretization in time is given by applying the leap-frog method
with $\Delta t=10^{-4}$ and carried out for $10^{4}$ time cycles in
order to reach time $T=1$.

For these calculations, we used the VEM approximation based on the
conforming space $\Vhk$ with \RED{$k=1,2,3,4$} and the convergence
curves for the three mesh sequences above are reported in
Figures~\ref{fig:quads:rates},~\ref{fig:hexa:rates}
and~\ref{fig:octa:rates}.
The expected rate of convergence is shown in each panel by the
triangle closed to the error curve and indicated by an explicit label.
According to 
Theorem~\ref{theorem:semi-discrete:convergence}, we expect that the
approximation error decreases as $\mathcal{O}\big(\hh^{k}\big)$,
\RED{or $\mathcal{O}\big(\Ndofs^{-k/2}\big)$ in terms of the number of degrees
  of freedom that scale like $\Ndofs\approx\hh^{-2}$}
when using the virtual element method of order $k$ and measuring the
error in the $\HONE$ norm.
Consistently with our approximation, we also expect to see the
approximation error to decrease as $\mathcal{O}\big(\hh^{k+1}\big)$,
 \RED{or $\mathcal{O}\big(\Ndofs^{-(k+1)/2}\big)$, when using the
  $\LTWO$ norm.}
  
Furthermore, Figure~\ref{fig:expo:rate} shows the semilog error curves
obtained through a``p''-type refinement calculation for the previous
benchmark, i.e.\ for a fixed $5\times 5$ mesh of type $I$ the order of
the virtual element space is increased from $k=1$ to $k=10$.
We refer the reader to
Reference\cite{BeiraodaVeiga-Chernov-Mascotto-Russo:2016} for a
detailed presentation of the $p/hp$ virtual element formulations.
Here, we compare the performance of the method for two different
implementations.
In the first one, the space of polynomials of degree $k$ is generated
by the standard scaled monomials, while in the second one we consider
an orthogonal polynomial basis.
The impact that such a basis can have on the accuracy of the
high-order VEM has already been shown in the VEM
literature.\cite{Berrone-Borio:2017,Mascotto:2018,Dassi-Mascotto:2018}
The behavior of the VEM when using nonorthogonal and orthogonal
polynomials basis shown in Figure~\ref{fig:expo:rate} is in accordance
with the literature.

These plots confirm that the conforming VEM formulations proposed in
this work provide a numerical approximation with optimal convergence
rate on a set of representative mesh sequences, including deformed and
nonconvex polygonal cells.
\subsection{Dispersion and dissipation analysis}
\label{subsec52:diss-disp}
In this section we want to study the dispersion and dissipation errors
for the VEM presented before, by means of the Von Neumann analysis,
i.e., the propagation of a plane wave in a homogeneous medium.
For a plane wave, we recall that the \emph{numerical dispersion} is
the gap between the phase of the numerical solution and that of the
exact solution while \emph{numerical dissipation} represents a
decrease in the amplitude.

In order to carry out the dispersion and dissipation analysis, we make
some standard hypothesis for the following plane wave
analysis.~\cite{Cohen:2002,Ainsworth:2004a,Ainsworth:2004b}
Indeed, we suppose that the medium is isotropic, homogeneous,
unbounded and source-free.
The latter assumptions are, in general, not respected when considering
realistic geophysical applications. However, this analysis provides
important information about the discretization parameters to be used
in practice for the numerical simulation.

Consider equation \eqref{eq:pblm:strong:A} for an unbounded elastic
medium where no external forces are applied and seek for solutions in
the following form\cite{Eringen-Suhubi:1975}
\begin{equation}
  \uv(\xx,t) = \av e^{i(\kv\cdot\xv - \omega\ts)},
  \label{eq:pwave}		
\end{equation}
where $\av=[\as_1,\as_2]^T$ is the wave amplitude, $\omega$ the
angular frequency and $\kv=2\pi\slash{L}(\cos\theta,\sin\theta)$ the
wavenumber vector, with $L$ being the wavelength and $\theta$ the
angle between the wave propagation direction and the coordinate axes.
Then, by taking the real part of \eqref{eq:pwave} the physical wave is
recovered.
Under these conditions the semi discrete problem
\eqref{eq:VEM:semi-discrete} becomes
\begin{equation}
  \label{eq:semi-discAlgebric}
  \MM\ddot{\UU} + \KK\UU = \mathbf{0},
\end{equation}
where $\MM$ and $\KK$ are defined as in
Section~\ref{subsec:vem:implementation}.
To comply with the conditions of periodicity and unboundedness, we
solve problem \eqref{eq:semi-discAlgebric} over the domain $E_{ref}$
with uniform size $\hh$, and impose periodic boundary conditions on
its boundary, see Figure \ref{fig:periodic}.
The periodic reference domain $E_{ref}$ can be $(a)$ a square; $(b)$
the union of two triangles; $(c)$ the union of a set of polygons,
namely, four hexagons, two quadrilaterals and two pentagons; $(d)$ the
union of one square and four pentagons; $(e)$ four octagons; $(f)$ one
hexagon and four triangles.
\begin{figure}[t]
  \centering
  \begin{tabular}{ccc}
    \includegraphics[width=0.32\textwidth]{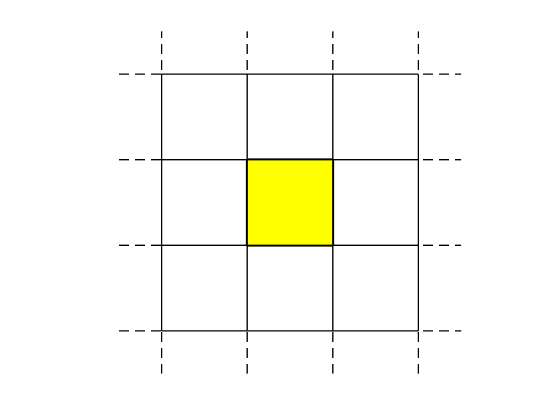} & 
    \includegraphics[width=0.32\textwidth]{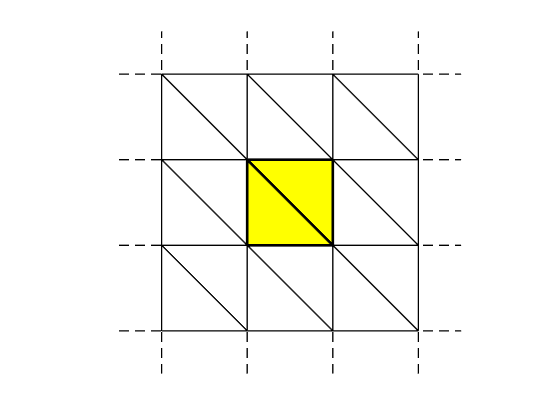} & 
    \includegraphics[width=0.32\textwidth]{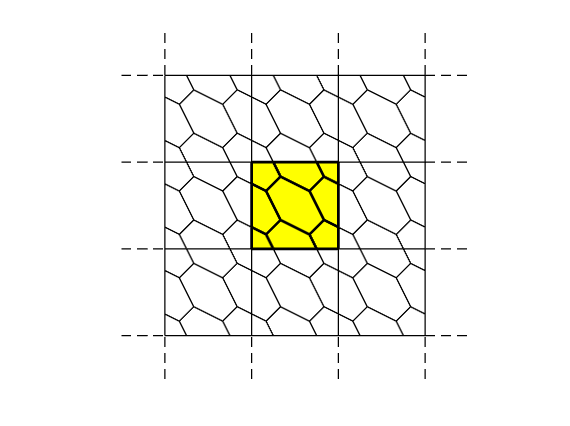} \\[-1em]
    \textit{Quadrilateral grid} & 
    \textit{Triangular grid}    &
    \textit{C1 grid}            \\[1em]
    \includegraphics[width=0.32\textwidth]{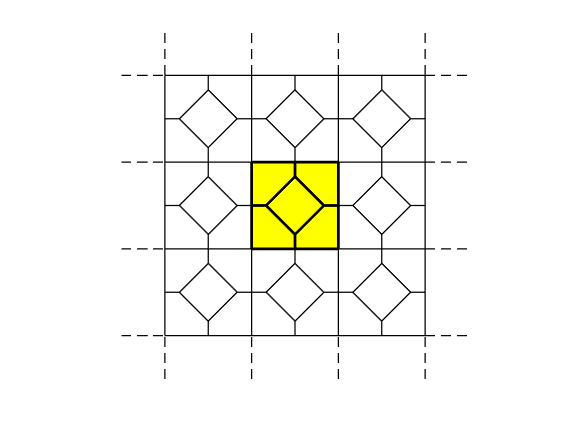} &
    \includegraphics[width=0.32\textwidth]{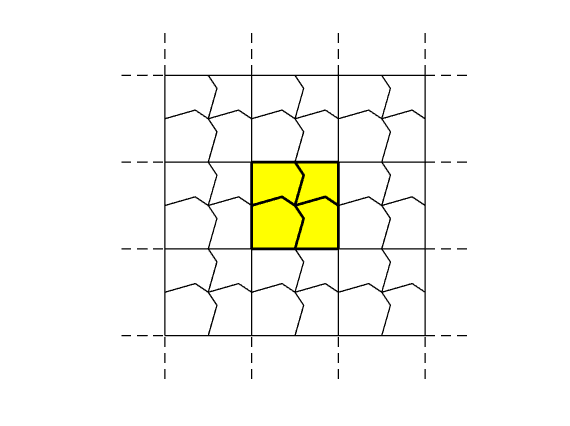} &
    \includegraphics[width=0.32\textwidth]{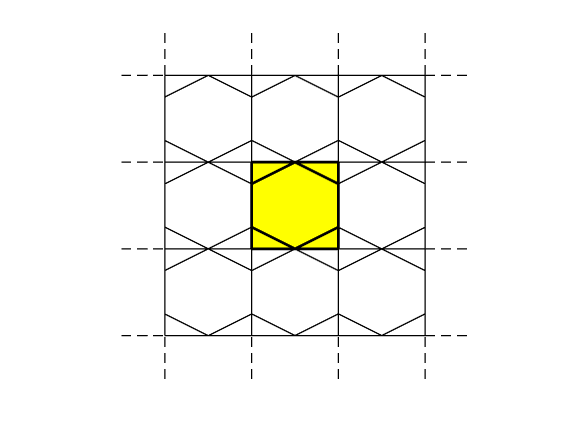} \\[-1em]
    \textit{C2 grid} &
    \textit{C3 grid} &
    \textit{C4 grid} 
  \end{tabular}
  \caption{Periodic reference element $E_{ref}$ (yellow) and periodic
    reference grids. Quadrilateral (a), triangular (b) and composite
    C1--C4 (c--f) grids.}
  \label{fig:periodic}
\end{figure}
To apply periodic boundary conditions on $E_{ref}$, we define the set
of ``master'' and ``slave'' indexes, respectively denoted by filled
and empty symbols in Figure~\ref{fig:dofs:diss-disp} for the case of
the virtual elements with polynomial degree $k=2$.
Master indexes refer to the degrees of freedom in which the solution
is computed and slave indexes to those where the periodicity
conditions are imposed.
Since the solution is a plane wave \eqref{eq:pwave}, we impose
periodic boundary conditions through a suitable projection matrix
$\mathrm{\Ps}$ and obtain the
system\cite{Antonietti-Mazzieri-Quarteroni-Rapetti:2012,mazzieri2012dispersion,ferroni2016dispersion,Antonietti-Mazzieri:2018}
\begin{equation} 
  \label{eq:eigs} 
  \calM\ddot{\Uv} + \calK\Uv = \mathbf{0},
\end{equation}
where $\calK$ and $\calM$ are given respectively by
$\calK=\mathrm{\Ps}^T\KK\Ps$ and $\calM=\mathrm{\Ps}^T\MM\Ps$.
\begin{figure}
  \centering
  \begin{tabular}{ccc}
    \includegraphics[width=0.32\textwidth]{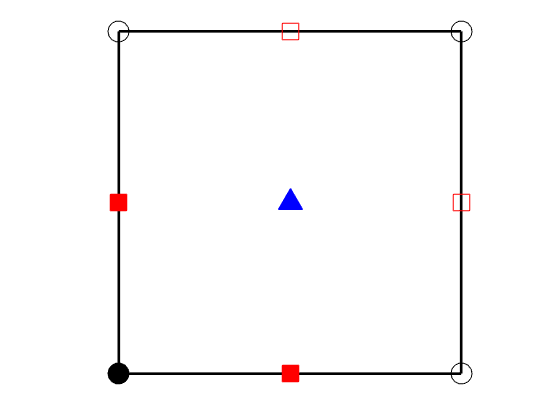} & 
    \includegraphics[width=0.32\textwidth]{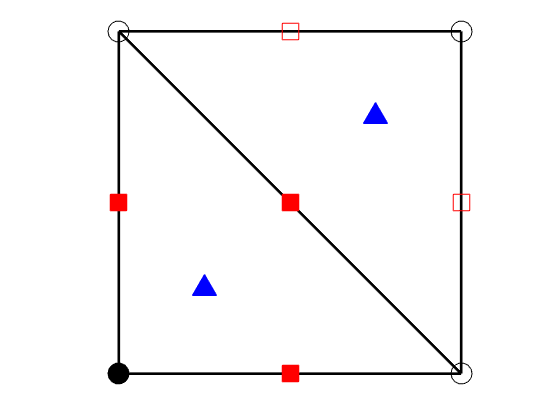} &
    \includegraphics[width=0.32\textwidth]{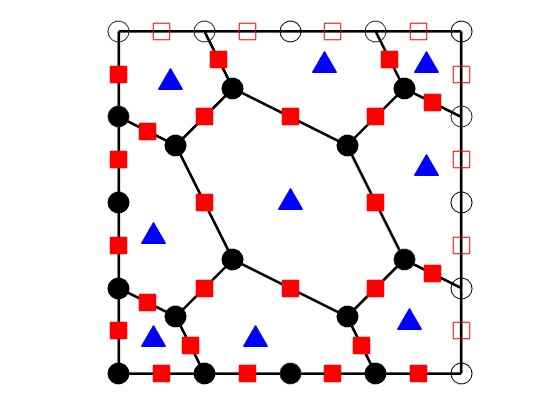} \\
    \textit{Quadrilateral grid} &
    \textit{Triangular grid} &
    \textit{C1 grid} \\[1.5em]
    \includegraphics[width=0.32\textwidth]{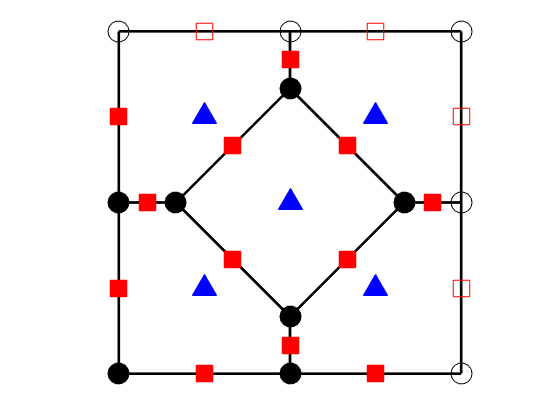} &
    \includegraphics[width=0.32\textwidth]{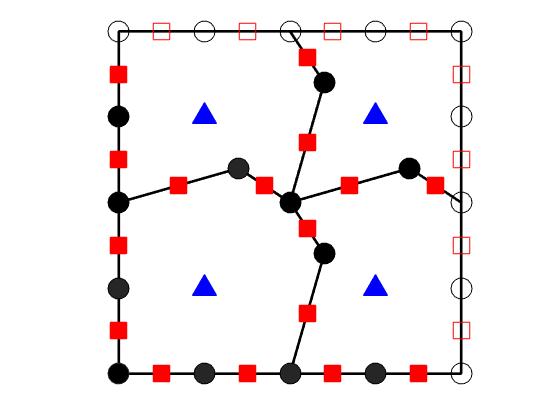} &
    \includegraphics[width=0.32\textwidth]{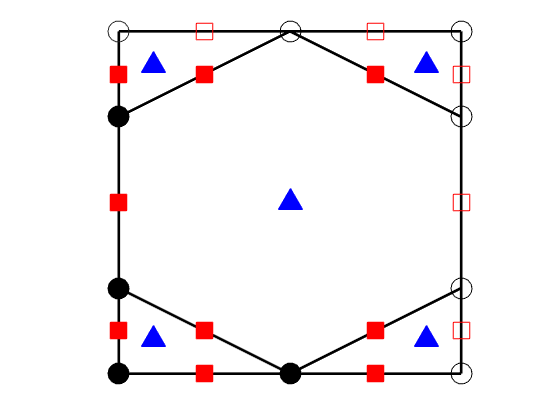} \\
    \textit{C2 grid} &
    \textit{C3 grid} &
    \textit{C4 grid} 
  \end{tabular}
  \caption{Degrees of freedom on the reference element $E_{ref}$ of
    the local virtual element space with polynomial degree $k=2$.
    Quadrilateral $(a)$, triangular $(b)$, and composite $(c)$--$(f)$
    grids: where periodic boundary conditions are assigned, the
    degrees of freedom are represented by empty squares or circles.}
  \label{fig:dofs:diss-disp} 
\end{figure}
Substitute\cite{Alford-Kelly-Boore:1974,Antonietti-Mazzieri:2018,ferroni2016dispersion}
\eqref{eq:pwave} into \eqref{eq:eigs}, discretize the system by means
of the leap-frog scheme \eqref{eq:VEM:fully-discrete} and obtain a
generalized eigenvalue problem of the form
\begin{align}
  \KK\Uv_0 = \Lambda\MM\Uv_0,
  \label{eq:gen_eig_fully}
\end{align} 
where the eigenvalues $\Lambda$ can be expressed in term of the
numerical angular frequency $\omega$ through the relation
\begin{align*}
  \Lambda= \frac{4}{\Delta t^2} \sin^2(\omega \frac{\Delta
    t}{2}).
\end{align*}
We use this equation after solving the eigenvalue problem
\eqref{eq:gen_eig_fully} in order to derive the grid-dispersion
relations as it will be shown later on.
We remark that for two dimensional elastic wave propagation only two
eigenvalues in \eqref{eq:gen_eig_fully} have a physical meaning as
they are related to P and S waves.\cite{DeBasabe-Sen:2007}
All the other eigenvalues correspond to nonphysical
modes.\cite{Cohen:2002}
Therefore, the relative dispersion errors are given by
\begin{align}\label{def:disp_errors}
  \es_P=\frac{c_{P,\hh}}{c_P}-1,
  \qquad
  \es_S=\frac{c_{S,\hh}}{c_S}-1,
\end{align}
where the numerical wave velocities $c_{P,\hh}$ and $c_{S,\hh}$ are
therefore given by
\begin{align*}
  c_{P,\hh}=\frac{\hh\,\omega_{P,\hh}}{2\pi\delta r}, 
  \qquad 
  c_{S,\hh}=\frac{\hh\,\omega_{S,\hh}}{2\pi\delta},
\end{align*}
where $\delta=\hh/(\ks\Ls)$ represents the sampling ratio, i.e. \RED{$\delta^{-1}$ is} the
number of interpolation points per wavelength, and $\Ls$ is the
wavelength.

Now, we study the stability properties of the leap-frog time-stepping
scheme and then we analyze the dispersion errors introduced by the
numerical formulation.
Since the leap-frog method is an explicit time integrator, the time
step $\Delta t$ must satisfy the Courant, Friedrichs, Lewy (CFL)
condition
\begin{align} 
  \label{eq:CFL} 
  \Delta t \leq C_{CFL} \frac{\hh}{c_P},
\end{align}
where we recall that $\hh$ is the meshsize and $c_P$ is the P-wave
velocity and we need to impose that $C_{CFL}\leq 1$.
We are interested in studying how the $C_{CFL}$ factor may depend on
the physical parameters considered in the model, i.e., $\lambda$ and
$\mu$, and the polynomial degree $k$.

\medskip
With this aim, we consider \eqref{eq:gen_eig_fully} and rewrite
matrices $\KK$ and $\MM$ by using a scaling argument, so that
\begin{align} 
  \widetilde{\KK}\Uv_0 = \Lambda' \widetilde{\MM} \Uv_0,
  \label{eq:eiggen_ref} 
\end{align}
where $\Lambda' = (\hh/\Delta t)^2 \sin^2(\omega_{\hh}\Delta t/2)$. 
Then, we introduce the stability parameter $\qs=c_P\frac{\Delta t}{\hh}$
for which it holds $ \qs^2 \Lambda' = c_P^2
\sin^2(\frac{\omega_{\hh}\Delta t}{2})\leq c_P^2 $, that is equivalent
to
\begin{align} 
  \qs\leq \frac{c_P}{\sqrt{\Lambda'}}= C_{CFL}(\Lambda').
  \label{eq:stability1}
\end{align}
The eigenvalue $\Lambda'$ depends on the wavenumber vector $\kk$ and,
therefore, on the values of the angle
$\theta$.\cite{DeBasabe-Sen:2010}
Thus, condition~\eqref{eq:stability1} is equivalent to
\begin{equation}
  \qs \leq \cs^*(\lambda,\mu) \frac{1}{\sqrt{\Lambda_{max} '}} = \qs_{CFL},
  \label{eq:q_clf}
\end{equation}
where $\Lambda_{max}'$ is the maximum eigenvalue of
\eqref{eq:eiggen_ref}, taken with respect to the value of $\theta$.

Constant $\cs^*$ depends on the Lam\'e parameters $\lambda$ and
$\mu$.\cite{Antonietti-Houston:2011}

We notice that a different approach (that does not make use of a plane
wave analysis) is employed in Park et al. (2019)\cite{Park2019} to
estimate $q_{CFL}$ for VEM approximation on arbitrary mesh.
The authors were there interested in establishing a rule of thumb for
choosing the discretization parameters that satisfy \eqref{eq:q_clf},
here the analysis is mainly focused on the dispersion and dissipation
property of the scheme, provided that condition \eqref{eq:q_clf} holds
true.
\BLUE{Finally, we remark that the present analysis is valid for mesh
  configurations that have a periodic pattern.  For a general
  polygonal mesh, made by elements of arbitrary shape, one should
  consider the associated covering mesh made by rectangular elements
  and check that, for the latter grid, a sufficient number of points
  per wavelength is employed. This can be inferred easily by applying
  the rule of thumb that will be given for quadrilateral elements.}

\subsubsection{Numerical dispersion and dissipation analysis: space discretization}

We start the section by analyzing the dispersion and dissipation
properties of the VEM by assuming exact time integration.
For this case, we consider equation \eqref{eq:eigs} and compute
analytically the second time derivative of the displacement vector.
In this way, we can obtain a generalized eigenvalue problem whose
eigenvalues $\Lambda$ can be written in term of the angular frequency
$\omega$ through the formula $\Lambda = \omega^2$.
Grid dispersion errors are computed according
to~\eqref{def:disp_errors}: (i) we solve the generalized eigenvalue
problem numerically and obtain the eigenvalues $\xi=\omega_{\hh}^2$
that represent the best approximations of the angular frequencies of
the travelling waves; (ii) we identify the numerical eigenvalues
$\xi_P$ and $\xi_S$, that correspond to the physical frequencies of
the longitudinal and transversal displacement; (iii) we compute the
numerical angular frequencies $\omega_{P,h}=\sqrt{\xi_P}$ and
$\omega_{S,h}=\sqrt{\xi_S}$ for the P-wave and the S-wave,
respectively.

We now study the dispersion errors as a function of the polynomial
degree $k$, the sampling ratio $\delta$ for the shortest wavelength,
the wavenumber vector $\mathbf{k}$ of the plane wave, i.e., the angle
$\theta$, and the ratio $r=c_P/c_S$.
In this first test case, we set $\delta=0.2$, i.e., five nodes per
shortest wavelength, $\theta= \pi/4$ and $r=2$.
Different choices of the angle $\theta$ provide similar results as it
is discussed later on.
For all the considered grids, an exponential decay of the dispersion
error with respect to the polynomial degree $k$ can be observed,
cf. Figure \ref{fig:dispN}.
In particular for $k\geq 4$ the dispersion errors are below the
threshold value $10^{-6}$ for all the considered grids, with the
composite grid C1 providing the best results.

\begin{figure}
  \hspace{-1cm}
  \includegraphics[width=1.1\textwidth]{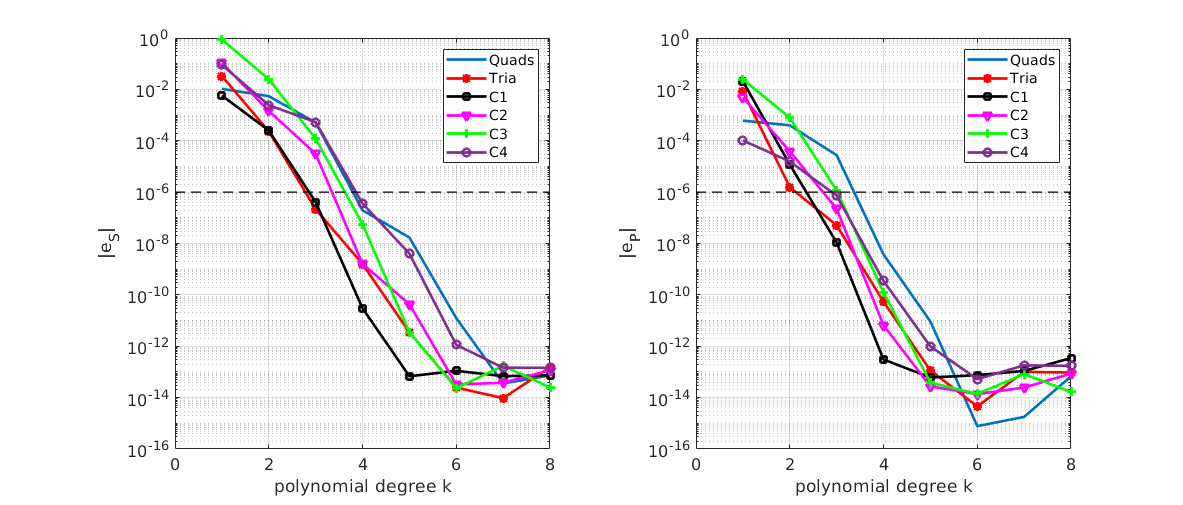}
  \caption{Computed dispersion errors $|e_S|$ and $|e_P|$ as a function
    of $k$ for $\delta=0.2$, $\theta=\pi/4$, $r=2$ and different mesh
    configuration. The threshold value $10^{-6}$ is also reported with a
    black dotted line.}
  \label{fig:dispN}
\end{figure}
However, if we look at the same results as a function of total number
of degrees of freedom in Figure~\ref{fig:disp_dofs} we see that the
quadrilateral and triangular meshes outperform the other grid
configurations.
This result is important from the computational point of view and
should be taken into account for realistic applications.
For completeness, in Table \ref{Dispersionall} we show the dispersion
errors for $r=2,5,10$ computed on the quadrilateral and triangular
grid.
Note that the same exponential decay of the dispersion error is
observed for $r=5$ and $r=10$.
Moreover, on such grid and in the variation range of $r$, the
dispersion errors are negligible, i.e., less than $10^{-6}$, when we
use a polynomial approximation degree $k>4$.
We obtained a similar behavior on the other grids (C1-C4), but we do
not report the results here for the sake of brevity.
\begin{figure}
  \includegraphics[width=1\textwidth]{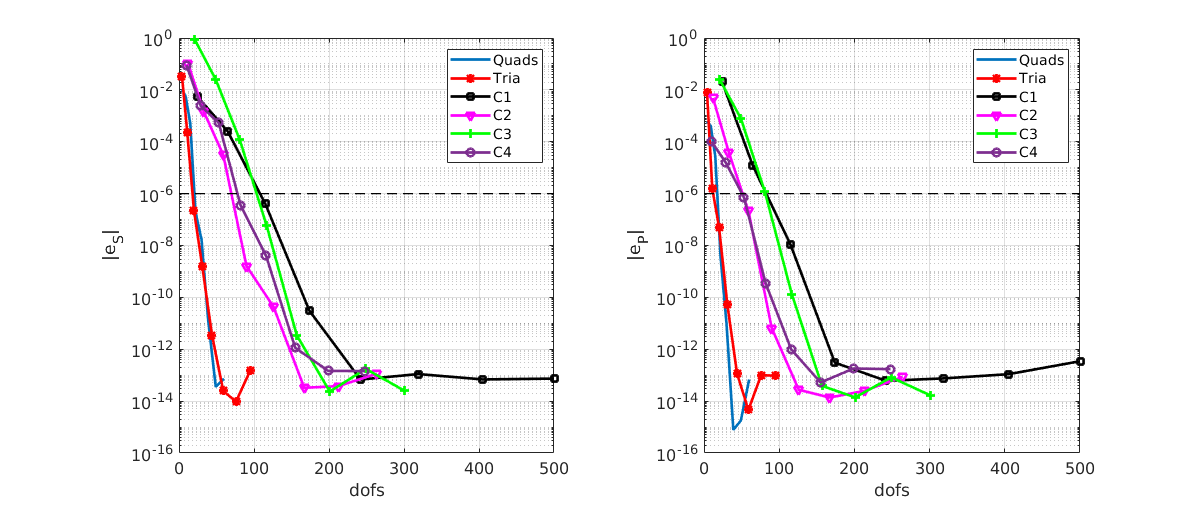}
  \caption{Computed dispersion errors $|e_S|$ and $|e_P|$ as a
    function of $dofs$ for $\delta=0.2$, $\theta=\pi/4$, $r=2$ and
    different mesh configuration. The threshold value $10^{-6}$ is
    also reported with a black dotted line.}
  \label{fig:disp_dofs}
\end{figure}
\newcommand{\TABROW}[8]{ \textrm{#1} & \quad$#2$ & \quad$#3$ & \quad$#4$ & \quad$#5$ & \quad$#6$ & \quad$#7$ & \quad$#8$ \\}
\renewcommand{\arraystretch}{1.}
\begin{table}
  \centering
  \caption{Dispersion errors $|e_P|$ and $|e_S|$ as a function of $k$ with $\delta=0.2$, $\theta=\pi/4$.}
  \label{Dispersionall}
  \begin{tabular}{@{}clccccccccc}
    \TABROW{     }{   }{ r = 2            }{ r = 2            }{ r = 5            }{ r = 5            }{ r=10             }{ r = 10           }\hline
    \TABROW{Grid }{ k }{ \ABS{\es_P}      }{ \RED{\ABS{\es_S}} }{ \ABS{\es_P}      }{ \ABS{\es_S}      }{ \RED{\ABS{\es_P}}}{ \ABS{e_S}        }\hline
    \TABROW{     }{ 1 }{ 6.0577e-04       }{ 1.0492e-02       }{ 1.5137e-05       }{ 1.0492e-02       }{ 9.4273e-07       }{ 1.0492e-02       }
    \TABROW{     }{ 2 }{ 3.9826e-04       }{ 5.4411e-03       }{ 9.7574e-06       }{ 1.6753e-03       }{ 6.0481e-07       }{ 1.5893e-03       }
    \TABROW{     }{ 3 }{ 2.7291e-05       }{ 5.1136e-04       }{ 6.0824e-08       }{ 6.3829e-05       }{ 8.6845e-10       }{ 6.0938e-05       }
    \TABROW{Quads}{ 4 }{ {3.7419e-09} }{ {1.9706e-07} }{ {3.8492e-13} }{ 1.0884e-06       }{ {6.8826e-14} }{ 1.0457e-06       }
    \TABROW{     }{ 5 }{ 9.1927e-12       }{ 1.7140e-08       }{ 8.5619e-14       }{ {1.2832e-08} }{ 1.9577e-13       }{ {1.1718e-08} }
    \TABROW{     }{ 6 }{ 4.4075e-15       }{ 1.2488e-11       }{ 1.2882e-13       }{ 1.2736e-10       }{ 1.2404e-12       }{ 1.0704e-10       }
    \TABROW{     }{ 7 }{ 2.9477e-14       }{ 3.6223e-14       }{ 3.0194e-13       }{ 5.1265e-13       }{ 8.5984e-13       }{ 4.5778e-13       }
    \TABROW{     }{ 8 }{ 1.0372e-13       }{ 7.6169e-14       }{ 4.6899e-13       }{ 1.5753e-13       }{ 8.5753e-13       }{ 9.5340e-13       }\hline
    \TABROW{     }{ 1 }{ 8.2437e-03       }{ 3.3135e-02       }{ 1.3165e-03       }{ 3.3135e-02       }{ 3.2902e-04       }{ 3.3135e-02       }
    \TABROW{     }{ 2 }{ 1.5669e-06       }{ 2.3285e-04       }{ 5.2817e-08       }{ 1.8368e-04       }{ 3.3775e-09       }{ 1.7914e-04       }
    \TABROW{     }{ 3 }{ {5.0853e-08} }{ {2.1675e-07} }{ {1.0534e-10} }{ 3.1105e-06       }{ {1.8894e-12} }{ 2.4701e-06       }
    \TABROW{Tria }{ 4 }{ 5.4010e-11       }{ 1.5845e-09       }{ 8.6515e-14       }{ {4.7236e-09} }{ 2.2625e-13       }{ {4.2769e-09} }
    \TABROW{     }{ 5 }{ 1.1938e-13       }{ 3.5012e-12       }{ 1.5574e-13       }{ 3.2109e-11       }{ 7.6318e-13       }{ 1.5237e-11       }
    \TABROW{     }{ 6 }{ 1.8481e-14       }{ 2.6318e-14       }{ 2.4079e-13       }{ 3.2615e-13       }{ 1.2960e-12       }{ 1.3384e-12       }
    \TABROW{     }{ 7 }{ 1.0408e-13       }{ 7.4694e-14       }{ 5.2719e-13       }{ 7.0126e-13       }{ 4.8367e-13       }{ 5.6891e-13       }
    \TABROW{     }{ 8 }{ 1.0693e-13       }{ 1.5421e-13       }{ 9.3388e-13       }{ 5.4632e-13       }{ 1.1880e-12       }{ 6.9274e-13       }
  \end{tabular}
\end{table}  
\renewcommand{\arraystretch}{1.}
In Figures \ref{fig:dispdelta:a}, \ref{fig:dispdelta:b}, and
\ref{fig:dispdelta:c}, we report the computed dispersion versus
$\delta$ for different approximation degrees $k=2,3,4$ and $r=2$.
Similar results can be obtained for $r=5,10$ but we do not report them for
the sake of brevity.
By computing the slopes of the $e_S$- and $e_P$-curves, we can obtain
the following empirical estimate of the orders of convergence:
$\ABS{\es_P}=\ABS{\es_S}=\calO(\hh^{2k})$ for the VEM with odd
polynomial degrees and $\ABS{\es_P}=\ABS{\es_S}=\calO(\hh^{2k-1})$ for
the VEM with even polynomial degree.
These results are in agreement with previous quantitative estimates
quadrilateral
elements.\cite{DeBasabe-Sen-Wheeler:2008,mazzieri2012dispersion}
\textcolor{red}{From the reported result we can infer that with a polynomial degree
$k=4$ and $\delta=0.2$, i.e., $5$ points per shortest wavelength, the
dispersion error is lower than the threshold $10^{-6}$ for all mesh
elements (see Figure~\ref{fig:dispdelta:c})}.
\begin{figure}
  \includegraphics[width=1.\textwidth]{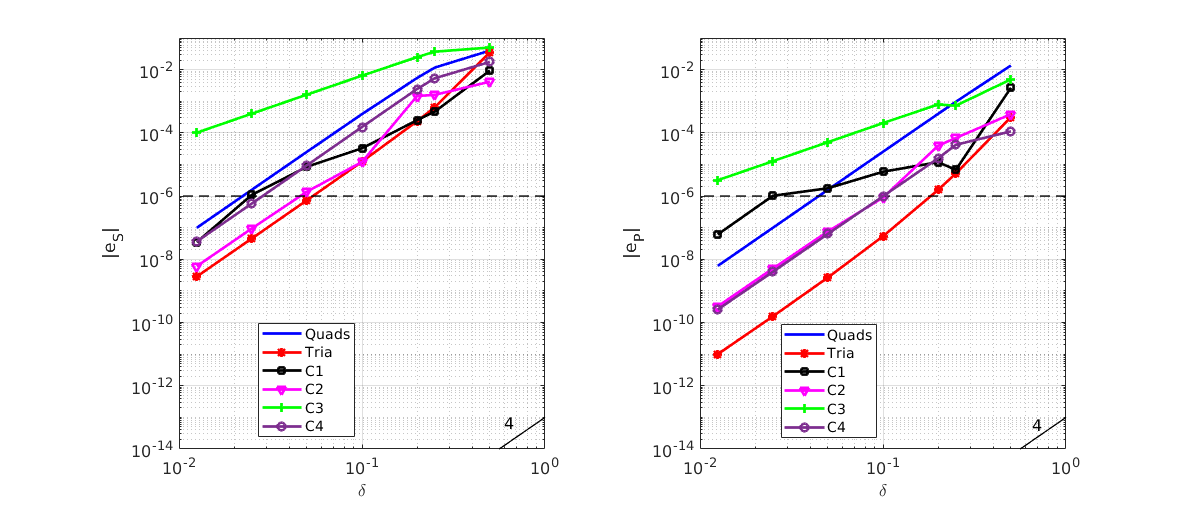}
  \caption{Computed dispersion errors as a function of $\delta$ with
    $k=2$ for all grid configurations.
    The threshold value $10^{-6}$ is also reported with a black dotted
    line.}
  \label{fig:dispdelta:a}
\end{figure}

\begin{figure}
  \centering
  \includegraphics[width=1.\textwidth]{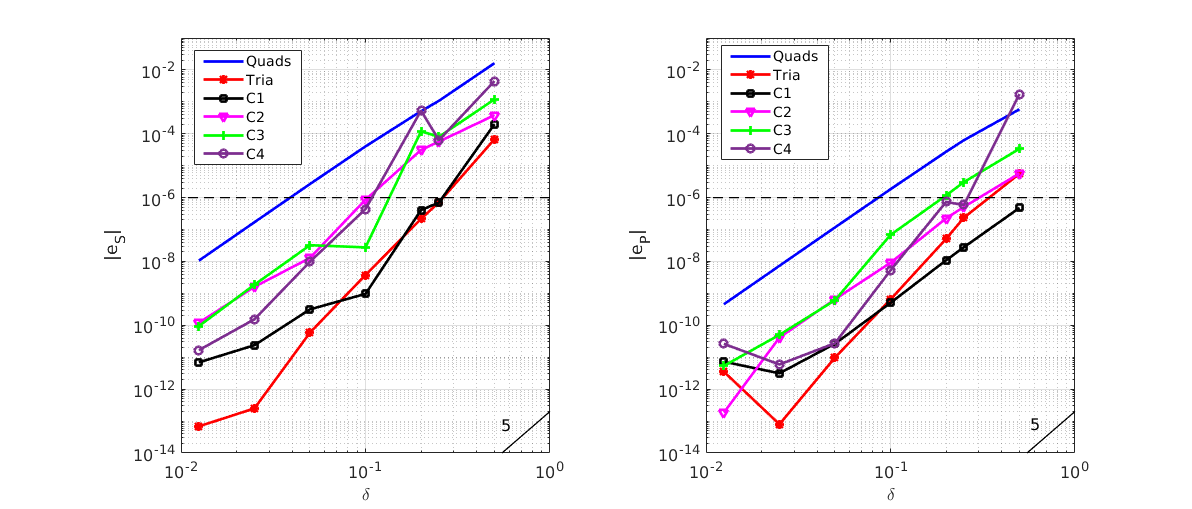}
  \caption{Computed dispersion errors versus $\delta$ with $k=3$ for
    all grid configurations.
    The threshold value $10^{-6}$ is also reported with a black dotted
    line.}
  \label{fig:dispdelta:b}
\end{figure}

\begin{figure}
  \centering
  \includegraphics[width=1.\textwidth]{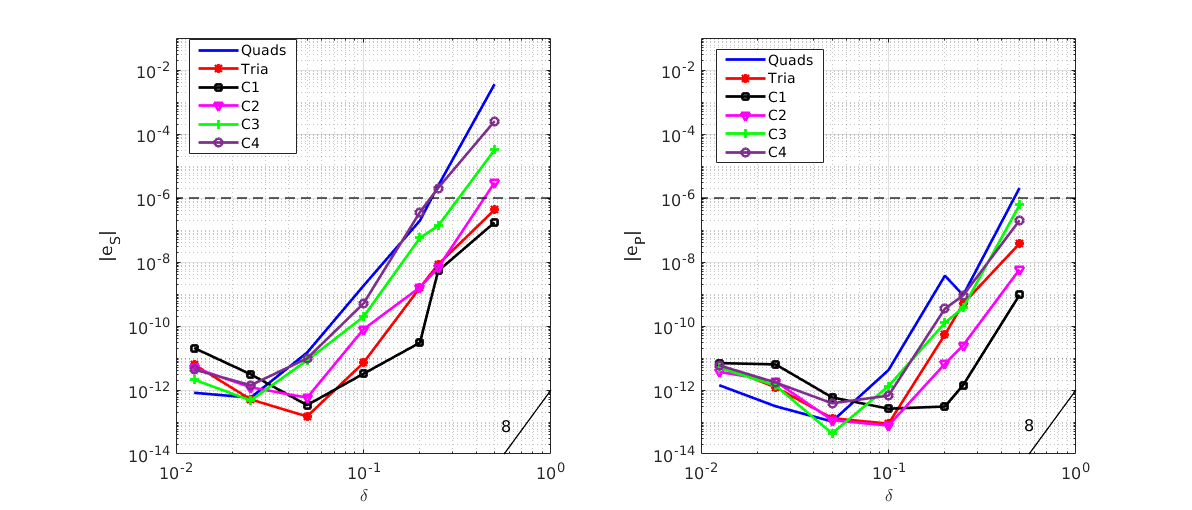}
  \caption{Computed dispersion errors versus $\delta$ with $k=4$ for
    all grid configurations.
    The threshold value $10^{-6}$ is also reported with a black dotted
    line.}
  \label{fig:dispdelta:c}
\end{figure}
Finally, we study the dispersion errors as a function of the angle
$\theta$ of the plane wave.
In Figure \ref{fig:anisotropy} we show the results obtained with the
virtual element discretization for polynomial degree $k=4$, the
sampling ratio $\delta=0.5$ and $r=2$. 
Anisotropy ratios $c_{P,h}/c_P$ and $c_{S,h}/c_S$ have been computed
with the same set of values, see Figure \ref{fig:anisotropy}.
Notice that with quadrilateral and C4 grids the error is symmetric
with respect to the origin of the axes, whereas with triangular grids
the error grows along the direction in which the periodic cell
$E_{ref}$ is split into triangles.
This grid orientation effect is in agreement with previous results
available in the
literature.\cite{Antonietti-Mazzieri-Quarteroni-Rapetti:2012,mazzieri2012dispersion,Antonietti-Mazzieri:2018}
Composite grids C1, C2 and C3 seem to perform better than other
configurations with respect to all incident angles: this fact can be
traced back to the asymmetrical pattern inside the reference cell
$E_{ref}$.
\begin{figure}
  \centering
  \subfigure[Quadrilateral grid]{\includegraphics[width=0.48\textwidth]{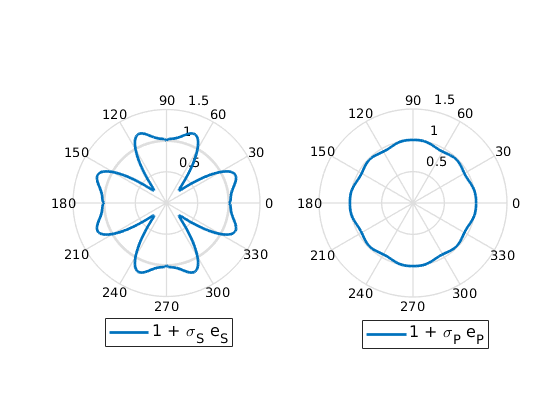}}
  \subfigure[Triangular grid   ]{\includegraphics[width=0.48\textwidth]{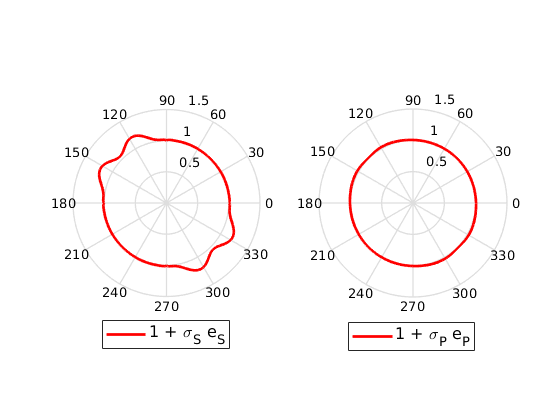}}
  \\
  \subfigure[C1 grid]{\includegraphics[width=0.48\textwidth]{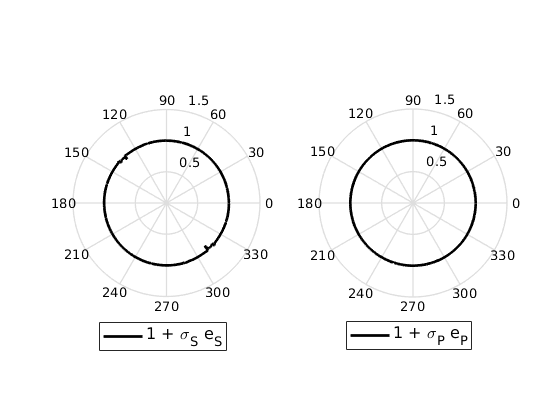}}
  \subfigure[C2 grid]{\includegraphics[width=0.48\textwidth]{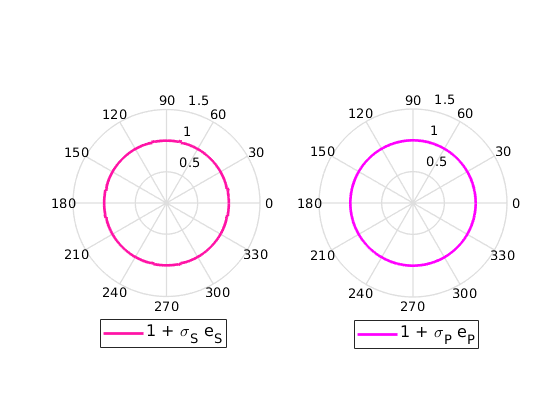}}
  \\
  \subfigure[C3 grid]{\includegraphics[width=0.48\textwidth]{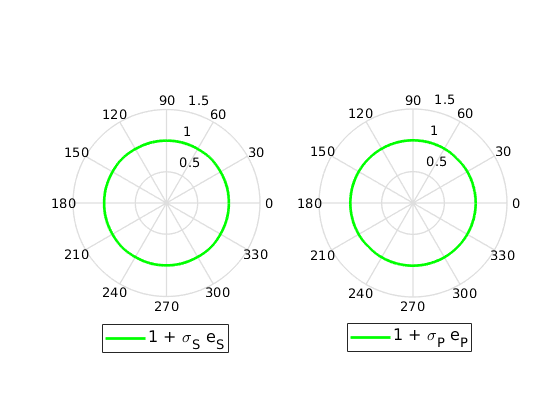}}
  \subfigure[C4 grid]{\includegraphics[width=0.48\textwidth]{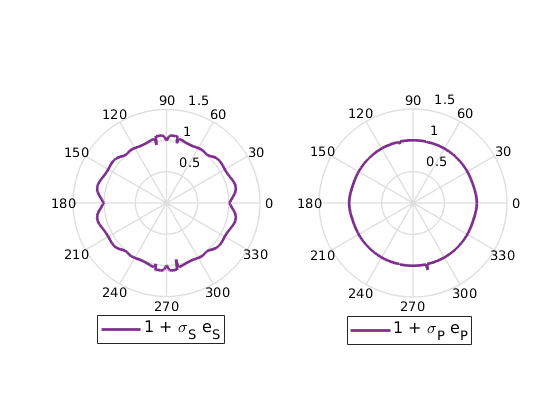}}
  \caption{Anisotropy ratios $c_{S,h}/c_S$ (left) and $c_{P,h}/c_P$
    (right) as a function of the incidence angles $\theta$ for $k=4$
    and $\delta=0.5$ for all grid configurations.
    For visualization purposes dispersion errors are magnified by
    factors $\sigma_S=200$ and $\sigma_P = 10^4$.}
  \label{fig:anisotropy}
\end{figure}
Concerning the dissipation error we study the amplitude of the
numerical displacement and take the plane wave
$\uv(\xx,t)=e^{i(\kv\cdot\xv-\omega\ts)}$ as exact solution of
\eqref{eq:pblm:strong:A}.
On the one hand, since ${\rm Im}(\omega) = 0$, its amplitude is equal
to $1$ for all times $t$.
On the other hand, the numerical approximation will give in general
${\rm Im}(\omega_h) \neq 0$.
Then, the scheme is non-dissipative if ${\rm Im}(\omega_h)=0$, whereas
it is dissipative if ${\rm Im}(\omega_h)<0$.
In our case, since the mass and the stiffness matrices appearing in
\eqref{eq:eiggen_ref} are symmetric and positive definite, the
computed eigenvalues are all real. This leads to non dissipative
schemes.

\subsubsection{Numerical dispersion and dissipation analysis: space-time
  discretization}
We present first some numerical tests for assessing a quantitative
estimate of the parameter $q_{CFL}$.
In the following, we consider $\cs_P=\sqrt{2}$, $r=2$,
$\theta\in(0,2\pi)$ and $\delta=0.2$.
Analogous conclusions can be drawn for $r=5$ and $r=10$.

\medskip
Figure \ref{fig:cfl} shows that the decay rate of $q_{CFL}$ is
approximately proportional to $k^{-3/2}$ for all the mesh element
configurations.
This result appears to be slightly better than the one obtained in
previous
works.~\cite{DeBasabe-Sen:2010,Antonietti-Mazzieri-Quarteroni-Rapetti:2012,Antonietti-Marcati-Mazzieri-Quarteroni:2016,ferroni2016dispersion,Antonietti-Mazzieri:2018}
In addition, for a given polynomial degree, quadrilateral and C1 grid
elements are subjected to a slightly more restrictive stability
condition, i.e., lower values of $q_{CFL}$ are obtained.
\begin{figure}
  \centering
  \includegraphics[width=0.5\textwidth]{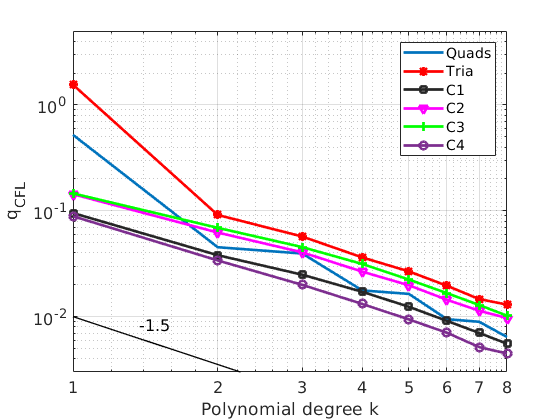}
  \caption{Stability parameter $q_{CFL}$ as a function of the
    polynomial degree $k$ for all grid configurations.}
  \label{fig:cfl}
\end{figure}
Next, we present the results of the dispersion and dissipation
analysis for the fully discrete approximation by varying the
parameters $k$, $\delta$ and the stability parameter $q$.
For the sake of conciseness, we will show only the results related to
 \RED{the triangular and the composite  C3 grids}.
Similar results have been obtained with the \RED{quadrilateral or
  composite C1, C2 and C4 meshes}.
First, we address the behavior of the dispersion error by varying the
sampling ratio $\delta$ with fix values of $k=4$ and $\theta=\pi/4$.
The relative stability parameter $\qs_{rel} = \qs\slash{\qs_{CFL}}$ is
in the range $(0.2,1)$, and $q_{CFL}$ is computed in agreement with
\eqref{eq:q_clf}.
In practice, we first compute $\qs_{CFL}$ (see \eqref{eq:q_clf}) for
any specific values $\delta$, cf. also Figure~\ref{fig:cfl}; then, we
choose $\qs\in (0,q_{CFL})$ and finally we sample the values of the
dispersion errors $\es_S$ and $\es_P$ as a function of $\qs$.
As $\Delta t$ (and then $\qs_{rel}$) goes to $0$, the fully discrete
curves tend to the semi-discrete ones, see Figure
\ref{fig:fullydisp_delta}.
In Figure~\ref{fig:fullydisp_deltadt001}, we compare the results
obtained with all mesh element configurations for $\qs_{rel}=0.2$.
On all the considered grids, the VEM reaches the same level of
accuracy, and, in particular $\delta \leq 0.2$, i.e., five points per
wavelength, are sufficient to obtain dispersion errors lower than
$10^{-6}$.
\begin{figure}[t]
  \hspace{-0.75cm}
  \begin{tabular}{c}
    \includegraphics[width=1.1\textwidth]{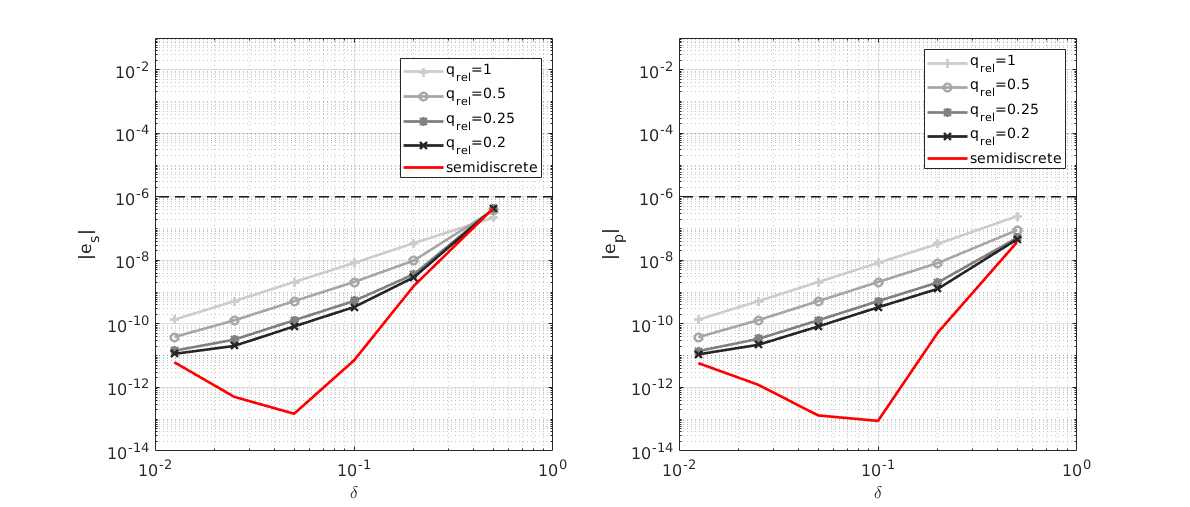}\\
    \textit{Triangular mesh} \\    
    \includegraphics[width=1.1\textwidth]{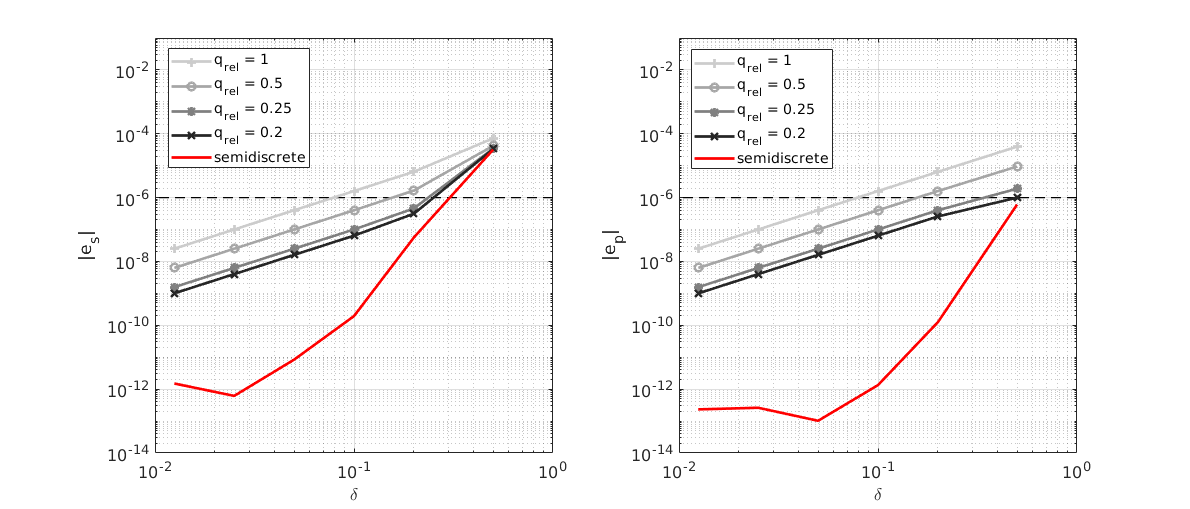} \\
    \textit{\RED{C3 mesh}}
  \end{tabular}
  \caption{Computed dispersion errors $\ABS{\es_S}$ (left) and
    $\ABS{\es_P}$ (right) as a function of $\delta$ using
    \RED{triangular} (a) and \RED{composite C3 }(b) grids, with $k=4$.
    The continuous red lines refer to the semi discrete approximation,
    while the others to the fully discrete approximation with
    $\qs_{rel}=0.2,0.25,0.5,1$.}
  \label{fig:fullydisp_delta}
\end{figure}
\begin{figure}[t]
  \centering
  \includegraphics[width=1.1\textwidth]{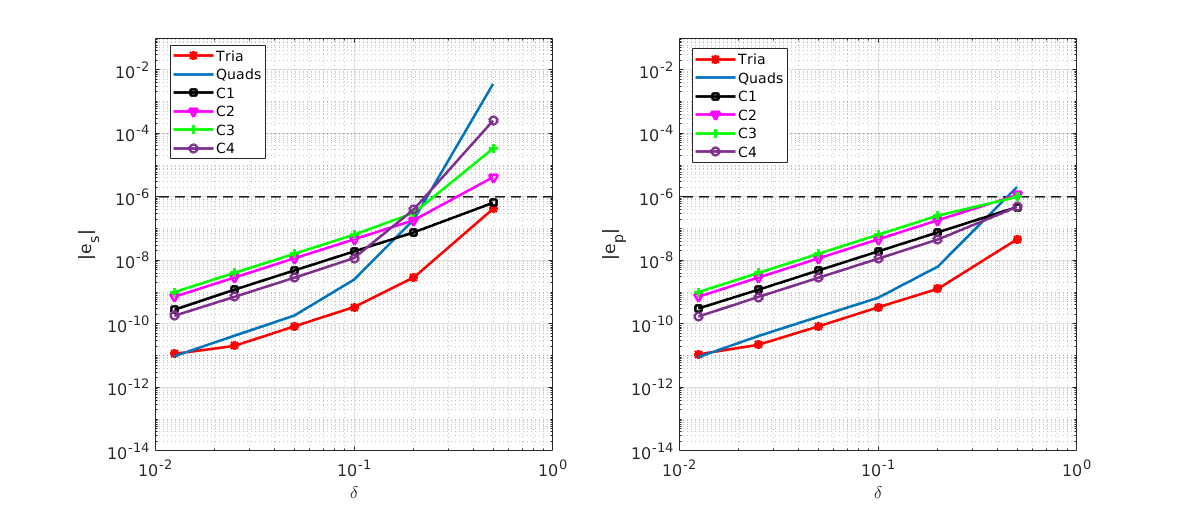}
  \caption{Computed dispersion errors $\ABS{\es_S}$ (left) and
    $\ABS{\es_P}$ (right) as a function of $\delta$, fixing $k=4$ and
    $\qs_{rel}=0.2$. }
  \label{fig:fullydisp_deltadt001}
\end{figure}
Next, by letting varying the polynomial degree $k$ for $\delta=0.2$
and $\theta=\pi/4$ we analyze the dispersion error.
In Figure \ref{fig:fullydispN}, we observe the exponential convergence
that was noted in the semi discrete case, cf. Figure~\ref{fig:dispN},
$\qs_{rel}$ goes to zero.
Indeed, for sufficiently small values of $q$, the following asymptotic
relation holds.~\cite{Antonietti-Marcati-Mazzieri-Quarteroni:2016}
\begin{align*}
  \omega_{\hh}\approx\sqrt{\Lambda}+\calO(\Delta t^2).
\end{align*}
Thus $\omega_{\hh}$ decays as in the semi discrete case until the term
$\Delta t^2$ becomes dominant.
In Figure \ref{fig:fullydispdt001}, we compare the behavior of the
fully discrete scheme obtained on \RED{all the considered grids} for $\qs_{rel}=0.2$.
We notice \RED{that for all the grid configurations} the same level of accuracy is achieved
with a polynomial degree \RED{$k\geq 4$, with a better performance of triangular and quadrilateral grids}.\\
Regarding the dissipation error of the space-time discretized scheme,
the considerations made for the space discretized case remain valid.
\begin{figure}
  \hspace{-0.75cm}
  \begin{tabular}{c}
    \includegraphics[width=1.1\textwidth]{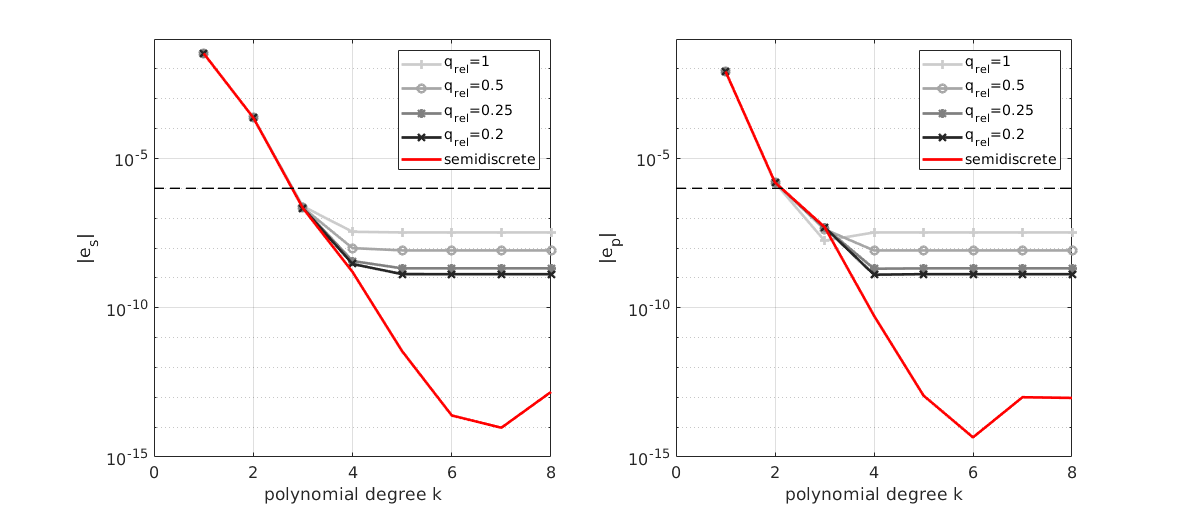} \\
    \textit{Triangular mesh} \\    
    \includegraphics[width=1.1\textwidth]{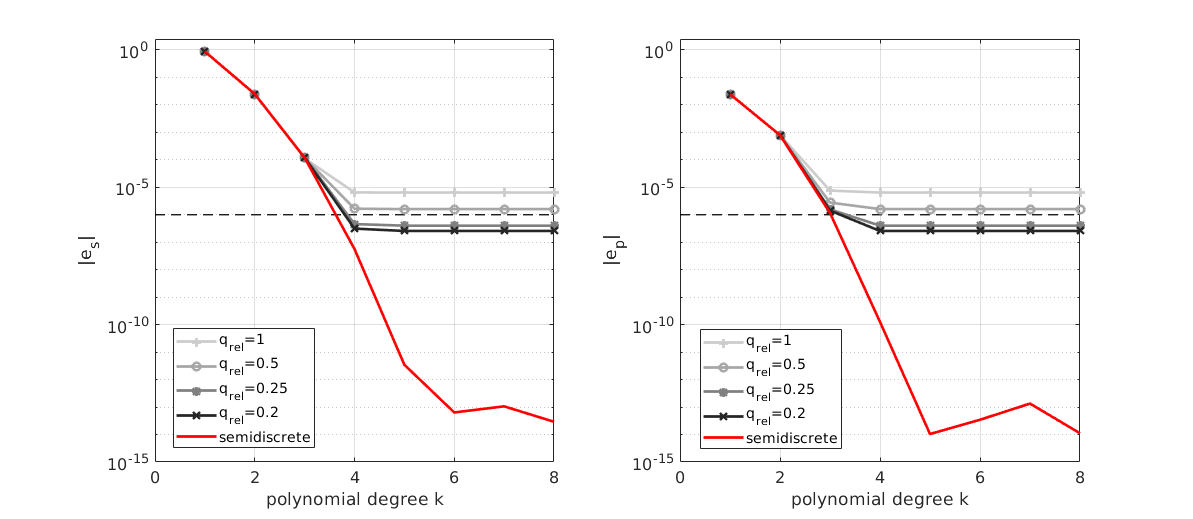} \\
    \textit{\RED{C3 mesh}}
  \end{tabular}
  \caption{Computed dispersion errors $\ABS{\es_S}$ (left) and $\ABS{\es_P}$
    (right) as a function of $k$ using \RED{triangular} (a) and
    \RED{composite C3} (b) grids, with a fixed sampling ratio $\delta=0.2$.
    The red lines refer to the semi discrete approximation, while the
    others to the fully discrete approximation with
    $\qs_{rel}=0.2,0.25,0.5,1$.}
  \label{fig:fullydispN}
\end{figure}
\begin{figure}
	\includegraphics[width=1.1\textwidth]{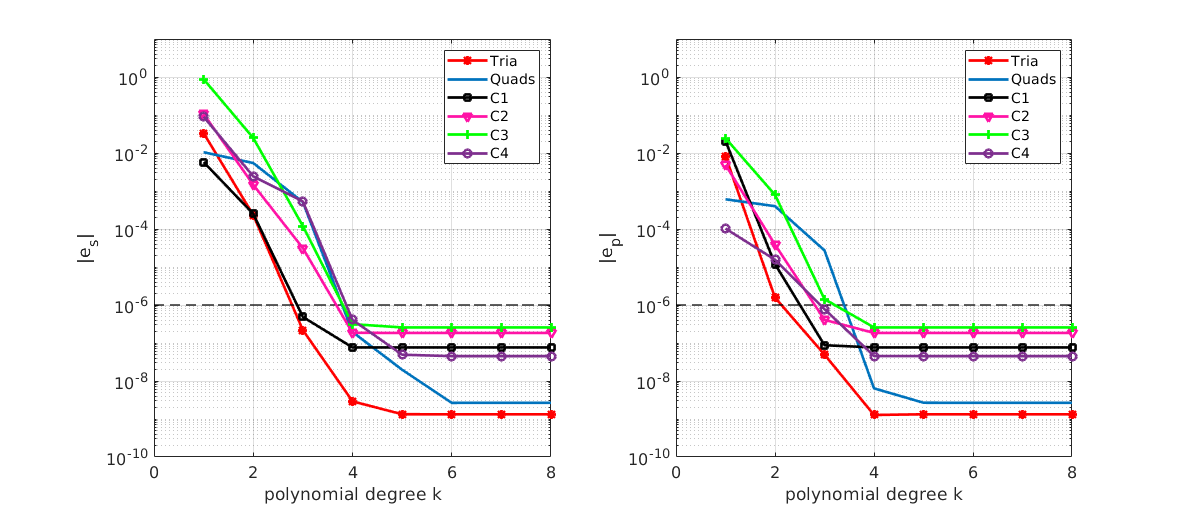}
  \caption{Computed dispersion errors $\ABS{\es_S}$ (left) and $\ABS{\es_P}$
    (right) as a function of $k$, for $\delta=0.2$ and
    $\qs_{rel}=0.2$.}
  \label{fig:fullydispdt001}
\end{figure}

\section{Conclusions}
\label{sec6:conclusion}

In this work, we extended the conforming virtual element method for
the numerical simulation of two dimensional time-dependent
elastodynamics problems.
The formulation of the VEM is investigated both theoretically and
numerically.
From the theoretical side, we proved the stability and the convergence
of the semi-discrete approximation in the energy norm and obtain
optimal rate of convergence.
We also derive $\LTWO$ error estimates with optimal convergence rate 
for the $\hh$- and $p$-refinement.
From the numerical side, we assessed the accuracy of the conforming
VEM on a set of different computational meshes, including non-convex
cells.
Optimal convergence rates in the energy norm and $\LTWO$ norm in the
$\hh$-refinement are numerically validated, and exponential
convergence is experimentally observed in both norms in the
$p$-refinement setting.
Moreover, a thorough comparison with standard VEM schemes on
simplicial/quadrilateral grids is presented in term of a
dispersion-dissipation and stability analysis, showing that polygonal
meshes behave as classical simplicial/quadrilateral grids in terms of
dispersion-dissipation and stability properties.


\section*{Acknowledgments}

\RED{We are grateful to the anonymous referees for their valuable and
  constructive comments, which have contributed to the improvement of
  the paper. We also thank Dr. Lorenzo Mascotto (University of Vienna)
  for a useful discussion on the content of Remark~\ref{remark5}.}
The first and last authors acknowledge the financial support of PRIN
research grant number 201744KLJL ``\emph{Virtual Element Methods:
Analysis and Applications}'' funded by MIUR.
The first, the third and the last authors also acknowledge the
financial support of INdAM-GNCS.
The work of the second author was supported by the Laboratory Directed
Research and Development (LDRD) Program of Los Alamos National
Laboratory under project number 20180428ER.
The work of the fourth author was supported by the LDRD program of Los
Alamos National Laboratory under project number 20170033DR.
Los Alamos National Laboratory is operated by Triad National Security,
LLC, for the National Nuclear Security Administration of
U.S. Department of Energy (Contract No. 89233218CNA000001).

\end{document}